\documentclass[a4paper, 11pt, reqno]{amsart}
\usepackage{amsaddr}
\usepackage{graphicx}

\usepackage{enumitem}
\usepackage[pdfstartview={XYZ null null null}]{hyperref}

\usepackage{pstricks,tikz}
\usetikzlibrary{positioning,chains,fit,shapes,calc,fit}

\newtheorem{theorem}{Theorem}[section]
\newtheorem{lemma}[theorem]{Lemma}
\newtheorem{proposition}[theorem]{Proposition}
\newtheorem{corollary}[theorem]{Corollary}
\newtheorem{conjecture}[theorem]{Conjecture}

\newtheorem{fact}[theorem]{Fact}

\numberwithin{equation}{section}

\newcommand{\defn}{\emph}

\newcommand{\N}{\mathbb{N}}

\renewcommand{\Pr}{\mathbb{P}}
\newcommand{\E}{\mathbb{E}}

\renewcommand{\P}{\mathcal{P}}

\def\eps{{\varepsilon}}
\renewcommand{\epsilon}{\varepsilon}

\newcommand{\emptygraph}{\emptyset}

\newcommand{\COMMENT}[1]{}

\title{Edge-decompositions of graphs with high minimum degree}

\author{Ben Barber, 
        Daniela K\"uhn, 
        Allan Lo \and
        Deryk Osthus}
\address{School of Mathematics, University of Birmingham,\\Birmingham, B15 2TT, UK}
\email{\{b.a.barber, d.kuhn, s.a.lo, d.osthus\}@bham.ac.uk}
\thanks{The research leading to these results was partially supported by the  European Research Council under the European Union's Seventh Framework Programme (FP/2007--2013) / ERC Grant Agreement n. 258345 (B.~Barber, D.~K\"uhn and A.~Lo) and 306349 (D.~Osthus).}

\begin{document}
\date{\today}

\maketitle

\begin{abstract}
A fundamental theorem of Wilson states that, for every graph~$F$, every sufficiently large $F$-divisible clique has an $F$-decomposition.
Here a graph $G$ is $F$-divisible if $e(F)$ divides $e(G)$ and the greatest common divisor of the degrees of $F$ divides the greatest common divisor of the degrees of $G$,
and $G$ has an $F$-decomposition if the edges of $G$ can be covered by edge-disjoint copies of $F$.
We extend this result to graphs~$G$ which are allowed to be far from complete.
In particular, together with a result of Dross, our results imply that every sufficiently large $K_3$-divisible graph of minimum degree at least $9n/10+o(n)$ has a $K_3$-decomposition.
This significantly improves previous results towards the long-standing conjecture of Nash-Williams that every sufficiently large $K_3$-divisible graph with minimum degree
at least $3n/4$ has a $K_3$-decomposition.
We also obtain the asymptotically correct minimum degree thresholds of $2n/3 +o(n)$ for the existence of a $C_4$-decomposition,
and of $n/2+o(n)$ for the existence of a $C_{2\ell}$-decomposition, where $\ell\ge 3$.
Our main contribution is a general `iterative absorption' method which turns an approximate or fractional decomposition into an exact one. 
In particular, our results imply that in order to prove an asymptotic version of Nash-Williams' conjecture, it suffices to show
that every $K_3$-divisible graph with minimum degree at least $3n/4+o(n)$ has an approximate $K_3$-decomposition, 

\end{abstract}

\section{Introduction}

Given a graph $F$, a graph $G$ has an \defn{$F$-decomposition} (is \defn{$F$-decomposable}), if the edges of $G$ can be covered by edge-disjoint copies of~$F$.
In this paper, we always consider decomposing a large graph $G$ into edge-disjoint copies of some small fixed graph~$F$.
The first such result was given by Kirkman~\cite{Kirkman} in~1847, who proved that the complete graph $K_n$ has a $K_3$-decomposition if and only if $n \equiv 1, 3 \mod{6}$.
To see that $n \equiv 1, 3 \mod{6}$ is a necessary condition, note that if $G$ has a $K_3$-decomposition, then the degree of each vertex of $G$ is even and $e(G)$ is divisible by~$3$.

There are similar necessary conditions for the existence of an $F$-decomposition.
For a graph $G$, let $\gcd (G)$ be the largest integer dividing the degree of every vertex of~$G$.
Given a graph $F$, we say that $G$ is \defn{$F$-divisible} if $e(G)$ is divisible by $e(F)$ and $\gcd(G)$ is divisible by $\gcd(F)$.
Being $F$-divisible is a necessary condition for being $F$-decomposable.
However, it is not sufficient: for example, $C_6$ does not have a $K_3$-decomposition.
In this terminology, Kirkman proved that every $K_3$-divisible clique has a $K_3$-decomposition. 
The analogue of this for general graphs $F$ instead of $K_3$ was an open problem
for a century until it was solved by Wilson~\cite{wilson1,wilson2,wilson3,Wilson} in 1975. 
Wilson proved that, for every graph~$F$, there exists an integer $n_0 =n_0(F)$ such that every $F$-divisible $K_n$ with $n \ge n_0$ has an $F$-decomposition.%
\COMMENT{Note that $K_{16}$ does not have a $K_{6}$-decomposition. }

\subsection{Decompositions of non-complete graphs}
In contrast, it is well known that the problem of deciding whether a general graph $G$ has an $F$-decomposition is NP-complete for every graph $F$ that contains a connected component with at least three edges~\cite{NPcomplete}. 
So a major question has been to determine the smallest minimum degree that guarantees an $F$-decomposition in any sufficiently large $F$-divisible graph~$G$.
Gustavsson~\cite{Gustavsson} showed that, for every fixed graph $F$, there exists $\epsilon = \epsilon(F) >0$ and $n_0= n_0(F)$ such that every $F$-divisible graph $G$ on $n \ge n_0$ vertices with minimum degree $\delta(G) \ge (1- \epsilon) n $ has an $F$-decomposition.
(This proof has not been without criticism.)
In a recent breakthrough, Keevash~\cite{Keevash} proved a hypergraph generalisation of Gustavsson's theorem.
His result actually states that every sufficiently large dense quasirandom hypergraph $G$ has a decomposition into cliques (subject to the necessary divisibility conditions).
The special case when $G$ is a complete hypergraph settles a question regarding the existence of designs going back to the 19th century.
Yuster~\cite{YusterBip} determined the asymptotic minimum degree threshold which guarantees an $F$-decomposition in the case when $F$ is
a bipartite graph with $\delta(F) =1$ (which includes trees). More recently, he~\cite{Yusternew} studied the problem of finding
many edge-disjoint copies of a given graph~$F$.
For a survey regarding $F$-decomposition of hypergraphs, directed graphs and oriented graphs, we recommend~\cite{Survey}.

In this paper, we substantially improve existing decomposition results when $F$ is an arbitrary graph. 
For $F = K_3$, Nash-Williams~\cite{NashWilliams} conjectured that every sufficiently large
$K_3$-divisible graph $G$ on $n$ vertices with $\delta(G) \ge 3 n /4$ has a $K_3$-decomposition.
This conjecture is still wide open.
For a general~$K_{r}$, the following (folklore) conjecture is a natural extension of Nash-Williams' conjecture. We describe the corresponding extremal construction in Proposition~\ref{prop:extremal}.

\begin{conjecture} \label{conjecture}
For every $r \in \mathbb{N}$ with $r \geq 2$, there exists an $n_0 = n_0(r)$ such that every $K_{r+1}$-divisible graph $G$ on $n \ge n_0$ vertices with $\delta(G) \ge (1-1/(r+2))n$ has a $K_{r+1}$-decomposition.
\end{conjecture}

Together with results by Dukes~\cite{Dukes,Dukes2} as well as Barber, K\"uhn, Lo, Montgomery and Osthus~\cite{BKLMO}, our main result (Theorem~\ref{thm:regular}) implies the following theorem, which gives the first significant step towards the conjectured bound and extends to decompositions into arbitrary graphs. 

\begin{theorem} \label{thm:general}
Let $F$ be a graph, let 
\[
C:= \min \{ 9\chi(F)^2(\chi(F)-1)^2/2 , 10^4 \chi (F)^{3/2}\} \; \; \mbox{ and let } \; \; t:=\max\{ C, 6e(F)\}.
\]
Then for each $\epsilon > 0$, there is an $n_0 = n_0(\epsilon,F)$ such that every $F$-divisible graph $G$ on $n \geq n_0$ vertices with $\delta(G) \geq ( 1 - 1/t + \eps)n$ has an $F$-decomposition.
\end{theorem}

Note that, for any $F$, we have $t \leq \min \{ 9 |F|^2 ( |F| - 1 )^2/2, 10^4 |F|^{2} \}$.
So $t=O(|F|^2)$.
The best previous bound in this direction is the one given by Gustavsson~\cite{Gustavsson}, who claimed that, if $F$ is complete, then a minimum degree bound of $(1- 10^{-37} |F|^{-94})n$ suffices.




We also obtain substantial further improvements for several families of graphs, in particular for cycles
(see Sections~\ref{subsec:improvements} and~\ref{sec:cycles}).

\subsection{Approximate $F$-decompositions}

The main contribution of this paper is actually a result that turns an `approximate' $F$-decomposition into an exact $F$-decomposition.
Let $G$ be a graph on $n$ vertices.
For a graph $F$ and $\eta \ge 0$, an \emph{$\eta$-approximate $F$-decomposition of $G$}
is a set of edge-disjoint copies of~$F$ covering all but at most $\eta n^2$ edges of~$G$. 
Note that a $0$-approximate $F$-decomposition is an $F$-decomposition.
For $ n \in \mathbb{N}$ and $\eta>0$, let $\delta^{\eta}_F (n)$ be the infimum over all~$\delta$ such that every graph~$G$ on $n$ vertices with $\delta(G) \geq \delta n$ has an $\eta$-approximate $F$-decomposition.%
	\COMMENT{Note that $\delta^{\eta}_F (n)$ exists for sufficiently large $n$ and it is less than $1$. 
Consider a $K_n$ and throw away small number of vertices until we have $K_{n'}$ that is $F$-divisible. Now apply Wilson to $K_{n'}$.}
We define $\delta^{0}_F (n)$ in a similar way, except that we only consider $F$-divisible graphs.
Let $\delta_F^{\eta} : = \limsup_{n \rightarrow \infty} \delta^{\eta}_F (n)$ be the \emph{$\eta$-approximate $F$-decomposition threshold}.
Clearly $\delta_F^{\eta'} \ge \delta_F^{\eta} $ for all $\eta' \le \eta$.
It turns out that there are $F$-divisible graphs with $\lim_{\eta  \rightarrow 0} \delta^{\eta}_F = \delta_F^{0}$, and graphs for which this equality does not hold
(see Section~\ref{sec:cycles} for a further discussion).

Our main result relates the `decomposition threshold' to the `approximate decomposition threshold' and an
additional minimum degree condition for $r$-regular graphs~$F$.
The dependence on $r$ is not far from best possible, since Proposition~\ref{prop:extremal} shows that the term $1/3r$ cannot be replaced by anything larger than $1/(r+2)$.

\begin{theorem} \label{thm:regular}
Let $F$ be an $r$-regular graph.
Then for each $\epsilon > 0$, there exists an $n_0 = n_0(\epsilon,F)$ and an $\eta  = \eta(\eps, F) $ such that every $F$-divisible graph $G$ on $n\geq n_0$ vertices with $\delta(G) \geq (\delta+\epsilon)n$, where $\delta : = \max \{ \delta^{\eta}_F ,  1-1/3r \}$, has an $F$-decomposition.
\end{theorem}

To derive Theorem~\ref{thm:general} from Theorem~\ref{thm:regular}, we will use a result of Haxell and R\"odl~\cite{HaxellRodl}
as well as a result of Yuster~\cite{YusterFracDecom}. Roughly speaking, the result in~\cite{HaxellRodl} implies that the minimum degree which guarantees a fractional $F$-decomposition in a graph also guarantees an $\eta$-approximate $F$-decomposition. 
This allows us to replace the $\eta$-approximate $F$-decomposition threshold $\delta^{\eta}_F$ in Theorem~\ref{thm:regular} by the `fractional $F$-decomposition threshold' (see Section~\ref{sec:fractional} for more details).
A result of Yuster~\cite{YusterChromatics} implies that we can consider the fractional $K_{\chi(F)}$-decomposition threshold instead of
the fractional $F$-decomposition threshold.
We will then use the results from~\cite{Dukes,Dukes2} as well as from~\cite{BKLMO}, which guarantee a fractional $K_r$-decomposition of any graph on $n$ vertices with minimum degree at least $(1-2/(9r^2(r-1)^2))n$ and minimum degree at least $(1-1/10^4 r^{3/2})n$, respectively.
Any improvement in the fractional $K_r$-decomposition threshold would immediately imply better bounds in Theorem~\ref{thm:general} (see Theorem~\ref{thm:general3}).

Our proof of Theorem~\ref{thm:regular} gives a polynomial time randomized algorithm which produces a decomposition with high probability (see Section~\ref{sec:main-theorem} for more details).
Our argument here and that in~\cite{BKLMO} is purely combinatorial.
In particular, the proofs of Theorems~\ref{thm:BKLMO} and~\ref{thm:general3} together  yield a combinatorial proof of
Wilson's theorem~\cite{wilson1,wilson2,wilson3,Wilson} that every large $F$-divisible clique has an $F$-decomposition.
(The original proof as well as that in Keevash~\cite{Keevash} relied on algebraic tools.)

\subsection{Further improvements: cycle decompositions} \label{subsec:improvements}

In Section~\ref{sec:main-theorem}, we state a version of Theorem~\ref{thm:regular} which is more technical but can be applied to give better bounds for many specific choices of~$F$ (Theorem~\ref{thm:regularstrong}). 
For example, in Section~\ref{sec:cycles}, we apply this to derive the following result on cycle decompositions.

\begin{theorem} 
\label{thm:cycles}
\noindent{\rm(i)}
Let $\ell\in\mathbb{N}$ with $\ell \geq 4$ be even, and let
\begin{align*}
\delta :=
\begin{cases}
  1/2  & \text{if } \ell \geq 6; \\
  2/3  & \text{if } \ell  =  4.
\end{cases}
\end{align*}
Then for each $\eps >0$, there exists an $n_0 = n_0(\eps,\ell)$ such that every $C_{\ell}$-divisible graph $G$
on $n\ge n_0$ vertices with $\delta(G) \ge (\delta + \eps) n$ has a $C_{\ell}$-decomposition.

\smallskip

\noindent{\rm(ii)}
Let $\ell\in\mathbb{N}$ with $\ell \geq 3$ be odd.
Then for each $\eps >0$, there exists an $n_0 = n_0(\eps,\ell)$ and an $\eta  = \eta(\eps,\ell) $ such that
every $C_{\ell}$-divisible graph $G$ on $n\ge n_0$ vertices with $\delta(G) \ge (\delta^\eta_{C_\ell} + \eps) n$
has a $C_{\ell}$-decomposition. 
Moreover, every $C_{\ell}$-divisible graph $G$ on $n\ge n_0$ vertices with $\delta(G) \ge (9/10+\eps) n$ has a $C_{\ell}$-decomposition.
\end{theorem}

Thus by Theorem~\ref{thm:cycles}(ii), it suffices to show that $\delta^{\eta}_{C_3} \le 3/4$ in order to
prove Conjecture~\ref{conjecture} for $r = 2$ asymptotically.
The value of the constant $\delta$ in Theorem~\ref{thm:cycles}(i) is the best possible
(see Propositions~\ref{cycle-lower-bounds} and~\ref{c4-lower-bounds}). The special case of Theorem~\ref{thm:cycles}(i)
when $\ell=4$ improves a result of Bryant and Cavenagh~\cite{C4decomp},
who showed that every $C_4$-divisible graph $G$ on $n$ vertices
with minimum degree at least $(31/32+o(1))n$ has a $C_4$-decomposition. 

It would be interesting to find other examples of graphs~$F$ for which Theorem~\ref{thm:regularstrong} can be used to obtain optimal or near optimal results. 

\subsection{Extremal graphs for Conjecture~\ref{conjecture}}

The following example from~\cite{Schlund} shows that the minimum degree condition in Conjecture~\ref{conjecture} is optimal.
We include a proof for completeness.

\begin{proposition} \label{prop:extremal}
For every $r \in \mathbb{N}$ with $r \geq 2$, there exist infinitely many $n$ such that there exists a $K_{r+1}$-divisible graph $G$ on $n$ vertices with $\delta(G)= \lceil (1-1/(r+2))n \rceil -1$ without a $K_{r+1}$-decomposition.
\end{proposition}

\begin{proof}
Let $\ell,s \in \mathbb{N}$. 
We first consider the case when $r := 2 \ell$.
Let $h: = (sr+1)(r+1)$.
Let $K_{2 \ell +2} - M$ be the subgraph of $K_{2 \ell +2}$ left after removing a perfect matching. 
Let $G^{2\ell}_h$ be the graph constructed by blowing up each vertex of $K_{r+2}-M$ to a copy of~$K_h$.
Thus $G^{2\ell}_h$ has $n : = (r+2) h$ vertices and is $d$-regular with
$d: = (h-1) + r h = (r+1) n/(r+2)-1$.
Since $r$ divides $d$ and $r+1$ divides $h$, $\binom{r+1}2$ divides~$e(G^{2\ell}_h)$, implying that $G^{2\ell}_h$ is $K_{r+1}$-divisible.
Call an edge \defn{internal} in $G^{2\ell}_h$ if it lies entirely within one of the copies of~$K_h$.
The number of internal edges is $I^{2 \ell }_h : = (r+2) \binom{h}2$.
Since $G^{2\ell}_h$ is a blow-up of $K_{r+2}-M$, each copy of $K_{r+1}$ in $G^{2\ell}_h$ must contain at least $r/2$ internal edges.%
\COMMENT{$2\ell+1$ vertices over $\ell+1$ matching edges means at least $\ell$ pairs share a copy of $K_h$.}
Thus the number of edge-disjoint copies of $K_{r+1}$ in $G^{2\ell}_h$ is at most $I^{2\ell}_h / (r/2) < e(G^{2\ell}_h)/ \binom{r+1}2$.%
	\COMMENT{\begin{align*}
	\frac{I^{2\ell}_h}{r/2} & =\frac{(r+2) \binom{h}2}{ r/2 } = \frac{(r+2)h}{r} (h-1)\\
	\frac{e( G^{2\ell}_h )}{\binom{r+1}2} & = \frac {(r+2) h  ((r+1)h - 1)}{r(r+1)} = \frac{(r+2)h}{r} \left( h- \frac{1}{r+1} \right)/
\end{align*}}
Therefore $G^{2\ell}_h$ does not have a $K_{r+1}$-decomposition.

For $r: = 2\ell +1$, let $h: = ( s(r+1) +1 ) r$.
Let $G^{2\ell +1}_h$ be the graph obtained from $G^{2\ell}_h$ by adding a set~$W$ of $h +1$ new vertices and joining each new vertex to each vertex in~$V(G^{2 \ell}_h )$. 
Note that $G^{2\ell+1}_h$ has $n : = (r+2) h+1$ vertices and is $d$-regular with
$d: =  (r+1) h  = (r+1) (n-1)/(r+2) = \lceil (1-1/(r+2))n \rceil -1$.\COMMENT{$x \in \mathbb{N} \Rightarrow \lceil x+ \eps \rceil = x+1$ for $0< \eps <1$.}
Since $r(r+1)$ divides $d$, $\binom{r+1}2$ divides~$e(G^{2\ell+1}_h)$, implying that $G^{2\ell+1}_h$ is $K_{r+1}$-divisible.
Let the \defn{internal} edges of $G^{2\ell+1}_h$ be the internal edges of $G^{2\ell}_h$.
Thus the number of internal edges is $I^{2 \ell +1}_h : = (r+1) \binom{h}2$.
Note that each copy of $K_{r+1}$ in $G^{2\ell+1}_h$ must contain at least $(r-1)/2$ internal edges.
Moreover, if $K_{r+1}$ contains precisely $(r-1)/2$ internal edges, then $K_{r+1}$ must contain a vertex in~$W$.
Hence there are at most $d |W|/r = (r+1)(h+1) ( s(r+1) +1 )$ edge-disjoint copies of $K_{r+1}$ in $G^{2\ell+1}_h$ that contain precisely $(r-1)/2$ internal edges.
Therefore, the number of edge-disjoint copies of $K_{r+1}$ in $G^{2 \ell +1}_h$ is at most\COMMENT{
Second line of the calculation\begin{align*}
	h(h-1) +2(h+1) (s(r+1)+1 ) 
	& = (s(r+1)+1)  ( r(h-1) + 2(h+1))
	 = (s(r+1)+1)  ( (r+2)h -(r-2) )
\end{align*}
}
\COMMENT{$e(G^{2 \ell +1}_h) = dn/2 = \frac12 (r+1)h \left( (r+2)h+1 \right) = \binom{r+1}2( s(r+1)+1 )\left( (r+2)h+1 \right)$.}
\begin{align*}
	& 	(r+1)(h+1) ( s(r+1) +1 ) + \frac{ I^{2 \ell +1}_h -  (r+1)(h+1) ( s(r+1) +1 )\frac{r-1}{2} }{ (r+1)/2 }\\
	& =  h(h-1) +2(h+1) (s(r+1)+1 ) 
	 = (s(r+1)+1)  ( (r+2)h -(r-2) )\\
	 & < (s(r+1)+1)  ( (r+2)h + 1 ) = \frac{ e(G^{2 \ell +1}_h) }{\binom{r+1}{2}}.
\end{align*}
Therefore $G^{2\ell+1}_s$ does not have a $K_{r+1}$-decomposition.
\end{proof}


\section{Sketches of proofs} \label{sec:sketch}

\subsection{Proof of Theorem~\ref{thm:general} using Theorem~\ref{thm:regular}.}
The idea of this proof is quite natural.
Given graphs $F$ and $G$ as in Theorem~\ref{thm:general}, we find an $F$-decomposable regular graph~$R$ such that both the degree $r$ of $R$
and the $\eta$-approximate decomposition threshold $\delta_R^{\eta}$ are not too large.
By removing a small number of copies of~$F$ from $G$, we may assume that $G$ is also $R$-divisible. 
By Theorem~\ref{thm:regular}, $G$ has an $R$-decomposition and so an $F$-decomposition, provided $\delta(G) \geq \max \{\delta^\eta_R, 1-1/3r\}$.
This reduction is carried out in Section~\ref{sec:general-case}.

To obtain the explicit bound on $\delta(G)$, we apply results of Dukes~\cite{Dukes,Dukes2} as well as
Barber, K\"uhn, Lo, Montgomery and Osthus~\cite{BKLMO} on fractional decompositions in graphs of large minimum degree together with a result of Haxell and R\"odl~\cite{HaxellRodl} relating fractional decompositions to approximate decompositions.  
We collect these tools in Section~\ref{sec:fractional}.

\subsection{Proof of  Theorem~\ref{thm:regular}.}

The proof of Theorem~\ref{thm:regular} develops an `iterative absorbing' approach.
The original absorbing method was first used for finding $K_3$-factors (that is, a spanning union of vertex-disjoint copies of $K_3$) by
Krivelevich~\cite{Krivelevich} and for finding Hamilton cycles in hypergraphs by R\"odl, Ruci\'{n}ski and Szemer\'{e}di~\cite{RRSz}.
An absorbing approach for finding decompositions was first used by K\"uhn and Osthus~\cite{KellyConj}.

More precisely, the basic idea behind the proof of Theorem~\ref{thm:regular} can be described as follows. 
Let $G$ be a graph as in Theorem~\ref{thm:regular}.
Suppose that we can find an $F$-divisible subgraph $A^*$ of $G$ with small maximum degree which is an $F$-absorber in the following sense: $A^* \cup H^*$ has an $F$-decomposition whenever $H^*$ is a sparse $F$-divisible graph on $V(G)$ which is edge-disjoint from~$A^*$.
Let $G'$ be the subgraph of $G$ remaining after removing the edges of~$A^*$.
Since $A^*$ has small maximum degree, $\delta(G') \ge (\delta^{\eta}_F+ \eps/2) n$.
By the definition of $\delta^{\eta}_F$, $G'$ has an $\eta$-approximate $F$-decomposition~$\mathcal{F}$.
Let $H^*$ be the leftover (that is, the subgraph of $G'$ remaining after removing all edges in $\mathcal{F}$).
Note that $H^*$ is also $F$-divisible.
Since $A^* \cup H^*$ has an $F$-decomposition, so does~$G$. 

Unfortunately, this naive approach fails for the following reason: we have no control on the leftover~$H^*$.
A first attempt at obtaining%
\COMMENT{20/7} 
$A^*$ would be to construct it as the edge-disjoint union of graphs $A$ such that each such $A$ has an $F$-decomposition and, for each possible leftover graph $H^*$, there is a distinct $A$ so that $A \cup H^*$ has an $F$-decomposition.
However, a typical leftover graph $H^*$ has $\eta n^2$ edges, so the number of possibilities for $H^*$ is exponential in $n$.
So we have no hope of finding all the required graphs $A$ in $G$ (and thus to construct $A^*$).
To overcome this problem, we reduce the number of possible configurations of $H^*$ (in turn reducing the number of graphs~$A$ required) as follows.
Roughly speaking, we iteratively find approximate decompositions of the leftover so that eventually our final leftover $H^*$ only has $O(n)$ edges whose location is very constrained---so one can view this step as finding a `near optimal' $F$-decomposition.

To illustrate this, suppose that $m \in \mathbb{N}$ is bounded and $n$ is divisible by $m$. 
Let $\P : = \{V_1, \dots, V_{q} \}$ be a partition of $V(G)$ into parts of size $m$ (so $q = n/m$).
We further suppose that $H^*$ is a vertex-disjoint union of $F$-divisible graphs $H^*_1, \dots, H^*_{q}$ such that $V(H^*_i) \subseteq V_i$ for each~$i$.
Hence to construct $A^*$, we only need to find one $A$ for each possible $H^*_i$.
(To be more precise, $A^*$ will now consist of edge-disjoint graphs $A$ such that each $A$ has an $F$-decomposition and, for each possible $H^*_i$, there is a distinct $A$ so that $A \cup H^*_i$ has an $F$-decomposition.)
For a fixed~$i$, there are at most $2^{\binom{|V_i|}2} = 2^{\binom{m}2}$ possible configurations of~$H^*_i$. 
Since $m$ is bounded, in order to construct $A^*$ we would only need to find $q 2^{\binom{m}2} = 2^{\binom{m}2} n /m$ different $A$.
Essentially, this is what Lemma~\ref{lma:Krabsorber} achieves.

We now describe in more detail the iterative approach which achieves the above setting.
Recall that $G'$ is the subgraph of $G$ remaining after removing all the edges of~$A^*$.
Since $A^*$ has small maximum degree, $G'$ has roughly the same properties as~$G$.
Our new objective is to find edge-disjoint copies of $F$ covering all edges of $G'$ that do not lie entirely within $V_i$ for some~$i$.
Since each $V_i$ has bounded size, these edge-disjoint copies of~$F$ will cover all but at most a linear number of edges of~$G'$.
As indicated above, we use an iterative approach to achieve this.
We proceed as follows. 
Let $k \in \mathbb{N}$. 
Let $\P_1$ be an equipartition of $V(G)$ into $k$ parts, and let $G_1$ be the $k$-partite subgraph of $G'$ induced by~$\P_1$ (here $k$ is large but bounded).
Suppose that we can cover the edges of $G_1$ by copies of~$F$ which use only a small proportion of the edges not in~$G_1$.
Call the leftover graph~$H_1$.
Let $\P_2$ be an equipartition of $V(G)$ into $k^2$ parts obtained by dividing each $V \in \P_1$ into $k$ parts. 
Let $G_2$ be the $k^2$-partite subgraph of~$H_1$ induced by~$\P_2$.
Each component of~$G_2$ will form a $k$-partite graph lying within some $V \in \P_1$.
So by applying the same argument to each component of $G_2$ in turn and iterating $\log_k (n/m)$ times we obtain an equipartition $\P = \P_\ell$ of $V(G)$ with $|V| = m $ for each $V \in\P$ such that all edges of $G'$ that do not lie entirely within some $V \in \P$ can be covered by edge-disjoint copies of~$F$.

In Section~\ref{sec:embedding} we prove an embedding lemma that allows us to find certain subgraphs in a dense graph. 
We will use this throughout the paper.
The formal definition of $\P_1, \P_2, \dots, \P_{\ell}$ is given in Section~\ref{sec:random}.
We construct the absorber graph $A^*$ in Section~\ref{sec:absorbers}.
The `near optimal' decomposition result is proved in Sections~\ref{sec:parity-graphs} and~\ref{sec:partial-decomposition}.
Finally, we prove Theorem~\ref{thm:regular} in Section~\ref{sec:main-theorem}.


\section{Notation} \label{sec:notation}

Let $G$ be a graph, and let $\P = \{V_1, \ldots, V_k\}$ be a partition of $V(G)$.
We write $G[V_1]$ for the subgraph of $G$ induced by the vertex set $V_1$, $G[V_1,V_2]$ for the bipartite subgraph induced by the vertex classes $V_1$ and $V_2$, and $G[\P] := G[V_1, \ldots, V_k]$ for the $k$-partite subgraph of $G$ induced by the $k$-partition $\P$.
Write $V_{<i}$ for $V_1 \cup \dots \cup V_{i-1}$ and $V_{\le i}$ for $V_1 \cup \dots \cup V_{i}$.
We say that $\P$ is \defn{equitable} (or \defn{$k$-equitable}) if $\big| |V_i| - |V_j| \big| \le 1 $ for all $1 \le  i, j \le k$.
For $V \subseteq V(G)$, $\P[V]$ denotes the restriction of $\P$ to~$V$. 
Note that a $k$-equitable refinement of a $k$-equitable partition~$\P$ (obtained by taking a $k$-equitable partition of each $V \in \P$) is a $k^2$-equitable partition of $V(G)$.

Given a graph $G$ and disjoint $U, V \subseteq V(G)$, 
let $e_G(U) : = e(G[U])$ and $e_G(U,V): = e(G[U,V])$.
For sets $S, V \subseteq V(G)$, we write $N_G(S,V) := \{v \in V : xv \in E(G) \text{ for all } x \in S\}$, and $d_G(S,V) := |N_G(S,V)|$.
If $S = \{v\}$ is a singleton, we instead write $N_G(v,V)$ and $d_G(v,V)$.
We sometimes omit the subscript $G$ if it is clear from the context.

For graphs $G$ and $H$, we write $G-H$ for the graph with vertex set $V(G)$ and edge set $E(G) \setminus E(H)$, and $G \setminus H$ for the subgraph of $G$ induced by the vertex set $V(G) \setminus V(H)$.
For a set of edges $E$, we write $V(E)$ for the set of all endvertices of edges in~$E$.
We write $G \cup E$ for the graph with vertex set $V(G) \cup V(E)$ and edge set $E(G) \cup E$.

For $r \in \mathbb{N}$, a graph $G$ is \emph{$r$-divisible} if $r$ divides the degree $d(v)$ of $v$ for all $v \in V(G)$. 

For an integer~$p$ and a graph~$F$, we write $pF$ for the graph consisting of $p$ vertex-disjoint copies of~$F$.
If $G$ is a graph and $pF$ is a spanning subgraph of~$G$, then $pF$ is an \emph{$F$-factor} in~$G$.

The constants in the hierarchies used to state our results are chosen from right to left.
For example, if we claim that a result holds whenever $0<1/n\ll a\ll b\ll c\le 1$ (where $n$ is the order of the graph), then there is a non-decreasing function $f:(0,1]\to (0,1]$ such that the result holds
for all $0<a,b,c\le 1$ and all $n\in \mathbb{N}$ with $b\le f(c)$, $a\le f(b)$ and $1/n\le f(a)$. 
Hierarchies with more constants are defined in a similar way.
We write $a = b \pm c$ to mean $a \in [ b - c , b + c ]$.


\section{Fractional and approximate $F$-decompositions}
\label{sec:fractional}

Let $F$ and $G$ be graphs.
Define $p_F(G)$ to be the maximum number of edges in~$G$ that can be covered by edge-disjoint copies of~$F$.
So if $G$ has an $\eta$-approximate $F$-decomposition, then $e(G) - p_F(G) \le \eta n^2$ (where $G$ has $n$ vertices).

\begin{theorem}[Yuster~\cite{YusterChromatics}] \label{thm:YusterChromatics}
Let $F$ be a graph with $\chi : = \chi (F)$.
For all $\eta >0$, there exists an $n_0 = n_0(\eta,F)$ such that every graph $G$ on $n \ge n_0$ vertices satisfies $p_F(G) \ge p_{K_{\chi}}(G) - \eta n^2$.
\end{theorem}


\begin{corollary} \label{cor:YusterChromatics}
Let $F$ be a graph with $\chi : = \chi (F)$.
Then $\delta^{\eta}_F \le \delta^{\eta/2}_{K_{\chi}}$ for all $\eta>0$.
\end{corollary}

\begin{proof}
Let $\eta >0$ and let $G$ be a sufficiently large graph on $n$ vertices with $\delta(G) > \delta^{\eta/2}_{K_{\chi}}(n)  n$.
By the definition of $\delta^{\eta/2}_{K_{\chi}}(n)$ and Theorem~\ref{thm:YusterChromatics}, 
\begin{align*}
e(G) \le  p_{K_{\chi}}(G) +  \eta n^2/2 \le p_F(G) + \eta n^2.
\end{align*}
Therefore $\delta^{\eta}_{F}(n) \le  \delta^{\eta/2}_{K_{\chi}}(n)$ for all sufficiently large~$n$, implying $\delta^{\eta}_F \le \delta^{\eta/2}_{K_{\chi}}$.
\end{proof}

Write $\nu_F(G) := p_F(G) / e(F)$ for the maximum number of edge-disjoint copies of $F$ in $G$.
If $G$ has an $F$-decomposition, then $\nu_F(G) = e(G)/e(F)$.
We now introduce a fractional version of $\nu_F(G)$.
Let $\binom{G}{F}$ denote the set of copies of $F$ in $G$.
A function~$\psi$ from $\binom{G}{F}$ to $[0,1]$ is a \defn{fractional $F$-packing} of $G$ if $\sum_{F' \in \binom{G}{F} : e \in F' } \psi( F' )  \le 1$ for each $e \in E(G)$. 
The \defn{weight} of $\psi$ is $| \psi | :  = \sum_{F' \in \binom{G}{F} } \psi(F')$.
Let $\nu_F^*(G)$ be the maximum value of $|\psi|$ over all fractional $F$-packings~$\psi$ of $G$.
Clearly, $\nu_F^*(G) \ge \nu_F(G)$.
If $\nu_F^*(G) = e(G) / e(F)$, then we say that $G$ has a \defn{fractional $F$-decomposition}.

In fact, $\nu_F(G)$ and $\nu_F^*(G)$ are closely related.
Haxell and R\"odl~\cite{HaxellRodl} proved that any fractional packing can be converted into a genuine integer packing that covers only slightly fewer edges.
(An alternative proof was given by Yuster~\cite{YusterFracDecom}.)

\begin{theorem}{\cite{HaxellRodl}}\label{thm:haxellrodl}
Let $F$ be a graph and let $\eta > 0$.
Then there is an $n_0 = n_0(F,\eta)$ such that for every graph $G$ on $n \ge n_0$ vertices, $\nu_F(G) \geq \nu^*_F(G) - \eta n^2$.
\end{theorem}

For a graph $F$ and $n \in \mathbb{N}$, let $\delta^*_F (n)$ be the infimum over all $\delta$ such that every graph $G$ on $n$ vertices with $\delta(G) \geq \delta n$ has a fractional $F$-decomposition.
Let $\delta_F^* : = \limsup_{n \rightarrow \infty} \delta^*_F(n)$ be the \defn{fractional $F$-decomposition threshold}.
The following corollary is an immediate consequence of Theorem~\ref{thm:haxellrodl} and the definitions of $\delta^{\eta}_F$ and $\delta^*_F$.

\begin{corollary} \label{cor:deltaeta}
For every graph $F$ and every $\eta >0$, we have $\delta^{\eta}_F  \le  \delta^*_F$. \qedhere
\end{corollary}

Together with Corollary~\ref{cor:YusterChromatics}, we get the following corollary. 

\begin{corollary} \label{cor:chromatic}
For every graph $F$ and every $\eta >0$, we have $\delta^{\eta}_F  \le  \delta^*_{K_{\chi(F)}}$. \qedhere
\end{corollary}

%

For $F = K_{r+1}$, Yuster~\cite{YusterFractKr} proved that $\delta^*_{K_{r+1}} \le 1 - 1/(9 (r+1)^{10})$.
The best known bound on $\delta^*_{K_{r+1}}$ is given in~\cite{Dukes,Dukes2} for small values of~$r$ and in~\cite{BKLMO} for large values of~$r$.

\begin{theorem}[Dukes~\cite{Dukes,Dukes2}] \label{thm:dukes}  
For $r \in \mathbb{N}$ with $r \ge 2$, $\delta^*_{K_{r+1}} \le 1 - 2/(9(r+1)^{2}r^{2})$.
\end{theorem}

\begin{theorem}[Barber, K\"uhn, Lo, Montgomery, Osthus~\cite{BKLMO}] \label{thm:BKLMO}  
For $r \in \mathbb{N}$ with $r \ge 2$, $\delta^*_{K_{r+1}} \le 1 - 1/(10^4(r+1)^{3/2})$.
\end{theorem}

For the case when $r =2$ (that is $\delta^*_{K_{3}}$), Garaschuk~\cite{Garaschuk} improved the bound to $\delta^*_{K_3} < 0.956$.
Recently this was further improved by Dross~\cite{Dross}.

\begin{theorem}[Dross~\cite{Dross}] \label{thm:Dross}
We have that $\delta^*_{K_3} \leq  9/10$.
\end{theorem}


\section{Finding subgraphs} \label{sec:embedding}

In this section we will prove a result guaranteeing that in our given graph $G$ we can always remove certain subgraphs that we need without significantly reducing the minimum degree of $G$.
(These subgraphs might for example be the absorbers and parity graphs defined in Sections~\ref{sec:absorbers} and~\ref{sec:parity-graphs}.)

Let $G$ and $H$ be graphs.  
Suppose that for each vertex of $H$ we specify a set of vertices of $G$.  
We will seek a copy of $H$ in $G$ that is compatible with this specification.
More formally, let $\P$ be a partition of $V(G)$.
We say that a graph $H$ is a \defn{$\P$-labelled graph} if
\begin{itemize}
	\item each vertex of $H$ is labelled either $V(G)$, $\{v\}$ for some $v \in V(G)$, or $V$ for some $V \in \P$;
	\item the vertices labelled by singletons have distinct labels and form an independent set in $H$.
\end{itemize}
We call the vertices labelled by singletons \defn{root vertices}; the other vertices are \defn{free vertices}.

An \defn{embedding of $H$ into $G$ compatible with its labelling} is an injective graph homomorphism $\phi : H \to G$ such that each vertex gets mapped to an element of its label. 

Given a graph $H$ and $U \subseteq V(H)$ with $e (H[U]) = 0$, we define the \emph{degeneracy of $H$ rooted at $U$} to be the least $d$ for which there is an ordering $v_1, \ldots, v_b$ of the vertices of $H$ such that
\begin{itemize}
	\item there is an $a$ such that $U  = \{v_1, \ldots, v_a \}$;
	\item for $a < j \leq b$, $v_j$ is adjacent to at most $d$ of the $v_i$ with $i < j$.
\end{itemize}
The order of $v_1, \ldots, v_a $ is not important as $U$ is an independent set of $H$.
Note that the requirement that the vertices in $U$ come first means that the degeneracy of $H$ rooted at~$U$ might be larger than the usual degeneracy of~$H$.
The \defn{degeneracy of a $\P$-labelled graph $H$} is the degeneracy of $H$ rooted at $U$, where $U$ is the set of root vertices of~$H$. 

We now prove a very general lemma guaranteeing that, provided the common neighbourhoods of sets of up to $d$ vertices are sufficiently large, we can embed any collection of $\P$-labelled graphs that does not use any root label too many times.

\begin{lemma} \label{ur-finding}
Let $n, k, d, b, s, m \in \N$ with $1/n \ll 1/k, 1/d, 1/b$.
Let $G$ be a graph on $n$ vertices and let $\P = \{V_1, \ldots, V_k\}$ be an equitable partition of $V(G)$ such that $d_G(S, V_i) \geq 2db(\sqrt m + s + 1)$
for each $1 \leq i \leq k$ and $S \subseteq V(G)$ with $|S| \leq d$.
Let $H_1, \ldots, H_m$ be $\P$-labelled graphs such that
\begin{enumerate}[label={\rm(\roman*)}]
	\item for each $1 \leq i \leq m$, $|H_i| \leq b$;
	\item the degeneracy of each $H_i$ is at most $d$;
	\item for each $v \in V(G)$, the number of indices $1 \leq i \leq m$ such that some vertex of $H_i$ is labelled $\{v\}$ is at most $s$.
\end{enumerate}
Then there exist edge disjoint embeddings $\phi(H_1), \ldots, \phi(H_m)$ of $H_1, \ldots, H_m$ compatible with their labellings such that the subgraph $H := \bigcup_{i=1}^m \phi(H_i)$ of $G$ satisfies $\Delta(H) \leq 2b(\sqrt m + s)$.
\end{lemma}

\begin{proof}
For each $v \in V(G)$ and each $0 \leq j \leq m$, let $s(v,j)$ be the number of indices $1 \leq i \leq j$ such that some vertex of $H_i$ is labelled $\{v\}$; so $s(v,j) \leq s$.

Suppose that, for some $1 \leq j \leq m$, we have already embedded $H_1, \dots, H_{j-1}$ such that 
\begin{align}
 d_{G_{j-1}}(v) \leq b(\sqrt m + s(v,j-1) + 1), \label{already-used}
\end{align} 
where $G_{j-1}$ consists of the subgraph of $G$ used to embed $H_1,\dots, H_{j-1}$.
Our next aim is to embed $H_j$ into $G - G_{j-1}$ such that \eqref{already-used} holds with $j$ replaced by $j+1$.
By~(ii), we can order the vertices of $H_j$ such that root vertices of $H_j$ precede free vertices of $H_j$ and each free vertex is preceded by at most $d$ of its neighbours.
Suppose that we have already embedded some vertices of $H_j$ one by one in this order and that the next vertex of $H_j$ to be embedded is $x$.

Let $B : = \{v \in V(G): d_{G_{j-1}}(v) \geq b \sqrt m \}$ be the set of vertices that are in danger of being used too many times.  
Since $e(H_i) \le \binom{|H_i|}{2} \le \binom{b}2$ for each~$i$, we have $2e(G_{j-1}) \le b^2 m$.
So we have that
\begin{align} \label{B-small}
|B| \leq b^2m/b \sqrt m = b \sqrt m.
\end{align}

If $x$ is a root vertex, then we can embed $x$ at its assigned position because we have yet to embed any of its neighbours.

If $x$ is a free vertex, then at most $d$ of its neighbours have already been embedded.
Let $U$ be the set of images of these neighbours, and let $V$ be the label of $x$.
Now
\begin{align*}
\arraycolsep=1pt\def\arraystretch{1.4}
\begin{array}{rcl}
d_{G-G_{j-1}}( U ,  V)
& \geq & d_G(U,V) - \sum_{ u \in U } d_{G_{j-1}}( u ,  V) \\ 
& \overset{\eqref{already-used}}{\geq} & 2db(\sqrt m + s + 1) - db(\sqrt m + s + 1)
 > b \sqrt m + b \overset{\eqref{B-small}} \geq |B| + |H_j|.
\end{array}
\end{align*}
So we can choose a suitable image for $x$ outside of $B$.

Suppose that we have completed the embedding of $H_j$.
We will now check that \eqref{already-used} holds with $j$ replaced by $j+1$.
Clearly \eqref{already-used} holds for every $v \in V(G) \setminus B$.
But if $v \in B$, then \eqref{already-used} holds for $v$ as well because free vertices of $H_j$ were embedded outside of $B$ and, if $v$ is the image of a root vertex of $H_j$, then $s(v,j) = s(v,j-1)+1$.

Finally observe that, by \eqref{already-used}, $\Delta(H) = \Delta(G_m) \leq b(\sqrt m + s(v,m) + 1) \leq 2b(\sqrt m + s)$.
\end{proof}

The following lemma follows immediately from Lemma~\ref{ur-finding}, but has conditions that will be slightly more convenient to check.

\begin{lemma} \label{lma:finding}
Let $n, k, d, b \in \N$ and let $\eta, \epsilon > 0$ with $1/ n \ll \eta \ll \eps, 1/d, 1/b , 1/k $.
Let $G$ be a graph on $n$ vertices, and let $\P = \{V_1, \ldots, V_k\}$ be an equitable partition of $V(G)$ such that, for each $1 \leq i \leq k$ and each $S \subseteq V(G)$ with $|S| \leq d$, $d_G(S,V_i) \geq \epsilon |V_i|$.
Let $m \leq \eta n^2$ and let $H_1, \ldots, H_m$ be $\P$-labelled graphs such that
\begin{enumerate}[label={\rm(\roman*)}]
	\item for each $1 \leq i \leq m$, $|H_i| \leq b$;
	\item the degeneracy of each $H_i$ is at most $d$;
	\item for each $v\in V(G)$, the number of indices $1 \leq i \leq m$ such that some vertex of $H_i$ is labelled $\{v\}$ is at most~$\eta n$.
\end{enumerate} 
Then there exist edge-disjoint embeddings $\phi(H_1), \ldots, \phi(H_m)$ of $H_1, \ldots, H_m$ compatible with their labellings such that the subgraph $H := \bigcup_{i=1}^m \phi(H_i)$ of $G$ satisfies $\Delta(H) \leq \epsilon n$.
\end{lemma}


\section{Deriving Theorem~\ref{thm:general} from Theorem~\ref{thm:regular}} \label{sec:general-case}

In this section we extend Theorem~\ref{thm:regular}, which applies to regular graphs $F$, to Theorem~\ref{thm:general}, which does not require the assumption of regularity.
Our approach is to combine multiple copies of $F$ into a regular graph $R$ and then apply Theorem~\ref{thm:regular} to~$R$.
We cannot do this immediately, as an $F$-divisible graph $G$ need not in general also be $R$-divisible.
We can however ensure that the extra divisibility conditions hold by removing a small number of copies of $F$ from $G$.

We first prove that we can combine multiple copies of $F$ to obtain a regular graph whose degree and chromatic number are not too large.

\begin{lemma} \label{regularise}
Let $F$ be a graph.
There is an $F$-decomposable $r$-regular graph $R$ with $r = 2e(F)$ and $\chi(R) = \chi(F)$.
\end{lemma}

We now give the main idea of the proof.
Throughout the proof of the lemma, we write $[a] := \{0, 1, \ldots, a-1\}$, thought of as the set of residue classes modulo $a$.
 Let $k := \chi(F)$ and fix a $k$-colouring of $F$.
Let $t$ be the size of the largest colour class.
By adding isolated vertices to $F$ if necessary, we may assume that $V(F) = [k] \times [t]$ with the $k$ colour classes of $F$ being $\{i\} \times [t]$ for each $i \in [k]$ (so there is no edge between $(x_1, y_1)$ and $(x_2, y_2)$ if $x_1 = x_2$).

For any injective function $\theta$ defined on the vertex set of a graph $H$, let $\theta(H)$ be the graph on the vertex set $\theta (V(H))$ for which $\theta : V(H) \to \theta ( V(H))$ is an isomorphism.
Thus for $w \in [k] \times [t]$, $F + w$ is the graph obtained from $F$ by translating each vertex by $w$ inside $[k] \times [t]$.
(To be precise, $F+w := \theta_w(F)$, where $\theta_w : (a,b) \mapsto (a+i,b+j)$ with $w = (i,j)$.)
Note that $F+w$ is still $k$-partite with the $k$ colour classes being $\{i\} \times [t]$ for each $i \in [k]$.
Since each vertex of $F$ is assigned to each possible position in $[k] \times [t] = V(F)$ exactly once under these translations, for each $x \in V(F)$ we have that $\sum_{w \in [k] \times [t]} d_{F+w}(x) = 2e(F)$.
We would like to take $R$ to be $\bigcup_{w \in [k] \times [t]} F + w$.
However, that might produce multiple edges, so we will actually take more copies of $F$ spread across a larger vertex set.
In this way, we can achieve a similar result without producing multiple edges.

More precisely, the vertex set of $R$ will be $V := [k] \times [t] \times [k^2t]$.  
(The length of the third dimension is chosen so that the multiplication maps $x \mapsto ax$ from $[kt]$ to $[k^2t]$ are injective for $a \in [k]\setminus \{0\}$.)
We will embed copies of $F$ in $k^2t$ sets of disjoint $[k] \times [t]$ `slices' of $V$.
Intuitively, these sets of slices will be taken at different angles to ensure that we do not create multiple edges.

\begin{proof}[Proof of Lemma~\ref{regularise}]
Let $k := \chi(F)$ and fix a $k$-colouring of $F$.
Let $t$ be the size of the largest colour class.
By adding isolated vertices to $F$ if necessary, we may assume that $V(F) = [k] \times [t]$ with the $k$ colour classes of $F$ being $\{i\} \times [t]$ for each $i \in [k]$.
Let $V:= [k] \times [t] \times [k^2t]$.

For $\ell \in [k^2t]$ and $s \in [kt]$, let $\phi_{\ell,s} : [k] \times [t] \to V$ be defined by $\phi_{\ell,s}(x, y) := (x, y, \ell+xs)$.
Define the \emph{slice} $\Phi_{\ell,s}$ to be $\{ \phi_{\ell,s}(x,y) : (x,y) \in [k] \times [t] \}$.
Note that, for fixed $s$, the set $\{ \Phi_{0,s}, \ldots, \Phi_{k^2t-1,s} \}$ of slices forms a partition of~$V$.

For $w \in [k] \times [t]$, observe that $\phi_{\ell,s}(F+w)$ is $k$-partite with $k$-colourings induced by projections onto the first coordinate of $V$.
Indeed, recall that $F+w$ has colour classes $\{0\} \times [t]$, $\{1\} \times [t]$, \dots, $\{k-1\} \times [t]$ and $\phi_{\ell,s}$ preserves first coordinates.
So $\phi_{\ell,s}(F+w)$ has no edges between vertices which agree in the first coordinate.



We will show that, given two points $v_1 = (x_1, y_1, z_1)$ and $v_2 = (x_2, y_2, z_2)$ with $x_1 \neq x_2$, there is at most one pair $(\ell, s)$ such that $v_1$ and $v_2$ are contained in $\Phi_{\ell,s}$.
Indeed, suppose that $v_1$ and $v_2$ are contained in both $\Phi_{\ell, s}$ and $\Phi_{\ell', s'}$.
Then $z_1 = \ell + x_1 s = \ell' + x_1 s'$ and $z_2 = \ell+ x_2 s = \ell' + x_2 s'$, so $z_2 - z_1 = (x_2 - x_1)s = (x_2 - x_1)s'$.
It follows that $s = s'$ since the map $u \mapsto (x_1 - x_2) u$ from $[kt]$ to $[k^2t]$ is injective, hence also that 
$\ell=z_1-x_1s = z_1 -x_1s' = \ell'$.
Recall that $\phi_{\ell, s}(F+w)$ never has an edge between two vertices which agree in the first coordinate.
So for any $w,w'$, we have that $\phi_{\ell, s}(F+w)$ and $\phi_{\ell', s'}(F+w')$ are edge-disjoint whenever $(\ell,s) \neq (\ell',s')$.

Now fix an enumeration $w_0, \ldots, w_{kt-1}$ of $[k] \times [t]$.
Define $R := \bigcup_{\ell \in [k^2t], s \in [kt]} \phi_{\ell,s}(F+w_s)$ with vertex set $V$.
Clearly, $R$ has an $F$-decomposition, is $k$-partite (with colour classes $\{i\} \times [t] \times [k^2t]$ for $i \in [k]$), and has no multiple edges.
Since $\Phi_{0,s}, \ldots, \Phi_{k^2t-1,s}$ partition $V$ for each $s \in [kt]$, for any vertex $v =(x,y,z) \in V$ and any $s \in [kt]$ there is precisely one $\ell \in [k^2t]$ such that $v$ is a vertex of $\phi_{\ell, s} (F + w_s)$.
Thus
\begin{align*}
d_R(v) =  \sum_{s \in [kt]} d_{F+w_s}( (x,y) ) = \sum_{u \in V(F)} d_F(u) = 2e(F).
\end{align*} 
Hence $R$ is $2e(F)$-regular.
\end{proof}

We next show that, given a graph~$F$, we can turn an $F$-divisible graph into an $R$-divisible graph by removing a small number of copies of~$F$.%

\begin{lemma} \label{R-divisible}
Let $F$ be a graph and let $R$ be an $F$-decomposable $r$-regular graph with $r = 2e(F)$.
Let $\eps > 0$.
Then there exists an $n_0 = n_0( \epsilon , F,R)$ such that, for $n \geq n_0$, the following holds.
Let $G$ be an $F$-divisible graph on $n$ vertices with $\delta(G) \ge (1-1/r + 2\epsilon)n$.
Then there is an $F$-decomposable subgraph $H$ of $G$ such that $\Delta(H) \leq \epsilon n$ and $G-H$ is $R$-divisible.
\end{lemma}

\begin{proof}
Choose $0 \leq t < e(R)/e(F)$ such that $e(G) \equiv te(F) \mod e(R)$.
Let $F_1, F_2, \ldots, F_t$ be $t$ vertex-disjoint copies of $F$ in $G$\COMMENT{which exist by the greedy algorithm}, and let $G_0 := G - F_1 - \cdots - F_t$.
Then $G_0$ remains $F$-divisible and $e(G_0)$ is divisible by $e(R)$. Note that $\delta(G_0)\ge (1-1/r+\eps)n$.

Consider an $F$-decomposition $\mathcal{F}$ of $R$ and fix an $F' \in \mathcal{F}$.
Let $\mathcal{D} \subseteq \mathbb{N}$ be the set of vertex degrees of $F$.
For each $d \in \mathcal{D}$, let $v_d$ be a vertex of $F'$ with $d_{F'}(v_d) = d$, and let $S_d$ be the star consisting of $v_d$ together with the incident edges of $F'$.
Let $R_d$ be the graph obtained from $R - S_d$ by adding a new vertex $v_d'$ attached to the neighbours of $v_d$ in $F'$.%
\COMMENT{B: Think of this as lifting up a corner of one of the copies of $F$ in $R$ to expose a vertex of degree $d$.}
By construction, $R_d$ is $F$-decomposable, $|R_d| = |R| +1$, $e(R_d) = e(R)$ and every vertex of $R_d$ has degree $r$ except for $v_d'$, which has degree $d$, and $v_d$, which has degree $r-d$.

Fix an enumeration $u_1, \ldots, u_n$ of $V(G)$ and, for each $1 \leq i \leq n-1$, choose $0 \leq a_i < r$ such that $\sum_{j=1}^{i} d_{G_0}(u_j) \equiv a_i \mod r$.
Since both $R$ and $G_0$ are $F$-divisible, each $a_i$ is divisible by $\gcd(F)$, so there exists a multiset $T_i$ with $d \in \mathcal{D}$ for all $d \in T_i$ such that $\sum_{d \in T_i}d \equiv a_i \mod{r}$.
Moreover, since there exist only $r$ possible values for $a_i$, we may assume that there exists a $c = c(F)$ such that $|T_i| \le c$ for all~$i$.

Let $\P_0 := \{V(G)\}$ be the trivial partition of $V(G)$.%
For each $1 \leq i \leq n-1$ and each $d \in T_i$, choose a $\mathcal{P}_0$-labelled copy of~$R_d$ such that the copy of $v_d'$ is labelled $\{u_i\}$, the copy of $v_d$ is labeled $\{u_{i+1}\}$ and all other vertices are labelled~$V(G)$ (we may assume that these copies are vertex disjoint).
Let $\mathcal{R}_i$ be the set of copies of~$R_d$ (one for each $d \in T_i$).
Let $\mathcal R := \bigcup_{i=1}^{n-1} \mathcal{R}_i$.
So $|\mathcal{R}_i | = |T_i| \le c$ for all~$i$ and $|\mathcal R| \le c(n-1)$.
For each $i$, the number of indices such that some vertex of $R_d$ in $\mathcal{R}$ is labelled $\{ u_i \}$ is at most $|T_i| + |T_{i-1}| \le 2c$. (Here $|T_0| = |T_n| = 0$.)
Recall that each copy of $R_d$ has degeneracy at most~$r$ since $\Delta(R_d) = r$.
Pick $\eta$ such that $1/n \ll \eta \ll \eps, 1/r, 1/f$ and apply Lemma~\ref{lma:finding} with $G_0$, $1$, $|R|+1$, $r$,  $\eps/2$, $\P_0$, $\mathcal{R}$ playing the roles of $G$, $k$, $b$, $d$, $\eps$, $\mathcal{P}$, $\{H_1, \dots, H_m\}$.
We obtain edge-disjoint embeddings $\phi(R_d)$ for all $R_d \in \mathcal{R}$ into $G_0$, which are compatible with their
labelling and such that $\Delta ( \bigcup_{R_d \in \mathcal{R}} \phi(R_d) )  \le \epsilon n/2$.
Let $H_0 : = \bigcup_{R_d \in \mathcal{R}} \phi(R_d)$; so $\Delta(H_0) \leq \epsilon n/2$.

Let $G_1 := G_0 - H_0$.
Note that, for each $1 \le i \le n-1$ and each $R_d \in \mathcal{R}_i$, we have $d_{\phi(R_d)}(u_i) \equiv d \mod{r}$, $d_{\phi(R_d)}(u_{i+1}) \equiv -d \mod{r}$, and $d_{\phi(R_d)}(u_j) \equiv 0 \mod{r}$ for each $j \notin \{i, i+1\}$.
Recall that $\sum_{d \in T_i}d  \equiv a_i  \equiv \sum_{j=1}^{i} d_{G_0}(u_j) \mod r$ for each $1 \le i \le n-1$.
We have that
\begin{align*}
d_{H_0}(u_1)   \equiv \sum_{ R_d \in \mathcal{R}_1 } d_{\phi(R_d)}(u_1) \equiv \sum_{d \in T_1} d
	\equiv d_{G_0}( u_1 )
 \mod{r},
\end{align*}
so $r$ divides $d_{G_1}( u_1 )$.
Similarly, for $ 2 \le i \le n-1$ we have that
\begin{align*}
d_{H_0}(u_i) & = \sum_{j = 1}^{n-1} \sum_{ R_d \in \mathcal{R}_j } d_{\phi(R_d)}(u_i) 
		  \equiv \sum_{ R_d \in \mathcal{R}_{i-1} } d_{\phi(R_d)}(u_i)  + \sum_{ R_{d'} \in \mathcal{R}_i } d_{\phi(R_{d'})}(u_i) \mod{r}\\
		  & \equiv - \sum_{d \in T_{i-1}}  d  + \sum_{d' \in T_{i}} d' 
		\equiv - a_{i-1}+ a_i \equiv d_{G_0}( u_i )  \mod{r},
\end{align*}
so $r$ divides $d_{G_1}( u_i )$.
Recall that $e(R)$ divides $e(G_0)$ and $r=2e(F)$, so $r$ divides $2 e(G_0)$.
Finally, for $i=n$ we have that
\begin{align*}
d_{H_0}(u_n) &  \equiv \sum_{ R_d \in \mathcal{R}_{n-1} } d_{\phi(R_d)}(u_n) \equiv  - \sum_{d \in T_{n-1}} d 
	\equiv - \sum_{j=1}^{n-1} d_{G_0}(u_j)
 \mod{r},
\end{align*}
so $d_{G_1}( u_n ) \equiv \sum_{j=1}^{n} d_{G_0}(u_j) \equiv 2 e(G_0) \equiv  0 \mod{r}$.
Hence $G_1$ is $r$-divisible.
Since $G_1$ was obtained from $G_0$ by deleting graphs with $e(R)$ edges, $e(G_1)$ is divisible by $e(R)$, so $G_1$ is $R$-divisible.
Take $H := H_0 \cup F_1 \cup \cdots \cup F_t$ and observe that $\Delta(H) \le \eps n/2  + t r \le \eps n$.
\end{proof}

We now prove the following theorem, which together with Theorems~\ref{thm:dukes} and~\ref{thm:BKLMO} implies Theorem~\ref{thm:general}.

\begin{theorem} \label{thm:general3}
Let $F$ be a graph.
Then for each $\epsilon > 0$, there exists an $n_0 = n_0(\epsilon,F)$ and an $\eta  = \eta(\eps, F) $ such that every $F$-divisible graph $G$ on $n\geq n_0$ vertices with $\delta(G) \geq (\delta+\epsilon)n$, where $\delta : = \max \{ \delta^{*}_{K_{\chi(F)}} ,  1-1/6e(F) \}$, has an $F$-decomposition.
\end{theorem}

\begin{proof}
Choose $n_0\in \mathbb{N}$ and $\eta>0$ such that $1/n_0\ll \eta\ll \eps,1/|F|$.
Let $n\ge n_0$ and
let $G$ be an $F$-divisible graph on $n$ vertices with $\delta(G) \ge (\delta+\eps)n$.
By Lemma~\ref{regularise}, there is an $F$-decomposable $r$-regular graph $R$ with $r = 2e(F)$ and $\chi(R) = \chi(F)$.
By Lemma~\ref{R-divisible}, there is an $F$-decomposable subgraph $H$ of $G$ such that $\Delta(H) \leq \epsilon n/2$ and $G' : = G-H$ is $R$-divisible.

Corrollary~\ref{cor:chromatic} implies that $\delta^{\eta}_R \le \delta^{*}_{K_\chi(R)}$.
Thus $\delta(G') \ge \delta(G) - \Delta(H) \ge (\delta^{\eta}_R + \eps/2 ) n$.
Moreover, $1/6e(F)=1/3r$.
So Theorem~\ref{thm:regular} implies that $G'$ has an $R$-decomposition, hence also an $F$-decomposition.
\end{proof}


\section{Random subgraphs and partitions} \label{sec:random}

Let $m,n, N \in \N$ with $\max\{m,n\} < N$.
Recall that the hypergeometric distribution with parameters $N, n$ and $m$ is the distribution of the random variable $X$ defined as follows.
Let $S$ be a random subset of $\{1, 2, \ldots, N\}$ of size $n$ and let $X := |S \cap \{1, 2, \ldots, m\}|$.
We use the following simple form of Hoeffding's inequality, which we shall apply to both binomial and hypergeometric random variables.

\begin{lemma}[see {\cite[Remark 2.5 and Theorem 2.10]{JLR}}] \label{lma:chernoff}
Let $X\sim B(n,p)$ or let $X$ have a hypergeometric distribution with parameters $N,n,m$.
Then
\begin{align*}
\Pr(|X - \E(X)| \geq t) \leq 2e^{-2t^2/n}.
\end{align*}
\end{lemma}

The following lemma is a simple consequence of Lemma~\ref{lma:chernoff}.

\begin{lemma} \label{lma:random-slice}
Let $k,s \in \N$ and let $0 < \gamma, \rho < 1$.
There is an $n_0 = n_0(k,s,\gamma)$ such that the following holds.
Let $G$ be a graph on $n \geq n_0$ vertices and let $V_1, \ldots, V_k$ be an equitable partition of its vertex set.
Let $H$ be a graph on $V(G)$.
 Then there is a subgraph $R$ of $G$ such that, for each $1 \leq i \leq k$ and each $S \subseteq V(G)$ with $|S| \leq s$,
 \begin{align*}
 d_R(S,V_i) = \rho^{|S|}d_G(S,V_i)\pm \gamma |V_i|
 \end{align*}
and for each $x ,y\in V(G)$,
\begin{align*}
d_{H}( y, N_{R}(x, V_i) )  = \rho d_H (y, N_{G}(x , V_i)) \pm \gamma n.
\end{align*}
\end{lemma}

\begin{proof}
Let $R$ be a random subgraph of $G$ in which each edge is retained with probability $\rho$, independently from all other edges.
 By Lemma~\ref{lma:chernoff}, for each $1 \leq i \leq k$ and each $S \subseteq V(G)$ with $|S| \leq s$,%
   \begin{align*}
 \Pr(|d_R(S,V_i)-\rho^{|S|} d_G(S,V_i)| \geq \gamma |V_i|) \leq 2e^{-2(\gamma|V_i|)^2/|V_i|} \leq 2e^{-2\gamma^2 \lfloor n/k\rfloor}.
 \end{align*}
Similarly, for each $x, y \in V(G)$, 
 \begin{align*}
\Pr(|d_H(y, N_{R}(x, V_i))- \rho d_H (y, N_{G}(x, V_i)) | \geq \gamma n) \leq 2e^{-2 \gamma^2 n}.
\end{align*}
Since there are only at most $k(n+1)^s + k n^2$ conditions to check and each fails with probability exponentially small in $n$, some choice of $R$ has the required properties if $n$ is sufficiently large.
\end{proof}

Let $G$ be a graph.
For $k \in \N$ and $\delta >  0$, a \defn{$(k, \delta)$-partition for} $G$ is an equitable partition $\P = \{V_1, \ldots, V_k\}$ of $V(G)$ such that, for each $1 \leq i \leq k$ and each $v \in V(G)$, $d_G(v,V_i) \geq \delta|V_i|$.
We will often use the fact that if $\P$ is a $(k, \delta+\epsilon)$-partition for $G$ and $H$ is a subgraph of $G$ with $\Delta(H) \leq \epsilon n/2k$, then $\P$ is a $(k, \delta)$-partition for $G-H$.

\begin{proposition} \label{prop:random-partition}
Let $k \in \N$, and let $0 < \delta < 1$.
Then there exists an $n_0 = n_0(k)$ such that any graph $G$ on $n \ge n_0$ vertices with $\delta(G) \ge \delta n$ has a $(k, \delta - 2n^{-1/3})$-partition.
\end{proposition}

\begin{proof}
Consider a random equitable partition of $V(G)$ into $V_1, \dots, V_k$ with $|V_1| \leq |V_2| \leq \cdots \leq |V_k|$.
Consider any $1 \leq i \le k$ and any $v \in V(G)$.
Note that $d(v,V_i) = |V_i \cap N(v)|$ has a hypergeometric distribution with parameters $n, |V_i|$ and $d(v)$.
Thus, for each $1 \leq i \le k$ and for each $v \in V(G)$, by Lemma~\ref{lma:chernoff} we have that
\begin{align*}
	\Pr ( d(v,V_i) \le \delta |V_i| - n^{2/3}/k )  \le  2 e^{ - 2n^{1/3}/k }.
\end{align*}
So for $n$ sufficiently large we can choose an equitable partition $V_1, \dots, V_k$ such that, for each $i \le k$ and $v \in V(G)$,
\begin{align*}
	d(v,V_i) \ge \delta |V_i| - n^{2/3}/k \ge  ( \delta - 2n^{-1/3} ) |V_i|,
\end{align*}
as required.
\end{proof}

Let $\P_1$ be a partition of $V(G)$ and for each $1 < i \le \ell$, let $\mathcal{P}_i$ be a refinement of $\P_{i-1}$.
We call $\P_1, \ldots, \P_\ell$ a \defn{$(k, \delta, m)$-partition sequence for $G$} if
\begin{enumerate}[label={\rm(\roman*)}]
	\item $\P_1$ is a $(k, \delta)$-partition for $G$;
	\item for each $2 \leq i \leq \ell$ and each $V \in \P_{i-1}$, $\P_i[V]$ is a $(k, \delta)$-partition for $G[V]$;
	\item for each $V \in \P_\ell$, $|V| = m \text{ or } m-1$.
\end{enumerate}
Note that (i) and (ii) imply that each $\P_i$ is an equitable partition of $V(G)$.

\begin{lemma} \label{lma:partition-structure}
Let $k \in \mathbb{N}$ with $k \ge 2$\COMMENT{NEW}, and let $\delta, \epsilon > 0$.  
There exists an $m_0 = m_0(k, \epsilon)$ such that for all $m' \ge m_0$, any graph $G$ on $n \geq m'$ vertices with $\delta(G) \geq \delta n$ has a $(k, \delta-\epsilon , m )$-partition sequence for some $m' \le m \le k m'$.
\end{lemma}

\begin{proof}
Take $m_0 \geq \max\{ n_0(k) , 1000/ \eps^{3} \}$, where $n_0$ is the function from Proposition~\ref{prop:random-partition}, and let $m' \geq m_0$.
Let $\ell := \lfloor\log_k(n/m')\rfloor$.
Define $\P_0, \ldots, \P_\ell$ as follows.
Let $\P_0 := \{V(G)\}$.
For $j \in \N$, let $a_{j} := n^{-1/3} + (n/k)^{-1/3} + \cdots + (n/k^{j-1})^{-1/3}$. 
Suppose that for some $1 \le i \le \ell$ we have already chosen $\P_0, \ldots, \P_{i-1}$ such that, for each $1 \leq j \leq i-1$ and each $V \in \P_{j-1}$, $\P_j[V]$ is a $(k, \delta - 2 a_j )$-partition for $G[V]$.
Since $|V|+1 \ge n/k^{i-1} \geq n/k^{\ell-1} \geq m_0$, for each $V \in \P_{i-1}$ we can choose by Proposition~\ref{prop:random-partition} a $(k, \delta - 2 a_i)$-partition for $G[V]$.
Observing that\COMMENT{NEW: Calculation
\begin{align*}
a_{\ell} & = n^{-1/3} + (n/k)^{-1/3} + \cdots + (n/k^{\ell-1})^{-1/3}
	  = n^{-1/3} ( 1 + k^{1/3}+\dots +k^{(\ell-1)/3})\\
	& = n^{-1/3} \frac{k^{\ell/3}-1}{k^{1/3}-1} 
	\le \frac{n^{-1/3} k^{\ell/3}}{k^{1/3}-1} =  \frac {(n/k^{\ell})^{-1/3}} {k^{1/3} - 1} 
\end{align*}
}
\begin{align*}
a_{\ell} & \leq \frac {(n/k^{\ell})^{-1/3}} {k^{1/3} - 1} 
\leq \frac{(m')^{-1/3}}{2^{1/3}-1} \leq \frac{m_0^{-1/3}}{2^{1/3}-1} \le \frac \eps 2 
\end{align*}
completes the proof with $m = \lceil n / k^{\ell} \rceil$.
\end{proof}


\section{Absorbers} \label{sec:absorbers}

Let $F$ be an $r$-regular graph.
Suppose that $G$ is an $F$-divisible graph on $n$ vertices with large minimum degree. 
Let $\P_1, \dots, \P_{\ell}$ be a $(k, \delta, m)$-partition sequence for $G$ given by Lemma~\ref{lma:partition-structure}.
In our proof of Theorem~\ref{thm:regular}, we will choose the partition so that $m$ is bounded (i.e. each $ V \in \P_{\ell}$ has bounded size).
In Section~\ref{sec:partial-decomposition} we will show that $G$ can be decomposed into many copies of $F$ and a leftover graph $H^*$ such that $e(H^*[\P_{\ell}]) = 0 $.
Our aim in this section is to prove the following lemma.
It guarantees the existence of an `absorber' $A^*$ in a dense graph $G$, which can absorb this leftover graph $H^*$ (i.e. $A^* \cup H^*$ has an $F$-decomposition whatever the precise structure of $H^*$).

\begin{lemma} \label{lma:Krabsorber}
Suppose that $n , m , r,f \in \N$ and $\eps > 0 $ with $1/n \ll  1 / m \ll   1/r, 1/f,\eps$.
Let $\delta := 1 - 1/3r + \eps$, and let $q := \lceil n / m\rceil$.
Suppose that $F$ is an $r$-regular graph on $f$ vertices and $G$ is a graph on $n$ vertices.
Let $\P = \{ V_1, \dots, V_{q} \}$ be an equitable partition of $V(G)$ such that, for each $1 \leq i \leq q$, $|V_i| = m \text{ or } m-1$.
Suppose that $\delta( G[\P] ) \ge \delta n $ and $\delta(G[V_i]) \ge  \delta |V_i|$ for each $1 \leq i \le q $.%
\COMMENT{We cannot say `Suppose that $\P$ is a $(q,\delta )$-partition for $G$,' because $(q, \delta)$-partitions need every vertex to see lots of things in every part, not just their own and overall.}
Then $G$ contains an $F$-divisible subgraph~$A^*$ such that\COMMENT{`$A^*$ is $F$-divisible' already follows from (ii) but perhaps it is clearer this way.}
\begin{enumerate}[label={\rm(\roman*)}]
\item $\Delta(A^*[\P]) \le \eps^2 n$ and  $\Delta( A^*[V_i] ) \le  r$ for each $1 \leq i \le q$, and 
\item if $H^*$ is an $F$-divisible graph on $V(G)$ that is edge-disjoint from $A^*$ and has $e ( H^* [ \P ] ) = 0 $, then $A^* \cup H^*$ has an $F$-decomposition.
\end{enumerate}
\end{lemma}

Note that Lemma~\ref{lma:Krabsorber} implies that $A^*$ itself has an $F$-decomposition (by taking $H^*$ to be the empty graph). 
The crucial building blocks for the graph $A^*$ in Lemma~\ref{lma:Krabsorber} are $F$-absorbers.
An \defn{$F$-absorber for a graph $H$} is a graph $A$ such that 
\begin{itemize}
	\item $A$ and $H \cup A$ each have $F$-decompositions;
	\item $A[V(H)]$ is empty.
\end{itemize}

Here, we sketch the proof of Lemma~\ref{lma:Krabsorber}.
The graph $A^*$ given by Lemma~\ref{lma:Krabsorber} will consist of an edge-disjoint union of a set $\mathcal{A}$ of $F$-absorbers and a set $\mathcal{M}$ of `edge-movers'.
These graphs have low degeneracy and will be found using Lemma~\ref{lma:finding}.
The edge-movers will ensure that each $H^*[V_i]$ can be assumed to be $F$-divisible.
Then for each $1 \leq i \leq q$, $\mathcal A$ will contain an $F$-absorber $A_i$ for $H^*[V_i]$.

In the next subsection we explicitly construct an $F$-absorber for a given $F$-divisible graph~$H$ (where we may think of $H$ as one of the possibilities for $H^*[V_i]$).
We will construct this $F$-absorber~$A$ in a series of steps:
$A$ will consist of two `transformers' $T_1$ and $T_2$, where $T_1$ will transform $H$ into a specific graph $L_h$ with $h := e(H)$ and $T_2$ will transform $L_h$ into $p$ vertex-disjoint copies of~$F$, where $p := e(H)/e(F)$.
This latter graph is trivially $F$-decomposable. 
Notice that if an $F$-absorber for $H$ exists, then $H$ is $F$-divisible.
Therefore, for the rest of this section, all graphs $H$ are assumed to be $F$-divisible.

\subsection{An $F$-absorber for a given graph $H$} \label{sec:absstruture}
Given an $r$-regular graph $F$ and two vertex-disjoint graphs $H$ and $H'$, an \defn{$(H,H')_{F}$-transformer} is a graph $T$ such that
\begin{itemize}
	\item $T \cup H$ and $T \cup H'$ each have $F$-decompositions;
	\item $V(H \cup H') \subseteq V(T)$ and $T[V(H \cup H')]$ is empty.
\end{itemize}
Thus if $\emptygraph$ is an empty graph, then an $(H,\emptygraph)_{F}$-transformer is an $F$-absorber for~$H$.
Write $H \sim_{F} H'$ if there exists an $(H,H')_{F}$-transformer.  
The relation $\sim_{F}$ is clearly symmetric.
We now show that it is transitive on collections of vertex-disjoint graphs.

\begin{proposition} \label{relation}
Let $r\in \mathbb{N}$ and let $F$ be an $r$-regular graph.
Suppose that $H$, $H'$ and $H''$ are vertex-disjoint graphs.
Let $T_1$ be an $(H,H')_{F}$-transformer, and let $T_2$ be an $(H',H'')_{F}$-transformer such that $V(T_1) \cap V(T_2) = V(H')$.
Then $T := T_1 \cup H' \cup T_2 $ is an $(H,H'')_{F}$-transformer.
\end{proposition}

\begin{proof}
Observe that $T \cup H = (T_1 \cup H) \cup (T_2 \cup  H')$ and $T \cup H'' = (T_1 \cup H') \cup (T_2  \cup H'')$ each have $F$-decompositions. 
\end{proof}

We will show that in fact $H \sim_{F} H'$ for all vertex-disjoint $F$-divisible graphs $H$ and $H'$.
Since the empty graph is $F$-divisible, this in turn implies that every such $H$ has an $F$-absorber.
We will further show that, for each such $H$, we can find an $F$-absorber for $H$ which has low degeneracy (rooted at $V(H)$).

We say that a graph $H'$ is \emph{obtained from a graph $H$ by identifying vertices} if there is a sequence of graphs $H_0, \dots, H_{s}$ and vertices $x_i,y_i \in V(H_i)$ such that 
\begin{enumerate}[label=(\roman*)]
		\item $H_0 = H$ and $H_s = H'$;
		\item $( N_{H_i}(x_i) \cup\{ x_i \} ) \cap ( N_{H_i}(y_i) \cup\{ y_i \} ) = \emptyset$ for all $i$;
		\item for each $0 \le i < s $, $H_{i+1}$ is obtained from $H_{i}$ by identifying the vertices $x_i$ and $y_i$.
\end{enumerate}
Condition (ii) ensures that the identifications do not produce multiple edges.
Note that if $H'$ can be obtained from $H$ by identifying vertices, then there exists a graph homomorphism $\phi: H \rightarrow  H'$ from $H$ to $H'$ that is edge-bijective.
Recall that a graph $H$ is $r$-divisible if $r$ divides $d(v)$ for all $v \in V(H)$.

\begin{fact} \label{can-split}
Let $r \in \mathbb{N}$ and let $H$ be an $r$-divisible graph.
Then there is an $r$-regular graph $H_0$ such that $H$ can be obtained from $H_0$ by identifying vertices.
\end{fact}

\begin{proof}
Split each vertex of degree $sr$ in $H$ into $s$ new vertices each of degree $r$.
\end{proof}

Fact~\ref{can-split} and the next lemma together imply that, for every $F$-divisible graph $H'$, there is some $r$-regular graph~$H$ such that $H \sim_{F} H'$.
Recall that the degeneracy of a graph $H'$ rooted at $U \subseteq V(H')$ was defined in Section~\ref{sec:embedding}.

\begin{lemma} \label{lma:regular}
Let $r,f \in \mathbb{N}$ and let $F$ be an $r$-regular graph on $f$ vertices.
Let $H$ be an $r$-regular graph.
Let $H'$ be a copy of a graph obtained from $H$ by identifying vertices.
Suppose that $H$ and $H'$ are vertex-disjoint.
Then $H \sim_{F} H'$.
Moreover, there exists an $(H,H')_{F}$-transformer~$T$ such that the degeneracy of $T$ rooted at $V(H \cup H')$ is at most $3r$ and $|T| \le  f r |H| + |H'| + f e(H)$.
\end{lemma}

\begin{proof}
Let $uv$ be an edge of $F$ and let $u,v,z_1, \dots, z_{f-2}$ be the vertices of~$F$.
Let $N_F(u) = \{v , z_{a_1}, \dots, z_{a_{r-1}}\}$ and $N_F(v) = \{u , z_{b_1}, \dots, z_{b_{r-1}}\}$.
(The indices $a_i$ and $b_i$ will be fixed throughout the rest of the proof.)

Let $\phi: H \to H'$ be a graph homomorphism from $H$ to $H'$ that is edge-bijective.
Orient the edges of $H$ arbitrarily. 
Then $\phi$ induces an orientation of $H'$.%
\COMMENT{Mental image: Since each edge $xy \in E(H)$ is oriented, we map $u$ to $x$ and $v$ to $y$ and $z_i$ to $z_i^{(xy)}$.}
Throughout the rest of the proof, we view $H$ and $H'$ as oriented graphs and we write $xy$ for the oriented edge from $x$ to~$y$.

For each $e \in E(H)$, let $Z^{e}:= \{ z_1^{(e)}, \dots,  z_{f-2}^{(e)}\}$ be a set of $f-2$ vertices such that $V(H)$, $V(H')$, $Z^e$ and $Z^{e'}$ are disjoint for all distinct $e,e' \in E(H)$. 
Define a graph $T_1$ as follows:
\begin{enumerate}[label=(\roman*)]
\item $V(T_1) := V(H) \cup V(H')  \cup \bigcup_{e \in E(H)} Z^e$;
\item $E_1 := \{ x z_{a_i}^{(xy)}, y z_{b_i}^{(xy)}  : 1 \le i \le r-1$ and $x y \in E(H) \}$;
\item $E_2 := \{ z_i^{(xy)} z_{j}^{(xy)} : z_iz_j \in E(F)$ and $ xy \in E(H)  \}$;
\item $E_3 := \{ \phi(x) z_{a_i}^{(xy)}, \phi(y) z_{b_i}^{(xy)}  : 1 \le i \le r-1$ and $x y \in E(H) \}$;
\item $E(T_1) := E_1 \cup E_2 \cup E_3$.
\end{enumerate}
Note that $T_1 [ V(H \cup H')]$ is empty.
Note also that $H \cup E_1 \cup E_2$ can be decomposed into $e(H)$ copies of $F$, where each copy of $F$ has vertex set $\{x, y\} \cup Z^{(xy)}$ for some edge $xy \in E(H)$.
Similarly, $H' \cup E_2 \cup E_3$ can be decomposed into $e(H)$ copies of $F$.%
\COMMENT{where each $F$ has vertex set $\{\phi(x), \phi(y) \} \cup Z^{(xy)}$ for some edge $xy \in E(H)$.}
In summary, 
\begin{align}
	\text{$H \cup E_1 \cup E_2$ and $H' \cup E_2 \cup E_3$ each have $F$-decompositions.} \label{eqn:regular1}
\end{align}
Note that every vertex $z \in V(T_1) \setminus V(H \cup H')$ satisfies
\begin{align}
	d_{T_1}(z) \le \max \{ r, 1 + (r-1)+1, 2 + (r-2)+2\} = r+2. \label{eqn:regularT1}
\end{align}

We will now construct an additional graph $T_2$ such that both $T_2 \cup E_1$ and $T_2 \cup E_3$ have an $F$-decomposition.
It will then follow that $T_1 \cup T_2$ is an $(H,H')_F$-transformer.
Note that $E_1$ is the edge-disjoint union of $|H|$ stars $K_{1,r(r-1)}$ with centres in $V(H)$.
We will obtain $T_2$ by viewing each star $K_{1,r(r-1)}$ as the union of $r-1$ smaller stars $K_{1,r}$, whose leaves form independent sets in $T_1$, and extending each of the smaller stars to a copy of $F$.

For each $x \in V(H)$, each neighbour~$y$ of~$x$ in~$H$ and each $1 \le j \le r-1$, let $u_j^{(xy)}: = z_{a_j}^{(xy)}$ if the edge between $x$ and $y$ in $H$ is directed toward~$y$; otherwise let $u_j^{(xy)}: = z_{b_j}^{(yx)}$. 
For each $x \in V(H)$ and each $1 \leq j \leq r-1$, let $N_j^x := \{u^{(xy)}_j : y \in N_H(x)\}$.
The $N_j^x$ partition $N_{T_1}(x)$ and each $N_j^x$ forms an independent set in $T_1$.

For each $x \in V(H)$ and each $1 \leq j \leq r-1$, let $W_j^x$ be a set of $f-(r+1)$ new vertices, disjoint from both $V(T_1)$ and the other $W_{j'}^{x'}$.
Fix a vertex $x_0 \in V(F)$.
Define a graph $T^x_j$ on vertex set $V(T^x_j) : = N^x_j \cup  W^x_j$ such that $T^x_j$ is isomorphic to $F \setminus x_0$ and the image of $N_F(x_0)$ is precisely $N^x_j$.
Then the $T^x_j$ are edge-disjoint and, for each $x \in V(H)$ and each $1 \leq j \leq r-1$, both $T_1[\{x\} \cup N_j^x] \cup T^x_j$ and $T_1[\{\phi(x)\} \cup N_j^x] \cup T^x_j$ are copies of~$F$.
Let $T_2 : = \bigcup_{x \in V(H)} \bigcup_{j=1}^{r-1} T^x_j$ and let $T: = T_1 \cup T_2$.
See Figure~\ref{fig:transformer} for an example with $F = C_6$.

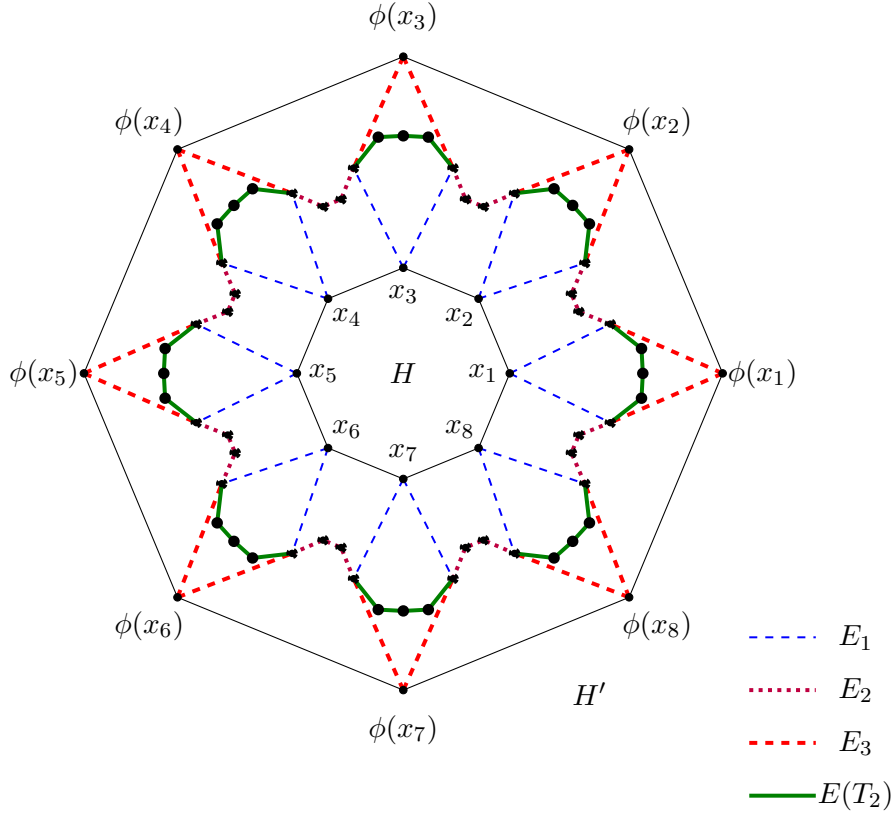
\begin{figure}[t]
\begin{tikzpicture}[scale=0.7]
	\begin{scope}
		\foreach \x in {0,1,...,7}
			{
			\begin{scope}[blue,dashed,line width=0.75pt]
			\draw (\x*45:2) to (\x*45+13.5:4);
			\draw (\x*45:2) to (\x*45-13.5:4);
			\end{scope}
			
			\begin{scope}[red, dashed,line width=1.5pt]
			\draw (\x*45:6) to (\x*45+13.5:4);
			\draw (\x*45:6) to (\x*45-13.5:4);
			\end{scope}
			}

 	\end{scope}

\begin{scope}[line width=1.5pt,black!50!green]
		\foreach \x in {0,1,...,7}
			{
			\draw (\x*45-13.5:4) to (\x*45 -6 :4.5) to (\x*45 :4.5) to (\x*45 +6 :4.5)  to  (\x*45+13.5:4);
			\foreach \y in {-1,0,1}
				{
			\filldraw[black, fill=black] ( {\x*45 + \y*6}  :4.5) circle (2pt);
				}
			}
\end{scope}
	\begin{scope}[line width=1.5pt, purple,dotted]
		\foreach \x in {0,1,...,7}
			{
			\draw (\x*45+13.5:4) to (\x*45 + 6 +13.5 :3.5) to (\x*45 + 2*6 +13.5 :3.5) to (\x*45+31.5:4);
						\foreach \y in {1,2}
				{
			\filldraw[black, fill=black] (\x*45 + \y*6 +13.5 :3.5) circle (2pt);
							}
			\filldraw[black, fill=black] (\x*45 + 13.5 :4) circle (2pt);
			\filldraw[black, fill=black] (\x*45 + 31.5 :4) circle (2pt);
			}
 	\end{scope}
	\begin{scope}
		\foreach \x [count=\y from 1] in {0,1,...,7}
			{\filldraw[fill=black] (\x*45:6) circle (2pt);
			\node  at (\x*45:6.75)  {$\phi(x_{\y})$};
			\draw (\x*45:6)	 to (\y*45:6);
			}
 	\end{scope}

	\begin{scope}
    ] 
		\foreach \x [count=\y from 1] in {0,1,...,7}
			{\filldraw[fill=black] (\x*45:2) circle (2pt);
			\node  at (\x*45:1.5)  {$x_{\y}$};
			\draw (\x*45:2)	 to (\y*45:2);
			}
 	\end{scope}
	\node  at (300:7)  {$H'$};
	\node  at (0,0)  {$H$};

	\draw[blue,dashed,line width=0.75pt] (6.5,-5) to (7.75,-5);
		\node  at (8.5,-5)  {$E_1$};
	
	\begin{scope}[line width=1.5pt]
	\draw[purple,dotted] (6.5,-6) to (7.75,-6);
		\node  at (8.5,-6)  {$E_2$};
	
	\draw[red, dashed] (6.5,-7) to (7.75,-7);
		\node  at (8.5,-7)  {$E_3$};
	
	\draw[black!50!green] (6.5,-8) to (7.75,-8);
		\node  at (8.5,-8)  {$E(T_2)$};
			\end{scope}

\end{tikzpicture}
\caption{An $(H , H')_{C_6}$- transformer, where $H$ and $H'$ are vertex-disjoint copies of $C_8$.
}
\label{fig:transformer}
\end{figure}

We now claim that $T$ is an $(H,H')_F$-transformer. 
Note that $T_2$ is edge-disjoint from $T_1$.
Since $T_2[ V ( H \cup H' ) ]$ is empty, $T[ V ( H \cup H' ) ]$ is empty. 
Note that $T_2 \cup E_1$ has an $F$-decomposition into $(r-1)|H|$ copies of $F$, where each copy of $F$ has vertex set $\{x\} \cup V(T^x_j)$ for some $x \in V(H)$ and some $1 \leq j \le r-1$.
Together with~\eqref{eqn:regular1}, this implies that $T \cup H' = (T_2 \cup E_1) \cup (E_2 \cup E_3 \cup H')$ has an $F$-decomposition.
Similarly $T_2 \cup E_3$ has an $F$-decomposition into $(r-1)|H|$ copies of $F$, where each $F$ has vertex set $\{\phi(x)\} \cup V(T^x_j)$ for some $x \in V(H)$ and some $1 \leq j \le r-1$.
So $T \cup H = (T_2 \cup E_3) \cup (H \cup E_1 \cup E_2)$ also has an $F$-decomposition.
Hence $T$ is indeed an $(H,H')_F$-transformer.

Note that each vertex in $W^x_j$ has degree $r$ in $T$.
By~\eqref{eqn:regularT1}, each vertex $z \in V(T_1) \setminus V(H \cup H')$ has degree at most $r+2 +2(r-1) = 3r$ in $T$.\COMMENT{The 2 is for each end of an edge $xy$, the $r-1$ for the degree of a vertex of $T_1 \setminus (H \cup H')$ in $T_j^x$.}
Therefore, $T$ has degeneracy at most $3r$ rooted at $V(H \cup H')$ and $|T| = |H| + |H'| + (f-2)e(H) + (f-r-1) (r-1) |H| \le  f r |H| + |H'| + f e(H)$.
\end{proof}

We remark that if the girth of $F$ is large, then the degeneracy of the $(H,H')_{F}$-transformer constructed in the proof of Lemma~\ref{lma:regular}, rooted at $V(H \cup H')$, is in fact smaller than $3r$.
We will use this fact, captured by the following lemma, in Section~\ref{sec:cycles}.

\begin{lemma} \label{lma:cycletransformer}
Let $r,f \in \mathbb{N}$ and let $F$ be an $r$-regular graph on $f$ vertices.
Suppose that $F$ contains a vertex which is not contained in any triangle in~$F$. 
Let $H$ be an $r$-regular graph.
Let $H'$ be a copy of a graph obtained from $H$ by identifying vertices.
Suppose that $H$ and $H'$ are vertex-disjoint.
Then $H \sim_{F} H'$.
Moreover, there exists an $(H,H')_{F}$-transformer~$T$ such that 
\begin{enumerate}[label=\rm(\roman*)]
	\item the degeneracy of $T$ rooted at $V(H \cup H')$ is at most $r+1$;
	\item if $F$ contains an edge $uv$ that is not contained in any triangle or cycle of length $4$ in~$F$, then the degeneracy of $T$ rooted at $V(H \cup H')$ is at most $r$.
\end{enumerate}
\end{lemma}

\begin{proof}
Let $x_0$ be a vertex of $F$ which is not contained in any triangle in~$F$. 
So $N_F(x_0)$ is an independent set in~$F$.
Also, $F$ must contain an edge $uv$  which is not contained in a triangle
(since $r \ge 1$, we can take any edge incident to $x_0$).
So $N_F(u)$ and $N_F(v)$ are disjoint.
Moreover, if $uv$ is not contained in any cycle of length 4, then $N_F(u) \setminus \{v\}$ and $N_F(v) \setminus \{u\}$ are disjoint sets of vertices with no edges between them.

Let $u,v,z_1, \dots, z_{f-2}$ be the vertices of~$F$.
Let $N_F(u) = \{v , z_{a_1}, \dots, z_{a_{r-1}}\}$ and $N_F(v) = \{u , z_{b_1}, \dots, z_{b_{r-1}}\}$.
Let $T$ be the $(H,H')_{F}$-transformer as defined in the proof of Lemma~\ref{lma:regular}
(with $x_0$ playing the role of $x_0$ in the proof of Lemma~\ref{lma:regular}).
To see that the degeneracy of $T$ rooted at $V(H \cup H')$ is as desired, consider the vertices in $H$, $H'$, $T_1 \setminus (H \cup H')$ and $T_2 \setminus T_1$ in that order with the vertices of $T_1 \setminus (H \cup H')$ ordered such that for each edge $xy \in E(H)$, the vertices $z_{a_1}^{(xy)},\dots, z_{a_{r-1}}^{(xy)}$, $z_{b_1}^{(xy)},\dots, z_{b_{r-1}}^{(xy)}$ come before $z_{j}^{(xy)}$ for $j \notin \{a_1, \dots, a_{r-1}, b_1, \dots, b_{r-1} \}$.
\end{proof}

Recall that the relation $\sim_{F}$ is transitive (on vertex-disjoint graphs) by Proposition~\ref{relation}.
By Lemma~\ref{lma:regular}, to show that $H \sim_{F} H'$ it suffices to show that there exists an $r$-regular graph $H_0$ (vertex-disjoint from both $H$ and $H'$) so that we can obtain both $H$ and $H'$ from a copy of $H_0$ by identifying vertices.
In Lemma~\ref{lma:extendedloop} we will construct such an~$H_0$ for $r$-divisible graphs $H$ and $H'$ with the same number of edges.

Fix an edge $uv \in E(F)$.
The following construction will enable us to identify vertices even if they are adjacent.
Given a graph $H$ and an edge~$xy$ of~$H$, the \defn{$F$-expansion of $xy$ via $(u,v)$} is defined as follows.  
Consider a copy $F'$ of $F$ which is vertex-disjoint from~$H$.
Delete $xy$ from $H$ and $uv$ from $F'$ and join $x$ to $u$ and join $y$ to $v$ (see Figure~\ref{fig:expanded-edge}).

If $x \in V(H)$, then $H$ \emph{with a copy of $F$ attached to~$x$ via~$v$} is the graph obtained from $F' \cup H$ by identifying $x$ and $v$ (where as before, $F'$ is a copy of $F$ which is vertex-disjoint from~$H$). 
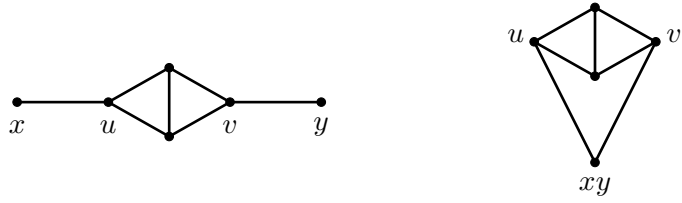
\begin{figure}[t!]
\centering
\begin{tikzpicture}[scale=0.8]
			\filldraw[fill=black] (1,0) circle (2pt);
			\filldraw[fill=black] (-1,0) circle (2pt);
			\filldraw[fill=black] (0,0.57) circle (2pt);
			\filldraw[fill=black] (0,-0.57) circle (2pt);
			\filldraw[fill=black] (2.5,0) circle (2pt);
			\filldraw[fill=black] (-2.5,0) circle (2pt);
			
			\node at (-2.5,-0.4)  {$x$};
			\node at (2.5,-0.4)  {$y$};
			\node at (-1,-0.4)  {$u$};
			\node at (1,-0.4)  {$v$};
			
			\draw[line width=1pt] (-2.5,0) -- (-1,0) --(0,0.57) --(1,0) --(0,-0.57)--(-1,0);
			\draw[line width=1pt] (2.5,0) -- (1,0);
			\draw[line width=1pt] (0,0.57) -- (0,-0.57);

\begin{scope}[shift={(7,1)}]

			\filldraw[fill=black] (1,0) circle (2pt);
			\filldraw[fill=black] (-1,0) circle (2pt);
			\filldraw[fill=black] (0,0.57) circle (2pt);
			\filldraw[fill=black] (0,-0.57) circle (2pt);
			\filldraw[fill=black] (0,-2) circle (2pt);
			
			\node at (0,-2.4)  {$xy$};
			\node at (-1.3,0.1)  {$u$};
			\node at (1.3,0.1)  {$v$};
			
			\draw[line width=1pt] (0,-2) -- (-1,0) --(0,0.57) --(1,0) --(0,-0.57)--(-1,0);
			\draw[line width=1pt] (0,-2) -- (1,0);
			\draw[line width=1pt] (0,0.57) -- (0,-0.57);
\end{scope}			

\end{tikzpicture}

\caption{A $K_4$-expanded edge and a $K_4$-expanded loop.}
\label{fig:expanded-edge}

\end{figure}

\begin{fact} \label{expand}
Let $F$ be an $r$-regular graph and let $uv \in E(F)$.
Suppose that the graph $H'$ is obtained from a graph $H$ by $F$-expanding an edge $xy \in E(H)$ via~$(u,v)$.
Then the graph obtained from $H'$ by identifying $x$ and $v$ is $H$ with a copy of $F$ attached to $x$ via~$v$.\qedhere
\end{fact}
 
Recall that we have fixed an edge~$uv$ of~$F$.
An \defn{$F$-expanded loop}~$L$ is the $F$-expansion of an edge $xy$ via $(u,v)$ with the vertices $x$ and $y$ identified
(see Figure~\ref{fig:expanded-edge}).  
Write $L_h$ for $h$ vertex-disjoint copies of $L$ with their distinguished vertices identified.
(The edge $uv \in E(F)$ used in $F$-expansions is always the same, so $L_h$ is uniquely defined.)
\begin{lemma} \label{lma:extendedloop}
Let $r,f \in \mathbb{N}$ and let $F$ be an $r$-regular graph on $f$ vertices.
Suppose that $H$ is an $r$-divisible graph with $h := e(H)$, and that $L_h$ is vertex-disjoint from~$H$.
Then $H \sim_{F} L_h$.
Moreover, there exists an $(H,L_h)_F$-transformer $T$ such that 
the degeneracy of $T$ rooted at $V(H \cup L_h)$ is at most $3r$ and $|T|  \le |H| + |L_h| + 7 f^2 r h$.
\end{lemma}

\newcommand{\Hatt}{H_{\text{att}}}
\newcommand{\Hexp}{H_{\text{exp}}}

\begin{proof}
Recall that we have fixed an edge~$uv$ of~$F$.
For each edge $e \in E(H)$, attach a copy of $F$ to one of its endpoints (chosen arbitrarily) via~$v$; call the resulting graph $\Hatt$.
Note that $|\Hatt| = |H| + (f-1)h$ and $e(\Hatt) = \left( e(F) + 1 \right) h$.
Let $\Hexp$ be the graph obtained from $H$ by $F$-expanding every edge in $H$ via~$(u,v)$.
By Fact~\ref{expand}, we can choose $\Hexp$ and $\Hatt$ such that $\Hatt$ can be obtained from $\Hexp$ by identifying vertices.
By Fact~\ref{can-split}, there is an $r$-regular graph $H_0$ such that $\Hexp$ (and so also $\Hatt$) can be obtained from (a copy of) $H_0$ by identifying vertices.

Lemma~\ref{lma:regular} implies that $H_0 \sim_{F} \Hatt$ and that there exists an $(H_0,\Hatt)_F$-transformer $T_1$ such that the degeneracy of $T_1$ rooted at $V(H_0 \cup \Hatt)$ is at most $3r$ and
\begin{align}
|T_1| 
\le f r |H_0| + |\Hatt| +  f e(H_0) .
\label{eqn:expand1}
\end{align}
Furthermore, we can choose $T_1$ such that $V(T_1) \cap V(L_h) = \emptyset$.

In $\Hexp$ the original vertices of $H$ are non-adjacent with disjoint neighbourhoods, so by identifying all original vertices of $H$ we obtain a copy of $L_h$ from $\Hexp$.
Hence $L_h$ can also be obtained from $H_0$ by identifying vertices, so Lemma~\ref{lma:regular} implies that there exists an $(H_0,L_h)_F$-transformer~$T_2$ such that the degeneracy of $T_2$ rooted at $V(H_0 \cup L_h)$ is at most $3r$ and 
\begin{align}
|T_2| 
\le f r |H_0|  + |L_h| +  f e(H_0).
\label{eqn:expand2}
\end{align}
Furthermore, we can choose $T_2$ such that $V(T_1) \cap V(T_2) = V(H_0)$.
So $T_1$ and $T_2$ are edge-disjoint.

By Proposition~\ref{relation}, $T_1 \cup H_0 \cup T_2$ is an $(\Hatt,L_h)_F$-transformer.
Define the graph $T$ to be $(\Hatt - H) \cup T_1 \cup H_0 \cup T_2$. 
Since $\Hatt - H$ trivially has an $F$-decomposition, it follows that $T$ is an $(H,L_h)_F$-transformer.
To see that $T$ has degeneracy at most $3r$ rooted at $V(H \cup L_h)$, consider the vertices in $H \cup L_h$, $\Hatt \setminus H$, $H_0$, $T_1 \setminus ( \Hatt \cup H_0 )$ and $T_2 \setminus ( L_h \cup H_0 )$ in that order. 

Recall that $|\Hatt| = |H| + (f-1)h$ and $e(H_0) = e(\Hatt) = \left( e(F) + 1 \right) h \le rfh$.%
\COMMENT{AL: $e(F) + 1 = rf/2 +1 \le rf$, since $f \ge r+1$.}
Since $H_0$ is $r$-regular, $|H_0| = 2 e(H_0)/r  \le 2 f h$.
By \eqref{eqn:expand1} and \eqref{eqn:expand2},
\begin{align*}
|T| & = | T_1 | + |T_2|  - |H_0| 
	    \le |\Hatt| + |L_h| +  2 f r |H_0| + 2 f e(H_0) \\
     &  \le  |H| + |L_h| + 7f^2 r h.
\end{align*}
This completes the proof of the lemma.
\end{proof}

We can now combine Lemma~\ref{lma:extendedloop} and Proposition~\ref{relation} to show that every $F$-divisible graph~$H$  has an $F$-absorber. 
Recall that $pF$ consists of $p$ vertex-disjoint copies of $F$. 
	
\begin{lemma} \label{lma:abs}
Let $r,f \in \mathbb{N}$ and let $F$ be an $r$-regular graph on $f$ vertices.
Let $H$ be an $F$-divisible graph.
Then there is an $F$-absorber $A$ for $H$ such that the degeneracy of $A$ rooted at $V(H)$ is at most $3r$ and $|A| \le 9 f^2 r |H|^2$.
\end{lemma}

\begin{proof}
Let $h := e(H)$ and let $p := e(H)/e(F)$.
Let $H$, $L_h$ and $p F $ be vertex-disjoint.
By Lemma~\ref{lma:extendedloop}, there exists an $(H,L_h)_F$-transformer $T_1$ such that 
the degeneracy of~$T_1$ rooted at $V(H \cup L_h)$ is at most $3r$ and
\begin{align*}
|T_1|  \le |H| + |L_h| +  7 f^2 r h \le |L_h|+  4 f^2 r |H|^2.
\end{align*}
Similarly by Lemma~\ref{lma:extendedloop}, there exists an $(L_h, p F )_F$-transformer $T_2$ such that the degeneracy of $T_2$ rooted at $V( L_h \cup p F)$ is at most $3r$ and
\begin{align*}
|T_2|  \le |p F| + |L_h| + 7 f^2 r h
	=  p f + h f + 1  + 7 f^2 r h \le 5 f^2 r |H|^2.
\end{align*}
Furthermore, we can choose $T_1$ and $T_2$ such that $V(T_1) \cap V(T_2) = V(L_h)$.
Let $A' : = T_1 \cup L_h \cup T_2$ and let $A:= A' \cup p F$.
By Proposition~\ref{relation}, $A'$ is an $( H , p F)_F$-transformer.
Thus $A$ is an $F$-absorber for $H$ with $|A| = |T_1| + |T_2| - |L_h| \le 9 f^2 r |H|^2$.
To see that the degeneracy of $A$ rooted at $V(H)$ is at most $3r$, consider the vertices in $H$, $L_h$, $p F $ and $T_1 \setminus (H \cup L_h) $  and $T_2 \setminus (p F \cup L_h)$ in that order (with the vertices of $L_h$ ordered such that the distinguished vertex comes first).
\end{proof}

\subsection{Proof of Lemma~\ref{lma:Krabsorber}}

Let $H$ be an $F$-divisible graph and let $\P = \{V_1, \ldots, V_q\}$ be a partition of its vertex set with $e(H[\P]) = 0$. 
(So $H$ is the disjoint union of the $H[V_i]$.)
We would like to absorb $H$ by using Lemma~\ref{lma:abs} to find an $F$-absorber for each graph $H[V_i]$ separately.
However, note that some $H[V_i]$ might not be $F$-divisible, as $e(H[V_i])$ might not be divisible by $e(F)$ for some $1 \leq i \le q$. 
We will use `edge-movers' to fix this problem.
We first make the following simple observation, which will be used in the construction of these edge-movers.

\begin{proposition} \label{prop:hcf}
Let $r \in \mathbb{N}$ and let\COMMENT{$a$ is the highest common factor of the numbers of edges of $r$-regular graphs.}
\begin{align*}
	a := a_r =
	\begin{cases}
	 r	 & \text{if $r$ is odd,}\\
	 r/2 & \text{if $r$ is even.}
	\end{cases}	
\end{align*}
\begin{enumerate}[label=\rm(\roman*)]
	\item Let $H$ be an $r$-divisible graph.
	      Then $e(H)$ is divisible by $a$.
	\item Let $f \in \mathbb{N}$ and let $F$ be an $r$-regular graph on $f$ vertices.
		If $r$ is odd, then let $Q$ be an $r$-regular bipartite graph with each vertex class having size $f+1$.		
		If $r$ is even, then let $Q$ be an $r$-regular graph on $2f+1$ vertices consisting of $r/2$ edge-disjoint Hamilton cycles on $V(Q)$.\COMMENT{Bipartite case: consider $1$-factorization of $K_{f+1,f+1}$. Non-bipartite case: apply Walecki.}
 	  Then $e(Q) \equiv a \mod{ e(F)}$.
\end{enumerate}
\end{proposition}

\begin{proof}
(i) holds since $2e(H) = \sum_{v \in V(H)} d_H(v) = p r$ for some $p \in \N$.

To see (ii), note that $e(F) = rf/2$. 
If $r$ is odd,  then $e(Q) = rf + r$; if $r$ is even, then $e(Q) = rf +r/2$. 
\end{proof}

Let $U$ and $V$ be disjoint vertex sets.
Let $r,f \in \mathbb{N}$ and let $F$ be an $r$-regular graph on $f$ vertices.
A \emph{$(U,V)_F$-edge-mover} is a graph~$M$ such that
\begin{enumerate}[label=(\roman*)]
	\item $E(M)$ can be partitioned into $E(Q)$, $E(\widetilde{Q})$ and $E(A)$;
	\item $Q$ is $r$-regular and $V ( Q ) \subseteq U$;
	\item $\widetilde{Q}$ is $r$-regular and $V ( \widetilde{Q} ) \subseteq V$;
	\item $e(Q) \equiv a \mod{ e(F)}$ and $e(\widetilde{Q}) \equiv -a \mod{ e(F) }$, where $a$ is as defined in Proposition~\ref{prop:hcf};
	\item $A$ is an $F$-absorber for $Q \cup \widetilde{Q}$.
\end{enumerate}
Since $A$ is an $F$-absorber for $Q \cup \widetilde Q$, both $M$ and $A$ have $F$-decompositions. 
Roughly speaking, a $(U,V)_F$-edge-mover allows us to move $a \mod{ e(F) }$ edges from $V$ to $U$ (by adding $Q$ and $\widetilde{Q}$ to the existing graph).

We are now ready to prove Lemma~\ref{lma:Krabsorber}.
In the proof, we find the copies of $Q$ and~$\widetilde{Q}$ in $G - G[\P]$, and the $F$-absorbers in $G[\P]$.

\begin{proof}[Proof of Lemma~$\ref{lma:Krabsorber}$]
Let $a$ and $Q$ be as defined in Proposition~\ref{prop:hcf}.
Let $\widetilde{Q} := (f-1) Q$.
Thus $\chi(Q) = \chi(\widetilde{Q}) \le r+1$.
Note that $\delta(G[V_i]) \ge (1 - 1/r + \eps ) |V_i|$ and $1/|V_i| \ll 1/r,1/f$.
So by the Erd\H{o}s--Stone--Simonovits theorem~\cite{Erdos,Simonovits}, for each $1 \leq i < q$, we can find $f$ copies of $Q$ in $G[V_i]$, and, 
for each $1 < i \le q $, we can find $f$ copies of $\widetilde Q$ in $G[V_{i}]$ so that all of these copies are vertex-disjoint.
Call these copies $Q^i_1, \ldots, Q^i_f$ and $\widetilde{Q}^{i}_1, \ldots, \widetilde{Q}^{i}_f$ respectively.

Proposition~\ref{prop:hcf}(ii) implies that $Q^i_j \cup \widetilde{Q}^{i+1}_j$ is $F$-divisible for all $1 \leq i < q$ and all $1 \le j \le f$.
Apply Lemma~\ref{lma:abs} to obtain an $F$-absorber $A^i_j$ for $Q^i_j \cup \widetilde{Q}^{i+1}_j$ such that the degeneracy of $A^i_j$ rooted at $V(Q^i_j \cup \widetilde{Q}^{i+1}_j)$ is at most $3r$ and $|A^i_{j}| \le 9 f^2 r m^2  $ (with room to spare).

Let $H_1, \dots, H_p$ be an enumeration of all $F$-divisible graphs $H$ such that $V(H) \subseteq V_i$ for some $1 \leq i \le q$. 
Since $|V_i| \le m $%
\COMMENT{we have $m-1 \le |V_i|$, but we do not need this fact.}
 for all $1 \leq i \le q$, for each $i$ there are at most $ 2^{ \binom{m}2 } $ many $H_{j'}$ with $V(H_{j'}) \subseteq V_i$.
Thus $p \le 2^{ \binom{m}2 } q$.
For each $1 \leq j' \le p$, apply Lemma~\ref{lma:abs} to obtain an $F$-absorber $A_{j'}$ for $H_{j'}$ such that the degeneracy of $A_{j'}$ rooted at $V(H_{j'})$ is at most $3r$ and $|A_{j'}| \le 9 f^2 r m^2$.

We now find the $F$-absorbers $A^i_{j}$ and $A_{j'}$ in $G[\P]$ as follows.
The number of $F$-absorbers we need to find is $(q-1) f + p$, and each of these $F$-absorbers has order at most~$b: =  9 f^2 r m^2$.
Let $\P_0 := \{V(G)\}$ be the trivial partition of $V(G)$.
Note that we can view each of the $A^i_{j}$ and $A_{j'}$ as a $\P_0$-labelled graph.
(For example, the $\P_0$-labelled graph $A_j^i$ is such that each $v \in V(Q^i_j \cup \widetilde{Q}^{i+1}_j)$ is labelled $\{v\}$ and every other vertex of $A_j^i$ is labelled $V(G)$.) 
Note that each $v \in V(G)$ is a root for at most $1 + 2^{ \binom{m}2 }$ of the $A^i_{j}$ and~$A_{j'}$.
Since $\delta(G[\P]) \ge (1-1/3r +\eps) n$, we have $d_{G[\P]}(S) \ge \eps n $ for any $S \subseteq V(G)$ with $|S| \le 3r$.
Pick $\eta $ with $1/n \ll \eta \ll 1/m$ and apply Lemma~\ref{lma:finding} with $G[\P], 1, 3r, \epsilon^2, \P_0 , A^1_1 , A^1_2 , \dots, A^{q-1}_f, A_1, \dots, A_p$ playing the roles of $G,k,  d, \epsilon, \P, H_1, \dots, H_{m}$. 
We obtain edge-disjoint  embeddings $\phi(A^1_1)$, $\phi (A^1_2)$, $\dots$, $\phi(A^{q-1}_f)$, $\phi(A_{1})$, $\dots$, $\phi(A_{p})$ of $A^1_1$, $A^1_2$, $\dots$, $A^{q-1}_f$, $A_{1}$, $\dots$, $A_{p}$ into $G[\P]$, which are compatible with their labellings and, moreover,
\begin{align}
\Delta\Big( \bigcup_{i = 1}^{q-1} \bigcup_{j=1}^f \phi(A^i_j) \cup \bigcup_{ j'=1}^p \phi(A_{j'}) \Big) \le \eps^2 n.
\label{eqn:DA}
\end{align}

For each $1 \le i < q$ and each $1 \le j \le f$, let $M^i_j := Q^i_j \cup \widetilde{Q}^{i+1}_j \cup \phi(A^i_j)$.
Using Proposition~\ref{prop:hcf} it is easy to check that $M^i_j$ is a $(V_i,V_{i+1})_F$-edge-mover.
Let $M:=\bigcup_{i=1}^{q-1} \bigcup_{j=1}^f M^i_j$, and let $A^* := M \cup \bigcup_{j'=1}^p \phi ( A_{j'} )$.

We now show that $A^*$ has the desired properties.
Since $A^*$ is an edge-disjoint union of $F$-absorbers and edge-movers, $A^*$ is $F$-divisible.
Note that  $A^*[V_1] = \bigcup_{j=1}^f Q_j^1$,  $A^*[V_{q}] = \bigcup_{j=1}^f \widetilde{Q}^{q}_{j}$ and, for each $1 < i < q$, $A^*[V_i] = \bigcup_{j=1}^f Q_j^i \cup \widetilde{Q}^i_{j} $.
Thus $\Delta(A^*[V_i]) =r$ for each $1 \leq i \le q$.
Moreover, $\Delta( A^*[\P] ) \le \eps^2 n $ by~\eqref{eqn:DA}.

Let $H^*$ be an $F$-divisible graph on $V(G)$ that is edge-disjoint from $A^*$ and has $e(H^*[\P]) = 0 $.
First we show that $H^* \cup M$ can be decomposed into a graph $H'$ and a set~$\mathcal{F}$ of edge-disjoint copies of $F$ such that $e(H'[\mathcal{P}]) = 0 $ and for each $1 \le i \le q$, $H'[V_i]$ is $F$-divisible. 
Recall the definition of $a$ from Proposition~\ref{prop:hcf}.
Proposition~\ref{prop:hcf}(i) applied to $H^*[V_{\le i}]$ tells us that, for each $1 \le i \le q$, we have $e(H^*[V_{\le i}]) \equiv - p_i a \mod{ e(F) }$ for some integer $p_i$ with $0 \le p_i < f$. 
Set $p_0 := 0 $.
For each $1 \le i < q$, add $Q^i_1, \dots, Q^i_{p_i}, \widetilde{Q}^{i+1}_1, \dots,  \widetilde{Q}^{i+1}_{p_i}$ to $H^*$ to obtain $H'$.
Since each $Q^i_j \cup \widetilde{Q}^{i+1}_j$ is $F$-divisible, so is $H'$. 
Also, for each $1 \le i < q$,
\begin{align*}
	e(H'[V_i]) & = e(H^*[V_i]) + \sum_{j=1}^{p_i} e(Q^i_{j}) + \sum_{j'=1}^{p_{i-1}} e(\widetilde{Q}^i_{j'})\\
	& \equiv e(H^*[V_i]) + p_i a - p_{i-1} a  \mod{ e(F)}\\
	& \equiv e(H^*[V_i]) - e(H^*[V_{\le i}]) + e(H^*[V_{\le i-1}]) \equiv 0  \mod{ e(F) }.
\end{align*}
Moreover, since $H^*$ is $F$-divisible, 
\begin{align*}
e(H'[V_{q}])  & = e(H^*[V_{q}]) +  \sum_{ j'=1 }^{p_{q-1}} e(\widetilde{Q}^i_{j'})
\equiv e(H^*[V_{q}]) - p_{q -1} a  \mod{ e(F) } \\
& \equiv e(H^*[V_{q}]) + e(H^*[V_{<q}])
\equiv e(H^*) \equiv 0  \mod{ e(F) }.
\end{align*} 
Therefore $H'[V_i]$ is $F$-divisible for each $1 \leq i\le q$.
Note that $M - H'$ can be decomposed into $\phi(A^i_1), \dots, \phi(A^i_{p_i}), M^i_{p_i+1}, \dots, M^i_{f}$ for each $1 \le i < q$, each of which has an $F$-decomposition.
Hence $H^* \cup M$ can be decomposed into a graph $H'$ and a set $\mathcal{F}$ of edge-disjoint copies of  $F$ such that $e(H'[\mathcal{P}]) = 0 $ and for each $1 \le i \le q$, $H'[V_i]$ is $F$-divisible as claimed.

Since each $H'[V_i]$ is $F$-divisible, there exists a $1 \leq j'_i \le p $ such that $A_{j'_i}$ is an $F$-absorber for $H'[V_i]$.
Note that the indices $j'_{i}$ are distinct for different $1 \leq i \le q$.
Therefore $H' \cup  \bigcup_{j'=1}^p \phi ( A_{j'} )$ has an $F$-decomposition $\mathcal{F}'$, so $H^* \cup A^*$ has an $F$-decomposition $\mathcal{F} \cup \mathcal{F}'$.
This completes the proof of the lemma.
\end{proof}

\subsection{A strengthening of Lemma~\ref{lma:Krabsorber} for certain graphs~$F$} \label{sec:absremarks}

Let $F$ be an $r$-regular graph on $f$ vertices.
Define $d_F$ to be the smallest integer $d$ such that for every pair of vertex-disjoint graphs $H$, $H'$ such that $H$ is $r$-regular and $H'$ can be obtained from a copy of~$H$ by identifying vertices, there exists an $(H,H')_{F}$-transformer~$T$ such that the degeneracy of $T$ rooted at $V(H \cup H')$ is at most~$d$.

With this terminology, Lemma~\ref{lma:regular} has the following immediate corollary.%
\COMMENT{Proof for completeness:
Let $H$ be an $r$-regular graph and let $H'$ be a obtained from $H$ by identifying vertices.
Suppose $H$ and $H'$ are vertex-disjoint. 
The result now easily follows from Lemma~\ref{lma:regular}.}

\begin{corollary} \label{cor:regular}
Let $r,f \in \mathbb{N}$ and let $F$ be an $r$-regular graph on $f$ vertices.
Then $d_F \le 3r$.\qedhere
\end{corollary}

Our argument in Section~\ref{sec:absstruture} actually gives the following lemma, which corresponds to Lemma~\ref{lma:abs}.
We omit its proof since it is virtually identical to the proof of Lemma~\ref{lma:abs} (with $d_F$ in place of $3r$).

\begin{lemma} \label{lma:abs2}
Let $r,f \in \mathbb{N}$ and let $F$ be an $r$-regular graph on $f$ vertices.
Let $H$ be an $F$-divisible graph.
Then there is an $F$-absorber $A$ for $H$ such that the degeneracy of $A$ rooted at $V(H)$ is at most $d_F$.
\end{lemma}

Furthermore, by replacing Lemma~\ref{lma:abs} with Lemma~\ref{lma:abs2} in the proof of  Lemma~\ref{lma:Krabsorber}, we get the following stronger lemma.
Note that we do not have an explicit bound on the number of vertices of $F$-absorbers~$A$ for~$H$ of degeneracy~$d_F$.
However, there is a function $g$ so that $|A| \le g(|H|)$, and such a bound is all we need to apply Lemma~\ref{lma:finding}.%
\COMMENT{
Let $H$ be an $r$-regular graph and let $H'$ be a obtained from $H$ by identifying vertices.
Suppose $H$ and $H'$ are vertex-disjoint. 
By definition of $d_F$, there exists $(H,H')_{F}$-transformer~$T$.
There exists a function $\theta = \theta_{F}$ such that $|T| \le \theta (|H|)$ (as there are finitely many $H$ and $H'$ once $|H|$ is fixed).
So Lemma~\ref{lma:abs} implies every $F$-divisible graph $H$ has an $F$-absorber $A$ with degeneracy $d_F$ and $|A| \le \theta'(|H|)$, for some function non-decreasing $\theta' = \theta'_{F}$.
\newline
\newline
MORE COMMENT: With hindsight (B: !), this argument would have been cleaner then the explicit bounds in the current proof of Lemma~\ref{lma:abs} but we would rather not change things now anymore.}
\COMMENT{Explanation why the proof of Lemma~\ref{lma:Krabsorber} still works.
Fix $F$.
In proof of Lemma~\ref{lma:Krabsorber}, we need to embeds lots of absorber into~$G$. 
We only ever consider $F$-divisible graphs $H$ of size at most $m$, so all our absorbers have bounded order $\theta'(m)$.
So finding lemma, Lemma~\ref{lma:finding} works. 
\newline
Typically speaking, the hierarchies need to be depending on~$F$, i.e. $1/n \ll_F  1 / m \ll_F   1/r, 1/f,\eps$.
But we can take the hierarchies over all $r$-regular graphs $F$ on $f$ vertices.
}

\begin{lemma} \label{lma:Krabsorber2}
Suppose that $n , m , r,f \in \N$ and $\eps > 0 $ with $1/n \ll  1 / m \ll   1/r, 1/f,\eps$.
Suppose that $F$ is an $r$-regular graph on $f$ vertices. 
Let $\delta := 1 - \min \{1/r,1/d_F \} + \eps$, and let $q := \lceil n / m\rceil$.
Let $G$ be a graph on $n$ vertices.
Let $\P = \{ V_1, \dots, V_{q} \}$ be an equitable partition of $V(G)$ such that, for each $1 \leq i \leq q$, $|V_i| = m \text{ or } m-1$.
Suppose that $\delta( G[\P] ) \ge \delta n $ and $\delta(G[V_i]) \ge  \delta |V_i|$ for each $1 \leq i \le q $.
Then $G$ contains an $F$-divisible subgraph~$A^*$ such that
\begin{enumerate}[label={\rm(\roman*)}]
\item $\Delta(A^*[\P]) \le \eps^2 n$ and  $\Delta( A^*[V_i] ) \le  r$ for each $1 \leq i \le q$, and 
\item if $H^*$ is an $F$-divisible graph on $V(G)$ that is edge-disjoint from $A^*$ and has $e ( H^* [ \P ] ) = 0 $, then $A^* \cup H^*$ has an $F$-decomposition.\qedhere
\end{enumerate}
\end{lemma}


\section{Parity graphs} \label{sec:parity-graphs}

Let $F$ be an $r$-regular graph, let $x$ be a vertex of $F$, and let $F_x := F[N_F(x)]$.
Let $G$ be an $F$-divisible graph with a $(k,\delta)$-partition $\P = \{V_1, \dots, V_k\}$, and suppose that $G[\P]$ is sparse.
Our aim is to use a small number of edges from $G - G[\P]$ to cover all edges of $G[\P]$ by copies of $F$.
We will do this by, for each $1 \leq i < j \leq k$ and each $v \in V_i$, finding an $F_x$-factor in $N_G(v,V_j)$.
We will then extend each copy of $F_x$ to a copy of $F - x$ using Lemma~\ref{lma:finding}.
Together with the edges incident to $v$, these copies of $F-x$ will form copies of $F$.
An obvious necessary condition for this to work is that each $d_G(v,V_j)$ is divisible by $r$. 
In this section we show that we can find certain structures, which we call parity graphs, that can be used to ensure that this divisibility condition holds.

Let $U$ and $V$ be disjoint vertex sets and let $x, y \in U$.
Let $F$ be an $r$-regular graph.
An \defn{$xy$-shifter with parameters $U, V, F$} is a graph $S$ with $V(S) \subseteq U \cup V$ such that $xy \notin E(S)$ and 
\begin{enumerate}[label=(\roman*)]
  \item $d_S(x,V) \equiv -1 \mod r$, $d_S(y,V) \equiv 1 \mod r$ and, for all $u \in U\setminus \{x,y\}$, $d_S(u, V) \equiv 0 \mod r$;
	\item $S$ has an $F$-decomposition.
\end{enumerate}
Condition (i) allows us to move excess degree (mod $r$) from $x$ to~$y$.%
\COMMENT{Condition (ii) ensures that, if $G-S$ is $F$-decomposable, then $G$ is also $F$-decomposable.}

Let $uv \in E(F)$. 
For a graph $H$ and an edge $xy \in E(H)$, $H$ \emph{with a copy of $F$ glued along $xy$ via $uv$} is a graph obtained from $H$ by adding a copy $F'$ of $F$ that is vertex-disjoint from $H$ and identifying $u$ with $x$ and $v$ with $y$.

\begin{proposition} \label{prop:construct-shifter}
Let $r,f \in \mathbb{N}$ and let $F$ be an $r$-regular graph on $f$ vertices.
Let $U$ and $V$ be disjoint vertex sets with $|U| \ge r+2$ and $|V| \geq \binom {r+1} 2 (f-2)$, and let $x, y \in U$. 
Then there exists an $xy$-shifter~$S$ with parameters $U, V, F$ with $r+2$ vertices in $U$, $\binom {r+1} 2 (f-2)$ vertices in $V$ and degeneracy at most $r$ rooted at $\{x,y\}$. 
\end{proposition}

\begin{proof}
Pick $r$ distinct vertices $u_1, \dots, u_r$ in~$U\setminus \{x, y\}$. 
We first define a subgraph $S_0$ of $S$ on vertex set $\{x, y, u_1, \ldots, u_r\} \subseteq U$.
Join $x$ to $u_1$, join $y$ to $u_2, \ldots, u_r$ and join $u_1, \ldots, u_r$ completely.
(So if $x$ and $y$ were identified we would obtain a copy of~$K_{r+1}$.)
Thus $d_{S_0}(x) = 1$, $d_{S_0}(y) = r-1$, and $d_{S_0}(u_j) = r$ for $1 \leq j \leq r$.  

Let $uv \in E(F)$.
Let $S$ be the graph obtained from $S_0$ by gluing a copy of $F$ along each edge of $S_0$ via $uv$ such that 
$V(F) \setminus \{u,v\} \subseteq V$ (and these sets are disjoint for different copies).
Then $S$ has an $F$-decomposition, $d_S(x,V) = r-1$, $d_S(y,V) = (r-1)^2$ and $d_S(u_j,V) = r(r-1)$ for each $1 \leq j \leq r$.
Ordering $V(S)$ such that $x$ and $y$ are the first two vertices, and all other vertices in $S_0$ precede those in $S \setminus S_0$, shows that the degeneracy of $S$ is at most $r$. 
\end{proof}

Let $\P = \{V_1, \ldots, V_k\}$ be an equitable partition of a vertex set~$V$.  
An \defn{$F$-parity graph with respect to $\P$} is an $F$-decomposable graph $P$ on $V$ such that, for every $r$-divisible graph $G$ on $V$ that is edge-disjoint from $P$, there is a subgraph $P'$ of $P$ such that
\begin{enumerate}
	\item[(P1)] for each $2 \leq i \leq k$ and each $x \in V_{<i}$, $r$ divides $d_{G\cup P'}(x,V_i)$;
	\item[(P2)] $P - P'$ has an $F$-decomposition.
\end{enumerate}
Next we show that $F$-parity graphs exist.

\begin{proposition} \label{prop:parity-graph}
Let $r,f,k \in \mathbb{N}$ and let $F$ be an $r$-regular graph on $f$ vertices.
Let $\P = \{V_1, \ldots, V_k\}$ be an equitable partition of a vertex set~$V$.  
Let $P_2, \ldots, P_k$ be edge-disjoint graphs on $V$ such that, for each $2 \le i \le k$,
\begin{itemize}
	\item $P_i$ is the edge-disjoint union of $E_i$ and $D_i$;
	\item $E_i$ is the edge-disjoint union of $r-1$ copies of $F$, each with $2$ adjacent vertices in $V_i$ and $f-2$ vertices in $V_{i-1}$;
	\item $D_i$ is the edge-disjoint union of $r-1$ $u_ju_{j+1}$-shifters with parameters $V_{<i}, V_i, F$ for each $1 \leq j < |V_{<i}|$, where $u_1, \ldots, u_{|V_{<i}|}$ is an enumeration of $V_{< i}$.
\end{itemize}
Then $P := P_2 \cup \cdots \cup P_k$ is an $F$-parity graph with respect to~$\P$.
\end{proposition}

\begin{proof}
The proof is by induction on~$k$.
If $k=1$, then there is nothing to prove, so assume that $k \geq 2$.
Since each $D_i$ has an $F$-decomposition, so does~$P$.

Let $G$ be an $r$-divisible graph on $V$ that is edge-disjoint from~$P$. 
First we show that there is a subgraph $P'_k$ of~$P$ such that 
\begin{enumerate}[label=(\roman*)]
	\item for each $x \in V_{<k}$, $r$ divides $d_{G\cup P'_k}(x,V_k)$;
	\item $P_k - P'_k$ has an $F$-decomposition.
\end{enumerate}

Suppose that $e_{G}(V_k) \equiv t \mod r$, where $0 < t \leq r$.  
Form a graph~$G_0$ from $G$ by adding $r-t$ of the copies of $F$ from $E_k$ to~$G$.  
Then $2e_{G_0}(V_k) \equiv 0 \mod r$, so
\begin{align}
\sum_{v \in V_{<k}} d_{G_0}(v,V_k) 
& = e_{G_0}(V_{<k}, V_k)
= \sum_{v \in V_k} d_{G_0}(v) - 2e_{G_0}(V_k) 
\equiv 0 \mod r. 
 \label{eqn:G_0}
\end{align}

Let $\ell : = |V_{< k}|$ and let $u_1, \ldots, u_{\ell}$ be the enumeration of $V_{<k}$ used in the definition of $D_k$.  

Let $0 \le t_1 < r$ be such that $d_{G_0}(u_1,V_k) \equiv t_1 \mod r$.  
Add $t_1$ of the $u_1u_2$-shifters in $D_k$ to $G_0$ to obtain $G_{1}$ in which $d_{G_1}(u_1, V_k) \equiv 0 \mod r$ and $d_{G_{1}}(u_i, V_k) \equiv d_{G_{0}}(u_i, V_k) \mod r$ for all $3 \le i \le \ell$. 

Let $0 \le t_2 < r$ be such that $d_{G_1}(u_2,V_k) \equiv t_2 \mod r$.  
Add $t_2$ of the $u_2u_3$-shifters in $D_k$ to $G_{1}$ to obtain $G_{2}$ in which $d_{G_{2}}(u_1, V_k) \equiv d_{G_{2}}(u_2, V_k) \equiv 0 \mod r$ and $d_{G_{2}}(u_i, V_k) \equiv d_{G_{1}}(u_i, V_k) \equiv d_{G_{0}}(u_i, V_k) \mod r$ for all $4 \le i \le \ell$.

Continuing in this way, we eventually obtain $G_{\ell - 1}$ in which $d_{G_{\ell - 1}}(u_i,V_k) \equiv 0 \mod r$ for each $1 \leq i \leq \ell -1$.
Note that 
\begin{align*}
d_{G_{\ell - 1}}(u_{\ell},V_k) & \equiv d_{G_{\ell - 2}}(u_{\ell}, V_k) + d_{G_{\ell - 2}}(u_{\ell-1}, V_k)   \mod r \\
	& \equiv d_{G_0}(u_{\ell}, V_k) + d_{G_{\ell - 3}}(u_{\ell-1}, V_k) + d_{G_{\ell - 3}}(u_{\ell-2}, V_k)  \mod r \\
	& \equiv \sum_{v \in V_{<k}} d_{G_0}(v,V_k) \equiv 0 \mod r,
\end{align*}
where the last equality holds by~\eqref{eqn:G_0}.
Let $P_k' := G_{\ell - 1} - G$; then (i) holds.
Observe also that $P_k - P'_k$ consists of some copies of $F$ from $E_k$ and some shifters from $D_k$, each of which has an $F$-decomposition, so (ii) holds.

Let $G^* : = ( G \cup P'_k) [V_{\le k-1}]$, $P^* : = P_2 \cup \cdots \cup P_{k-1}$ and $\mathcal{P}^* : = \{V_1, \dots, V_{k-1} \}$.
Note that $G^*$ and $P^*$ are edge-disjoint. 
Recall that $G$, $P_k$ and $P_k - P'_k$ are $r$-divisible.
So $G \cup P'_k$ is $r$-divisible. 
Thus (i) implies that $G^*$ is also $r$-divisible.
By the induction hypothesis, $P^*$ is an $F$-parity graph with respect to~$\P^*$.
Therefore, there exists a subgraph $P_0$ of~$P^*$ such that for each $2 \leq i \leq k-1$ and each $x \in V_{<i}$, $r$ divides $d_{G^* \cup P_0}(x,V_i)$ and $P^* - P_0$ has an $F$-decomposition.
Let $P' : = P_0 \cup P'_k$.
Then $P'$ satisfies~(P2).
Note that $(G \cup P')[ V_{<k},V_k ] = (G \cup P'_k)[ V_{<k},V_k ]$ and $(G \cup P')[ V_{<i},V_i ] = (G^* \cup P_0)[ V_{<i},V_i ]$ for all $1 \le i <k$.
Thus $P'$ satisfies~(P1).
Therefore $P$ is an $F$-parity graph with respect to~$\P$. 
\end{proof}

The next lemma finds an $F$-parity graph $P$ as in Proposition~\ref{prop:parity-graph} within a dense graph~$G$ using Lemma~\ref{lma:finding}.

\begin{lemma} \label{lma:find-parity-graph}
Let $r,f \in \mathbb{N}$ and let $F$ be an $r$-regular graph on $f$ vertices.
Let $\gamma >0$.
Then there exists an $n_0 = n_0(k, \gamma , F)$ such that the following holds.
Let $G$ be a graph on $n \geq n_0$ vertices and let $\P = \{V_1, \ldots, V_k\}$ be a $(k, \delta)$-partition for $G$ with $\delta \geq 1-1/r+\gamma$.  
Then $G$ contains an $F$-parity graph $P$ with respect to $\P$ such that $\Delta(P) \leq \gamma n$.
\end{lemma}

\begin{proof}
It is enough to show that we can embed a graph $P$ as described in Proposition~\ref{prop:parity-graph} into $G$ in such a way that the maximum degree of the image of the embedding is not too large.
We will assign labels to the graphs making up $P$ and then check that the conditions of Lemma~\ref{lma:finding} hold.

For each $2 \le i \le k$ and each $1 \le j \le r-1$, let $F_{i,j}'$ be a $\P$-labelled copy of $F$ with $2$ adjacent vertices labelled $V_i$ and $f-2$ vertices labelled $V_{i-1}$.

For each $2 \le i \le k$, let $n_{< i} : = |V_{< i}|$ and let $u^i_1, \dots,  u^i_{n_{<i}}$ be an enumeration of the vertices of $V_{<i}$.
For each $2 \le i \le k$ and each $1 \le j < n_{<i}$, apply Proposition~\ref{prop:construct-shifter} to obtain a $u^i_j u^i_{j+1}$-shifter $S_{i,j}$ with parameters $V_{<i}, V_i, F$ such that $|S_{i,j}| = r+2 +\binom{r+1}2 (f-2)$ and $S_{i,j}$ has degeneracy at most~$r$ rooted at $\{ u^i_j,u^{i}_{j+1} \}$.
We may view $S_{i,j}$ as a $\P$-labelled graph by giving $u^i_j$ the label $\{u^i_j\}$, giving $u^i_{j+1}$ the label $\{u^i_{j+1}\}$, giving $u$ the label $V_i$ for all $u \in V(S_{i,j}) \cap V_i$ and giving $u'$ the label $V_{i-1}$ for all $u' \in ( V(S_{i,j}) \cap V_{<i} ) \setminus \{u^i_j, u^i_{j+1}\}$.
Let $S'_{i,j,1}, \dots, S'_{i,j,r-1}$ be $r-1$ copies of $S_{i,j}$, and let
\begin{align*}
\mathcal{F}: = \bigcup_{i = 2}^k \bigcup_{\ell = 1 }^{r-1} \Big( \{F_{i,\ell}'\} \cup \bigcup_{j = 1}^{ n_{<i}-1} \{ S'_{i,j,\ell} \} \Big) .
\end{align*}
So $\mathcal{F}$ is a family of $\P$-labelled graphs and $|\mathcal{F}| \le k r n$.
For each $F' \in \mathcal{F}$, $|F'| \le r+2 +\binom{r+1}2 (f-2)$ and $F'$ has degeneracy at most $r$.
Furthermore, each $v \in V(G)$ is a root vertex for at most $2 r k$ members of $\mathcal{F}$.
Since $\P$ is a $(k, \delta)$-partition for $G$ with $\delta \geq 1-1/r+\gamma$, we have that $d_G(S,V_i) \ge \gamma |V_i|$ for each $S \subseteq V(G)$ with $|S| \le r$ and each $1 \le i \le k$.
Therefore we can apply Lemma~\ref{lma:finding} to find edge-disjoint embeddings $\phi(F')$ for all $F' \in \mathcal{F}$ in~$G$ in such a way that $\Delta( \bigcup_{F' \in  \mathcal{F}} \phi( F ')) \le \gamma n$.
Take $P : = \bigcup_{F' \in  \mathcal{F}} \phi(F')$.
By Proposition~\ref{prop:parity-graph}, $P$ is an $F$-parity graph with respect to~$\mathcal{P}$.
\end{proof}


\section{Near optimal decompositions} \label{sec:partial-decomposition}

Let $G$ be a dense graph as defined in Theorem~\ref{thm:regular}, and let $\P_1, \dots, \P_{\ell}$ be a $(k, \delta + \epsilon, m)$-partition sequence for $G$.
In Section~\ref{sec:absorbers}, we constructed a graph $A^*$ that can `absorb' any $F$-divisible graph~$H^*$ satisfying $e(H^*[\P_\ell]) = 0$.
Our aim in this section is to show that we can indeed decompose $G$ into edge-disjoint copies of $F$ and such a remainder~$H^*$.
More precisely, in this section, we prove the following lemma, which guarantees the existence of such a `near optimal' $F$-decomposition (in particular note that, as $m$ is bounded, $e(H^*)$ is at most linear in~$n$).

\begin{lemma} \label{lma:do-the-iteration2}
Let $r,f,m,k,\ell \in \mathbb{N}$ and let $\eps, \eta> 0$ with $1/m \ll \eta \ll 1/k \ll \epsilon, 1/r,1/f$.%
\COMMENT{Here, $k^{\ell}(m-1)  < n \le k^{\ell}m$.}
Let $F$ be an $r$-regular graph on $f$ vertices and let $G$ be an $r$-divisible graph.
Let $\delta: = \max \{ \delta^{\eta}_F, 1- 1/(r+1) \}$.
Suppose that $\P_1, \ldots, \P_\ell$ is a $(k, \delta + \epsilon, m)$-partition sequence for $G$.
Then there exists a subgraph $H^*$ of $\bigcup_{V \in \P_\ell} G[V]$ such that $G-H^*$ has an $F$-decomposition.
In particular, if $G$ is $F$-divisible, then so is $H^*$.
\end{lemma}

Recall that the definition of $\delta^{\eta}_F$ implies that $G$ contains an $\eta$-approximate $F$-decomposition.
We would like the remainder $H^*$ in Lemma~\ref{lma:do-the-iteration2} to contain no edges of $G[\P_\ell]$, but the definition of $\delta^{\eta}_F$ does not guarantee this.
The key idea of the proof of Lemma~\ref{lma:do-the-iteration2} is to proceed via an iterative process, which repeatedly invokes the definition of~$\delta^{\eta}_F$.
More precisely, suppose that we are able to prove the following result:
\begin{itemize}
	\item[ ($\dagger$) ] If $\P$ is a $(k, \delta)$-partition for a graph $G$, then $G[\P]$ can be covered by edge-disjoint copies of $F$ in $G$ which use only a small number of edges from $G - G[\P]$.
\end{itemize}
Suppose that we apply ($\dagger$) with $\P = \P_1$.
We are then left with edges in $G - G[\P_1] = \bigcup_{V \in \P_1} G[V]$.
But since $\P_1, \ldots, \P_\ell$ is a $(k, \delta + \epsilon, m)$ partition sequence and we have used very few edges of $G - G[\P_1]$, we have for each $V \in \P_1$ that $\P_2[V]$ is a $(k, \delta)$-partition of the remaining part of $G[V]$.
So we can apply ($\dagger$) to each part to cover the remaining edges of $G[\P_2]$ by edge-disjoint copies of $F$, using only a few edges from $G - G[\P_2]$.
Continuing in this way, we eventually obtain edge-disjoint copies of $F$ covering all edges of $G - G[\P_{\ell}]$, which implies Lemma~\ref{lma:do-the-iteration2}.
(To avoid our bound on the minimum degree deteriorating in each step, we actually prove a stronger version of~($\dagger$) which gives us more control on the edges we use from $G- G[\mathcal{P}]$.)

The rest of this section is divided into three subsections.  
In Section~\ref{sec:boundmaxdeg}, we show that we can find an approximate $F$-decomposition of $G[\P]$ such that the remainder has low maximum degree (at the cost of using a small number of additional edges from $G - G[\P]$).\COMMENT{14/7 This cannot be taken literally: an $F$-decomposition of a graph cannot use edges that are not in that graph.  Similar comments apply to the chat at the start of some of the sections.}
In Section~\ref{sec:sparse-remainder-cover} we show how such a remainder of low maximum degree can be covered by copies of~$F$.
In Section~\ref{sec:iteration} we give a formal statement of ($\dagger$) and perform the iteration described above.

\subsection{Bounding the maximum degree of the remainder graph} \label{sec:boundmaxdeg}
Consider an $\eta$-approximate $F$-decomposition~$\mathcal{F}$ of $G[\mathcal{P}]$ guaranteed by~the definition of~$\delta^{\eta}_F$.
Let $H$ be the remainder of $G[\P]$ (after removing all the edges of $\mathcal{F}$), and suppose that $d_H(x,V)$ is large for some $V \in \P$ and some $x \in V(G) \setminus V$.
Note that $x$ together with a copy of $K_{r}$ that lies in $N_H(x,V)$ forms a copy of $K_{r+1}$.
Using some additional vertices and edges inside~$V$, we can then extend a spanning subgraph of this copy of $K_{r+1}$ to a copy of~$F$.
So we can reduce $d_H(x,V)$ by finding vertex-disjoint copies of $K_{r}$ lying entirely in $N_H(x,V)$, which are then extended into copies of $F$.
This is formalised in Lemma~\ref{lma:partial-decomposition}.
To find the above copies of $K_r$ we shall use the Hajnal--Szemer\'edi theorem~\cite{HajnalSzemeredi}.

\begin{theorem}[\cite{HajnalSzemeredi}] \label{thm:HSz}
Let $r \in \N$ with $r \ge 2$.
Every graph $G$ on $n$ vertices with $\delta(G) \ge (1-1/r) n$ contains $\lfloor n / r \rfloor$ vertex-disjoint copies of $K_r$.
\end{theorem}

\begin{lemma} \label{lma:findcliques}
Let $r,k,n \in \N$ and let $\gamma > 0$ with $1/n \ll \gamma, 1/k, 1/r$. 
Let $H$ be a graph on $n$ vertices.
Let $U,V \subseteq V(H)$ be disjoint with $|V| \ge \lfloor n/k \rfloor$.
Suppose that, for each $x \in U$ and each $y \in V$,
\begin{enumerate}[label=\rm(\roman*)]
	\item $r$ divides $d_H(x,V)$; \label{divides-properly}
	\item $\delta(H[N_H(x,V)]) \geq ( 1 - 1/r ) d_H(x,V) + \gamma |V|$; \label{high-minimum-degree-in-neighbourhood}
	\item $d_H( y , U ) \leq \gamma  |V|/r$. \label{each-not-used-too-much}
\end{enumerate}
Then there is a subgraph $H_V$ of $H[V]$ such that $H[ U,V ] \cup H_V$ has a $K_{r+1}$-decomposition and $\Delta(H_V) \leq \gamma |V|$.
\end{lemma}

\begin{proof}
For each $x \in U$ in turn we will choose a $K_{r}$-factor from the unused part of $H[N_H(x,V)]$ and take $H_V$ to be the union of these edge-disjoint $K_{r}$-factors.

We claim that we can choose these $K_{r}$-factors greedily.  
Indeed, suppose we seek a $K_r$-factor for $x \in U$.
Consider any vertex $y \in N_H(x,V)$.
By~\ref{each-not-used-too-much}, at most $r d_H( y, U ) \leq  \gamma |V|$ of the edges at $y$ in $H[N_H(x,V)]$ have been used already.
So by~\ref{divides-properly}, \ref{high-minimum-degree-in-neighbourhood} and Theorem~\ref{thm:HSz} there exists a $K_{r}$-factor in the unused part of $H[N_H(x,V)]$.

Since at most $r d_H( y, U ) \leq \gamma |V|$ edges are used at each $y \in V$, we have that $\Delta(H_V) \leq \gamma |V|$.
\end{proof}

\begin{lemma} \label{lma:extendtoF}
Let $r,f,k,n \in \N$ and let $\eta, \gamma > 0$ with $1/n \ll \eta \ll \gamma, 1/k, 1/r,1/f$.%
\COMMENT{B: This is where we need to apply the finding lemma to quadratically many graphs.}
Let $F$ be an $r$-regular graph on $f$ vertices and let $H$ be a graph on $n$ vertices.
Let $U,V \subseteq V(H)$ be disjoint with $|V| \ge \lfloor n/k \rfloor$. 
Suppose that, for each $x \in U$ and each $y \in V$,
\begin{enumerate}[label=\rm(\roman*)]
	\item $r$ divides $d_H(x,V)$;
	\item $\delta(H[N_H(x,V)]) \geq ( 1 - 1/r ) d_H(x,V) + \gamma |V|$;
	\item $d_H( y , U ) \leq \eta |V|$;
	\item $\delta(H[V]) \ge (1-1/r + 2 \gamma) |V|$.
\end{enumerate}
Then there is a subgraph $H_V'$ of $H[V]$ such that $H[ U,V ] \cup H_V'$ has an $F$-decomposition and $\Delta(H_V') \leq 2 \gamma |V|$.
\end{lemma}

\begin{proof}
By Lemma~\ref{lma:findcliques}, there is a subgraph $H_V$ of $H[V]$ such that $H[ U,V ] \cup H_V$ has a $K_{r+1}$-decomposition and $\Delta(H_V) \leq  \gamma |V|$.
Choose such an $H_V$ with as few edges as possible, and
let $W_1, \dots, W_p$ be an enumeration of a $K_{r+1}$-decomposition of $H[ U,V ] \cup H_V$.
By the minimality of $H_V$, each $W_j$ has vertex set $\{ w_j \} \cup W'_j$ with $w_j \in U$ and $W'_j \subseteq V$.
Note that
\begin{align}
p \le \sum_{y \in V} d_H( y , U ) \le \eta |V|^2. \nonumber
\end{align}
Let $H' : = H[ V ]  - H_V$; then $|H'| = |V|$ and $\delta( H' ) \ge (1-1/r +  \gamma) |V|$ by~(iv).
Let $u \in V(F)$ and let $F^* : = F \setminus \{u\} - F [ N_F(u)]$.
Note that $F^*$ trivially has degeneracy at most $r$ rooted at $N_F(u)$.
Let $F^*_1, \dots, F^*_p$ be copies of $F^*$.
We now embed $F^*_1, \dots, F^*_p$ into $H'$ in such a way that, for each $F^*_j$, the image of $N_F(u)$ is precisely $W'_j$ as follows. 
Let $\P_0 := \{ V \}$ be the trivial partition of $V$.
We view each $F^*_j$ as a $\P_0$-labelled graph such that the root vertices of $F^*_j$ are precisely $N_F(u)$, and the union of their labels is $W'_j$; each other vertex of $F^*$ is labelled $V$.
By (iii), there are at most $d_H(y,U) \le \eta |V|$ indices~$j$ with $1 \le j \le p$ such that some vertex of $F^*_j$ is labelled~$\{y\}$.
Since $\delta( H' ) \ge (1-1/r +  \gamma) |V|$, we have that $d_{H'}(S) \ge \gamma |V|$ for each $S \subseteq V$ with $|S| \le r$. 
So by Lemma~\ref{lma:finding}, with $H'$, $1$, $r$, $f$, $\gamma$, $\P_0$, $F^*_1, \dots, F^*_p$ playing the roles of $G$, $k$, $d$, $b$, $\eps$, $\P$, $H_1, \dots, H_m$, there exist edge-disjoint embeddings $\phi(F^*_1)$, $\dots$, $\phi(F^*_p)$ of $F^*_1$, \dots, $F^*_p$ into $H'$ which are compatible with their labelling such that $\Delta ( \bigcup_{j=1}^p \phi(F^*_j) )  \le \gamma |V|$.

Each $W_j \cup \phi(F^*_j)$ contains a copy $F_j$ of $F$ such that $H[U,V] \subseteq \bigcup_{j=1}^p F_j$.
Let $H'_V : = \bigcup F_j [V]$.
Note that $H[ U,V ] \cup H_V'$ has an $F$-decomposition and $\Delta(H_V') \leq \Delta(H_V) + \Delta ( \bigcup_{j=1}^p \phi(F^*_j)) \le  2 \gamma |V|$.
\end{proof}

\begin{proposition} \label{prop:neighbourhood-degree}
Let $r,k \in \N$ and let $\eps \ge 0$.
Let $G$ be a graph and let $\P$ be a $(k, 1- 1/(r+1) + \eps)$-partition for $G$.  
Let $x \in V(G)$ and let $V \in \P$.
Then
 \begin{align*}
\delta(G[N_G(x,V)]) \geq (1- 1/r ) d_G(x,V) + \eps |V|.
\end{align*}
\end{proposition}

\begin{proof}
Let $y \in N_G(x,V)$.  Since $d_G(y,V), d_G(x,V) \geq  (1- 1/(r+1) + \eps)|V|$,\COMMENT{B: For the third inequality, note that $d \geq \frac r {r+1} |V|$, so $d/r \geq |V| / (r+1)$.}
\begin{align*}
d_G(y,N_G(x,V)) & \geq d_G(x,V) + d_G(y,V) - |V| 
                 \ge  d_G(x,V) - \frac{|V|}{r+1} +\eps |V|  \\
				& \geq (1- 1/r ) d_G(x,V) + \eps |V|.  \qedhere
\end{align*}
\end{proof}

\begin{lemma} \label{lma:partial-decomposition}
Let $r,f,k,n \in \mathbb{N}$ and let $\gamma, \eta,\eps >0$ with $1/n \ll \eta \ll \gamma\ll 1/k \ll  \eps, 1/r,1/f$.%
\COMMENT{B: $f$ doesn't appear in the proof any more, but we still need $n$ to be much larger than it to apply the previous lemmas.}
Let $F$ be an $r$-regular graph on $f$ vertices and let $G$ be a graph on $n$ vertices. 
Let $\delta: = \max \{ \delta^{\eta}_F, 1- 1/(r+1)\}$.
Suppose that $\P = \{V_1, \ldots, V_k\}$ is a $(k, \delta + \eps)$-partition for~$G$.
Then there is a subgraph $H$ of $G$ such that 
\begin{enumerate}[label={\rm(\alph*)}]
	\item $G-H$ has an $F$-decomposition;
	\item $\Delta( H [ \P ] ) \leq \gamma n $.
	\item for each $1 \leq i \le k$, $\Delta ( G[V_i] - H[V_i] ) \le 2 \gamma |V_i| $.
\end{enumerate}
\end{lemma}

\begin{proof}
Let $0 < q < 1$ with $\eta \ll q \ll \gamma$,
and let $G'$ be a subgraph of $G[\P]$ such that
\begin{itemize}
	\item[(G1)] $\Delta(G') \leq 2qn$;
	\item[(G2)] for every $S \subseteq V(G)$ with $|S| \leq r$, $d_{G'}(S,V(G)) \geq q^r \eps n/2$.
\end{itemize}
(To see that such a subgraph $G'$ exists, first note that since $\P$ is a $(k, 1-1/r + \eps)$-partition,
for each $S \subseteq V(G)$ with $|S| \leq r$, we have that $d_{G[\P]}(S,V(G)) \geq \eps n$.
Consider a random subgraph of $G[\P]$ in which each edge is retained independently with probability $q$; then (G1) and (G2) are satisfied with high probability.)

Note that $\delta(G[\P] - G') \geq (\delta+\eps)(n - \lceil n/k \rceil) - \Delta(G')\geq \delta^\eta_F n$,
so by the definition of $\delta^{\eta}_F$, there exists an $\eta$-approximate $F$-decomposition $\mathcal F_0$ of~$G[\P]-G'$.  
Let $G_0$ be the subgraph of $G[\P]-G'$ which consists of the uncovered edges; so $e(G_0) \le \eta n^2$.
To satisfy (b) our next aim is to cover the edges of $G_0$ incident to vertices of high degree in $G_0$  by copies of $F$.
But we know very little about the neighbourhoods of the high degree vertices, so we cannot achieve this directly.
Instead our first step will be to transform the approximate $F$-decomposition $\mathcal F_0$ into an approximate $F$-decomposition $\mathcal F_1$ such that $\bigcup \mathcal F_1$ contains no edge of $G$ incident to a vertex of high degree in $G_0$.

Let $B : = \{v \in V(G): d_{G_0}(v) > \eta^{1/2} n\}$ and let $A : = V(G) \setminus B$; observe that $|B| \leq 2\eta^{1/2} n$.
Let $\mathcal F' := \{F \in \mathcal{F}_0: V(F) \cap B \neq \emptyset\}$,
and enumerate the elements of $\mathcal F'$  as $F^{(1)}, \ldots, F^{(m)}$.
For each $1 \leq i \leq m$, let $F^{(i)}_0 := F^{(i)} - B$, let $R_i := \{v \in V(F^{(i)}_0) : d_{F^{(i)}}(v, B) \geq 1\}$ and
let $F^{(i)}_1 := F^{(i)}[R_i,B] \cup F^{(i)}[B]$. Note that $F^{(i)}_0$ and $F^{(i)}_1$ form a decomposition of $F^{(i)}$.
We consider each $F^{(i)}_1$ to be rooted at $R_i$ and label the non-root vertices $\{A\}$.
We will replace each $F^{(i)}$ by a copy of $F$ in $G' \cup F_0^{(i)}$ that contains $F^{(i)}_0$ but contains no vertex of $B$.
Note that each $v \in A$ is in at most $|B| \leq 2\eta^{1/2} n$ of the $R_i$.

Now let $G'' := G'[A]$ and let $n'' := |G''|$.
Note that $R_i \subseteq A$ for each $1 \leq i \leq m$ and, for each $S \subseteq A$ with $|S| \leq r$, 
\begin{align*}
d_{G''}(S,V(G'')) \geq d_{G'}(S,V(G)) - |B| 
\overset{{\rm (G2)}}{\geq} q^r \eps n/2 - 2\eta^{1/2} n \geq q^{r+1} n''.
\end{align*}
Then by Lemma~\ref{lma:finding}, with $G''$, $1$, $r$, $f$, $2\eta^{1/2}$, $q^{r+1}$, $\{A\}$, $F^{(1)}_1, \ldots, F^{(m)}_1$ playing the roles of $G$, $k$, $d$, $b$, $\eta$, $\eps$, $\P$, $H_1, \dots, H_m$, there exist edge-disjoint embeddings $\phi(F^{(1)}_1)$, $\dots$, $\phi(F^{(m)}_1)$ of $F^{(1)}_1, \ldots, F^{(m)}_1$ into $G''$ which are compatible with their labellings.
Let $\mathcal F_1 := \{\phi(F^{(i)}_1) \cup F^{(i)}_0 : 1 \leq i \leq m\} \cup (\mathcal F_0 \setminus \mathcal F')$.
Then $\mathcal F_1$ is a collection of edge-disjoint copies of $F$ with $|\mathcal F_1| = |\mathcal F_0|$ and no edge of $\bigcup \mathcal F_1$ is incident to $B$.
Let $H':= G[\P] - \bigcup \mathcal F_1$.
Then $\mathcal{F}_1$ is an $F$-decomposition of $G[\P]-H'$, and 
\begin{align}
N_{H'}(v) = N_{G[\P]}(v) \text{ for all $v \in B$}. \label{eqn:partial1}
\end{align}
Moreover,
\begin{align}
d_{H'}(v) \leq d_{G_0}(v) + d_{G'}(v) 
\overset{{\rm (G1)}}{\leq} \eta^{1/2} n + 2qn \leq 3q n \label{eqn:partial2}
\end{align}
for all $v \in A$.

We now find a set $\mathcal{F}_2$ of edge-disjoint copies of~$F$ that cover most of the edges incident to $B$ in $H'$.
To do this we will use some edges of $G - G[\P]$.

For each $1 \leq i \le k$, let $B_i : = B \setminus V_i$ and let $V_i' : = V_i \setminus B$.
Let $H_i^*$ be the graph on vertex set $V(G)$ with $E(H_i^*) : = E( H' [B_i, V_i'] ) \cup  E ( G [V_{i}'] )$.
Note that the $H_i^*$ are edge-disjoint.
By removing at most $r-1$ edges incident to each $v \in B_i$ from $H_i^*$, we obtain a spanning subgraph $H'_i$ of $H_i^*$
which has the property that $r$ divides $d_{H_i'} (v, V'_i)$ for all $v \in B_i$.

We aim to apply Lemma~\ref{lma:extendtoF} to each $H_i'$ with $B_i, V'_i, 4qk$ playing the roles of $U, V, \eta$. 
We now check that conditions (i)--(iv) of Lemma~\ref{lma:extendtoF} hold for $H_i'$. 

Condition (i) holds by our construction.
Note that for all $v \in B_i$,  \eqref{eqn:partial1} implies that $d_G(v, V_i) \le d_{H'_i}(v, V_i') + |B| + r-1$
(recall that we deleted at most additional $r-1$ edges at~$v$ to obtain $H'_i$ from $H^*_i$). 
Recall that $\P$ is a $(k,1-1/(r+1)+ \eps )$-partition for~$G$.
By Proposition~\ref{prop:neighbourhood-degree}, for all $v \in B_i$ we have that
\begin{align*}
\delta(   H'_i  [ N_{H'_i}(v,V_i')]  ) 
& =  \delta(   G  [ N_{H'_i}(v,V_i') ]  )   
 \ge \delta(   G  [ N_{G}(v,V_i)]  ) - |B| -(r-1)\\
& \ge (1- 1/r) d_{G}(v , V_i) + \eps |V_i| - 3\eta^{1/2}  n\\
& \ge (1- 1/r) d_{H'_i}(v , V_i') + \gamma |V_i'|,
\end{align*}
so condition~(ii) of Lemma~\ref{lma:extendtoF} holds.
Condition~(iii) holds since $d_{H'_i}(y,B_i) \le d_{H'}(y,B) \le 3qn \leq 4qk|V_i'|$ for all $y \in V_i'$ by \eqref{eqn:partial2}.
To see that (iv) holds, notice that
\begin{align*}
\delta(H'_i[V'_i]) \ge ( 1 - 1/(r+1) + \eps) |V_i| - |B| \ge ( 1 - 1/r + 2 \gamma) |V_i'|.
\end{align*}
So by Lemma~\ref{lma:extendtoF}, there is a subgraph $H_i$ of $H_i'[V_i']$ such that $H'_i[ B_i,V_i' ] \cup H_i$ has an $F$-decomposition $\mathcal{F}'_i$ and $\Delta(H_i) \leq 2 \gamma |V_i|$.
Let $\mathcal{F}_2 : = \bigcup_{i=1}^k \mathcal{F}'_i$.

Let $H : = H'  \cup (G - G[\mathcal{P}]) - \bigcup_{i=1}^k (H'_i[B_i,V'_i] \cup H_i)    = G - \bigcup \mathcal{F}_1 -\bigcup \mathcal{F}_2$. 
Then (a) holds.
To see that (b) holds note that by \eqref{eqn:partial2}, for each $v \in A$, $d_{H[\P]}(v) \leq d_{H'}(v) \le 3q n \leq \gamma n$ and, for each $v \in B$, $d_{H[ \P]}(v)  = d_{H' - \bigcup \mathcal{F}_2}(v) \le |B| + k(r-1) \le 3\eta^{1/2}n \leq \gamma n $.
Finally, (c) holds since $(G - H) [V_i] = H_i$.
\end{proof}

\subsection{Covering a pseudorandom remainder} \label{sec:sparse-remainder-cover}
Lemma~\ref{lma:partial-decomposition} gives us an approximate $F$-decomposition such that the remainder~$H$ has the property that $H[\P]$ has low maximum degree.
We can also use an $F$-parity graph from Section~\ref{sec:parity-graphs} to ensure that, for each $2 \leq i \leq k$ and each $x \in V_{<i}$, $r$ divides $d_{H}(x,V_i)$.
We now cover all remaining edges of $H[\P]$ by using a small number of edges from $H - H[\P]$.
We are unable to apply Lemma~\ref{lma:extendtoF} directly, as the greedy algorithm used to prove Lemma~\ref{lma:findcliques} fails when $H$ is approximately regular and $U$ is much larger than $V$.
However, if $H$ is pseudorandom then we can recover an appropriate version of Lemma~\ref{lma:findcliques} by using a random greedy algorithm instead; this is because, when the codegrees of $H$ are small, an edge used in one copy of $K_r$ will only be contained in a small proportion of the other neighbourhoods that we consider.

Throughout this subsection $H$ should be thought of as a random graph of density~$\rho$.
In Section~\ref{sec:iteration} we will justify this assumption by combining the low degree remainder from Lemma~\ref{lma:partial-decomposition} with a random subgraph of $G$ of larger density.

\begin{lemma} \label{lma:sparsefindcliques}
Let $r,k,n \in \mathbb{N}$ and let $\rho >0$ with $1/n \ll 1/r,1/k,\rho\le 1$.
Let $H$ be a graph on $n$ vertices.
Suppose that $U_1, \dots, U_{p}$ are subsets of $V(H)$ with $p \le k n$ such that 
\begin{enumerate}[label=\rm(\roman*)]
	\item $r$ divides $|U_j|$ for all $1 \leq j \le p$; \label{eqn:Krm1}
	\item $\delta ( H [ U_j ] ) \ge (1- 1/r)|U_j| +  9 r k \rho^{3/2} n $ for all $1 \leq j \le p$;  \label{eqn:Krm2}
	\item $|U_j \cap U_{j'}| \le 2 \rho^2 n$ for distinct $1 \leq j,j' \le p$;\label{eqn:Krm3}
	\item each $v \in V(H)$ is contained in at most $ 2 k \rho n $ of the $U_j$. \label{eqn:Krm4}
\end{enumerate}
Then there exist edge-disjoint subgraphs $T_1, \dots, T_p$ in $H$ such that each $T_j$ is a $K_{r}$-factor in $H[U_j]$.
\end{lemma}

We will use the following simple result.

\begin{proposition}[Jain, see {\cite[Lemma 8]{super-chernoff}}] \label{generalised-chernoff}
Let $X_1, \ldots, X_n$ be Bernoulli random variables such that, for any $1 \leq s \leq n$ and any $x_1, \ldots, x_{s-1}\in \{0,1\}$,
 \begin{align*}
\Pr(X_s = 1 \mid X_1 = x_1, \ldots, X_{s-1} = x_{s-1}) \leq p.
\end{align*}
Let $X:= \sum_{i \in [n]} X_i$ and let $B \sim B(n,p)$.
Then $\Pr(X \geq a) \leq \Pr(B \geq a)$ for any~$a\ge 0$.
\end{proposition}

\begin{proof}[Proof of Lemma~\ref{lma:sparsefindcliques}]
Let $t: = \lceil 8 k \rho^{3/2} n \rceil $, and let $H_j := H [ U_j] $ for all $1 \leq j \le p$.
We construct ${T}_1, \dots, {T}_p$ in turn using a randomised algorithm. 
Suppose that we have already found $T_1, \dots, T_{s-1}$ for some $1 \le s \le p$; we will find ${T}_s$ as follows.

Let $G_{s-1} : = \bigcup_{i=1}^{s-1}  T_{i}$ be the subgraph of $H$ consisting of the edges that have already been used.\COMMENT{AL: defined $H_s'$ first, so that we can say $A_1, \dots, A_t$ are subgraphs of $H'_s$, i.e. each $A_i$ is edge-disjoint from $G_{s-1}$.}
Let $H_s' : = H_s - G_{s-1}[U_s]$.
If $ \Delta ( G_{s-1}[U_s] ) > r \rho^{3/2} n$, then let $A_1, \ldots, A_t$ be empty graphs on $U_s$.
If $\Delta ( G_{s-1}[U_s] ) \le r \rho^{3/2} n$, then $\delta(H'_s) \ge (1- 1/r )|H'_s| + 8 k r \rho^{3/2} n \ge (1- 1/r )|H'_s| + (r-1)(t-1)$ by~\ref{eqn:Krm2}.
So by \ref{eqn:Krm1} and Theorem~\ref{thm:HSz}, there exist $t$ edge-disjoint $K_{r}$-factors $A_1, \dots, A_{t } $ in~$H_s'$.

In either case, we have found edge-disjoint subgraphs $A_1, \dots, A_t$ of $H'_s$.
Pick $1 \leq i \leq t$ uniformly at random and set $T_s : = A_i$.
To prove the lemma, it suffices to show that, with positive probability,
\begin{align}
\Delta ( G_{s-1}[U_s] ) \le r  \rho^{3/2} n \text{ for all $1 \leq s \le p$.}	\label{eqn:Krmkey}
\end{align}

Consider $1 \leq j \le p$ and $u \in U_j$.
For $1 \leq s \le p$, let $Y^{j,u}_s$ be the indicator function of the event that $T_s$ contains an edge incident to $u$ in~$H_{j}$.
Let $X^{j,u} :  = \sum_{s=1}^p Y^{j,u}_s$.
Note that if $Y^{j,u}_s = 1$, then at most $r-1$ edges at $u$ in $H_j$ are used for $T_s$, so $d_{G_p}(u,U_j) \leq r X^{j,u}$.
Therefore to prove \eqref{eqn:Krmkey} it suffices to show that $X^{j,u} \le  \rho^{3/2} n$ for all $1 \leq j \le p$ and $u \in U_j$. 

Fix $1 \leq j \le p$ and $u \in U_j$.
Let $J^{j,u}$ be the set of indices $s \ne j$ such that $u \in U_{s}$.
By~\ref{eqn:Krm4}, $ |J^{j,u}|  \le 2 k  \rho n$. 
Note that $Y_s^{j,u} = 0 $ for all $s \notin J^{j,u} \cup \{j\}$.
So 
\begin{align}
X^{j,u} \le 1 + \sum_{s \in J^{j,u}} Y^{j,u}_{s}. \label{eqn:Xju}
\end{align}
Let $s_1, \dots, s_{|J^{j,u}|}$ be the enumeration of $J^{j,u} $ such that $s_b < s_{b+1}$ for all $1 \leq b \le |J^{j,u}|$.%
	\COMMENT{AL:need the enumeration to fit Proposition~\ref{generalised-chernoff}.}
For $b \le |J^{j,u}|$, note that  $d_{H_{s_b}} ( u, U_j )  \le |U_j \cap U_{s_b}| \le 2 \rho^2 n$ by~\ref{eqn:Krm3}.
So at most $2 \rho^2 n$ of the subgraphs $A_{i}$ that we picked in $H'_{s_b}$ contain an edge incident to~$u$ in~$H_j$. 
This implies that
\begin{align*}
	\mathbb{P} ( Y^{j,u}_{s_b} = 1 \mid Y^{j,u}_{s_1} = y_1, \dots, Y^{j,u}_{s_{b-1}} = y_{b-1})  \le 
	\frac{2 \rho^2 n}{ t } \le \frac {\rho^{1/2}} {4k}
\end{align*}
for all $y_1,\dots,y_{b-1}\in \{0,1\}$ and all $1 \le b \le |J^{j,u}| $.
Let $B \sim B( |J^{j,u}|  , \rho^{1/2}/4k )$.
By \eqref{eqn:Xju}, Proposition~\ref{generalised-chernoff}, Lemma~\ref{lma:chernoff} and the fact that $|J^{j,u}|\le 2k\rho n$ we have that
\begin{align*}
\mathbb{P}( X^{j,u} >  \rho^{3/2} n ) 
& \le
\mathbb{P}( \sum_{s \in J^{j,u}} Y^{j,u}_{s}  >  3 \rho^{3/2} n/4 ) 
\le \mathbb{P}( B >  3 \rho^{3/2} n/4  )
\\
& \le \mathbb{P}( | B - \mathbb{E}(B) | >  \rho^{3/2} n/4 )
\le 2 e^{-  \rho^{2} n/16k}.
\end{align*} 
Since there are at most $k n^2 $ pairs $(j, u)$, there is a choice of $T_1, \ldots, T_p$ such that $X^{j,u} \le  \rho^{3/2} n$ for all $1 \leq j \le p$ and all $u \in U_j$, provided $n$ is sufficiently large.\COMMENT{NEW and 20/7}
\end{proof}
We now use Lemma~\ref{lma:sparsefindcliques} to prove the corresponding version of Lemma~\ref{lma:findcliques}.

\begin{corollary} \label{cor:sparsefindcliques2}
Let $r,k,n \in \mathbb{N}$ and let $ \rho >0$ with $1/n \ll 1/r,1/k,\rho \le 1$.
Let $H$ be a graph on $n$ vertices.
Let $U,V \subseteq V(H)$ be disjoint with $|V| \ge \lfloor n/k \rfloor$. 
Suppose that, for all distinct $x,x' \in U$ and each $y \in V$,
\begin{enumerate}[label=\rm(\roman*)]
	\item $r$ divides $d_H(x,V)$; 
	\item $\delta(H[N_H(x,V)]) \geq ( 1 - 1/r ) d_H(x,V) + 9rk \rho^{3/2} |V|$; 
 	\item $|N_H(x,V) \cap N_H(x',V)| \le 2 \rho^2 |V|$;
	\item $d_H( y , U ) \leq 2 k \rho |V|$.
\end{enumerate}
Then there is a subgraph $H_V$ of $H[V]$ such that $H[ U,V ] \cup H_V$ has a $K_{r+1}$-decomposition and $\Delta(H_V) \leq 2 r k  \rho |V|$.
\end{corollary}

\begin{proof}
Let $p : = |U|$; note that $ p \le k |V|$. 
Let $u_1, \dots, u_p$ be an enumeration of~$U$.
Let $U_j : = N_H(u_j,V)$ for all $1 \leq j \le p$.
Apply Lemma~\ref{lma:sparsefindcliques} with $H[V], |V|$ playing the roles of~$H$, $n$ to obtain edge-disjoint subgraphs $T_1, \dots, T_p$ in $H[V]$ such that each $T_j$ is a $K_{r}$-factor in $H[U_j]$.
Let $H_V:=  \bigcup_{j=1}^p T_j$.
Note that $H[ U,V ] \cup H_V  = \bigcup_{j=1}^p ( H[ \{ u_j \}, U_j] \cup T_j )$ has a $K_{r+1}$-decomposition. 
Since $d_{H_V}(y) \le r d_H( y , U ) \leq  2 r k \rho |V|$ for each $y \in V$ by~(iv), we have $\Delta(H_V) \leq 2rk \rho |V|$.
\end{proof}

The following lemma follows from Corollary~\ref{cor:sparsefindcliques2} in the same way that Lemma~\ref{lma:extendtoF} follows from Lemma~\ref{lma:findcliques}, so we omit a detailed proof.\COMMENT{AL: Proof moved to end of paper}

\begin{lemma} \label{lma:sparseextendtoF}
Let $r,k,n,f \in \mathbb{N}$ and let $\alpha, \rho >0$ with $1/n \ll \rho \ll \alpha, 1/k, 1/r,1/f \le 1$.
Let $F$ be an $r$-regular graph on $f$ vertices and let $H$ be a graph on $n$ vertices.
Let $U,V \subseteq V(H)$ be disjoint with $|V| \ge \lfloor n/k \rfloor$. 
Suppose that, for all distinct $x,x' \in U$ and each $y \in V$,
\begin{enumerate}[label=\rm(\roman*)]
	\item $r$ divides $d_H(x,V)$; 
	\item $\delta(H[N_H(x,V)]) \geq ( 1 - 1/r ) d_H(x,V) + 9rk \rho^{3/2} |V|$; 
 	\item $|N_H(x,V) \cap N_H(x',V)| \le 2 \rho^2 |V|$;
	\item $d_H( y , U ) \leq 2 k \rho  |V|$;
	\item $\delta(H[V]) \ge (1-1/r + 2 \alpha) |V|$.
\end{enumerate}
Then there is a subgraph $H_V'$ of $H[V]$ such that $H[ U,V ] \cup H_V'$ has an $F$-decomposition and $\Delta(H_V') \leq 2 \alpha |V|$.\qedhere
\end{lemma}

Lemma~\ref{lma:sparseextendtoF} easily implies the following corollary.

\begin{corollary} \label{cor:sparseextendtoF}
Let $r,k,n,f \in \mathbb{N}$ and let $\alpha, \rho >0$ with $1/n \ll \rho \ll \alpha, 1/k, 1/r,1/f \le 1$.
Let $F$ be an $r$-regular graph on $f$ vertices and let $H$ be a graph on $n$ vertices.
Let $\P = \{V_1, \ldots, V_k\}$ be an equitable partition of $V(H)$.  
Suppose that, for each $2 \leq i \leq k$, all distinct $x, x' \in V_{<i}$ and each $y \in V_i$,
\begin{enumerate}[label=\rm(\roman*)]
	\item $r$ divides $d_H(x,V_i)$; 
	\item $\delta(H[N_H(x,V_i)]) \geq ( 1 - 1/r ) d_H(x,V_i) + 9r k\rho^{3/2} |V_i|$; 
	\item $|N_H(x,V_i) \cap N_H(x',V_i)| \le 2 \rho^2 |V_i|$;
	\item $d_H( y , V_{<i}) \leq  2k\rho  |V_i|$;
	\item $\delta(H[V_i]) \ge (1-1/r + 2 \alpha) |V_i|$.
\end{enumerate}
Then there is a subgraph $H_0$ of $H - H[ \P ] $ such that $H[ \P ] \cup H_0$ has an $F$-decomposition and $\Delta(H_0) \leq 2\alpha n$.
\end{corollary}

\begin{proof}
For each $2 \leq i \leq k$, let $U_i : = V_{<i}$, and let $H_i$ be the graph on $V(H)$ with $E(H_i) := E(H[U_i, V_i])\cup E(H[V_i])$.
Note that $H_2, \dots, H_k$ are pairwise edge-disjoint and $H[\P] \subseteq \bigcup_{i=2}^k H_i$.
We apply Lemma~\ref{lma:sparseextendtoF} to each $H_i$ with $U_i$,$V_i$ playing the roles of $U$,$V$ to obtain a subgraph $H_i'$ of $H_i[V_i]$ such that $H_i[ U_i,V_i ] \cup H_i'$ has an $F$-decomposition and $\Delta(H_i') \leq 2 \alpha |V_i|$.
Let $H_0 : = \bigcup_{i=2}^k H'_i$.
Note that $H[ \P ] \cup H_0 = \bigcup_{i=2}^k \left( H_i[ U_i,V_i ] \cup H_i' \right) $ has an $F$-decomposition and $\Delta(H_0) = \max_{2 \le i \le k} \Delta(H_i')   \le 2\alpha n$ since $V(H_i') \subseteq V_i$ for each~$i$.
\end{proof}

\subsection{Proof of Lemma~\ref{lma:do-the-iteration2}} \label{sec:iteration}

We now present the formal version of the statement~($\dagger$) at the beginning of Section~\ref{sec:partial-decomposition}.
Recall that if $\P$ is a $(k, \delta+\epsilon)$-partition for $G$ and $H$ is a subgraph of $G$ with $\Delta(H) \leq \epsilon n/2k$, then $\P$ is a $(k, \delta)$-partition for $G-H$.

\begin{lemma} \label{lma:iterate}
Let $r,f,k,n \in \mathbb{N}$ and let $\eta, \eps >0$ with $1/n \ll \eta \ll 1/k \ll \epsilon, 1/r, 1/f$.
Let $F$ be an $r$-regular graph on $f$ vertices.
Let $G$ be an $r$-divisible graph on $n$ vertices and let $G_0$ be a subgraph of $G - G[\P]$. 
Let $\delta: = \max \{ \delta^{\eta}_F, 1- 1/(r+1)\}$.
Suppose that $\P = \{V_1, \ldots, V_k\}$ is a $(k, \delta + 3 \epsilon)$-partition for~$G - G_0$.%
	\COMMENT{$\P$ is a $(k, \delta + 3 \epsilon)$-partition for~$G$.}
Then there is a subgraph $H$ of $G - G[\P] - G_0$ such that $G[\P] \cup H$ has an $F$-decomposition and $\Delta(H) \leq \epsilon n /2k^2$.%
\COMMENT{A:16/06 changed from $\epsilon n /2k^2$ to $\epsilon n /k^2$

B: 20/6 We really do need the 1/2, so I put it back.
The change to this proof is easy: we just have to half the value of $\alpha$ we use when applying Corollary 10.11.}
\end{lemma}

In our application of Lemma~\ref{lma:iterate} the graph $G_0$ will consist of edges which will be used in later iterations and are therefore not allowed to be used in the current one, so $H$ needs to avoid $G_0$.

The proof of Lemma~\ref{lma:iterate} uses Corollary~\ref{cor:sparseextendtoF}.
In order to guarantee that condition (ii) of Corollary~\ref{cor:sparseextendtoF} will hold, we first remove a sparse random graph $R$ from $G[\P]$.
We then add $R$ back to the remainder graph $H$ obtained from Lemma~\ref{lma:partial-decomposition} so that $H[\P]$ essentially behaves like a random subgraph of $G[\P]$.%

\begin{proof}[Proof of Lemma~$\ref{lma:iterate}$]
Choose $\gamma, \rho$ such that $1/n \ll \eta \ll \gamma \ll \rho  \ll 1/k\ll \eps,1/r, 1/f$.
Let $G_1 : = G - G_0$, and let $G_1' := G_1 - G[\P]$.
By Lemma~\ref{lma:random-slice}, there is a subgraph $R$ of $G_1[\P]$ such that, for each $1 \leq i \le k $ and all distinct $x,y \in V(G)$,
\begin{align}
d_R(x,V_i) & =  \rho d_{ G_1[\P]}( x , V_i) \pm \gamma |V_i|; \label{random-neighbourhoods-are-not-too-large}\\
d_R( \{ x, y\} ,V_i) & \le \rho^2 d_{ G_1[\P]}( \{ x,y\} , V_i) + \gamma |V_i| \le (\rho^2 + \gamma)|V_i| ; \label{random-codeg-are-not-too-large}\\
d_{G_1'}( y , N_R(x,V_i)) &\ge  \rho  d_{G_1'}(y, N_{G_1}(x, V_i)) - \gamma n . \label{random-neighbourhood-nice}
\end{align}
For each $2 \le i \le k$, each $x \in V_{<i}$ and each $y \in N_{G_1}(x,V_i)$, we have $d_{G_1'}( y , N_{G_1}(x,V_i)) = d_{G_1}( y , N_{G_1}(x,V_i))$, so \eqref{random-neighbourhood-nice} and Proposition~\ref{prop:neighbourhood-degree} imply that 
\begin{align}
d_{G_1'}( y , N_R(x,V_i)) & \ge 
 \rho \left(  (1-1/r) d_{G_1}(x, V_i)  + \eps |V_i|\right)   - \gamma n \nonumber \\
 & \ge 
 (1-1/r) \rho  d_{G_1}(x, V_i) + 10 rk \rho^{3/2}  |V_i|. 
\label{random-neighbourhoods-have-high-minimum-degree}
\end{align}

Let $G_2 := G_1 - R$.
Note that $\P$ is a $(k, \delta + 2 \epsilon)$-partition for $G_2$ since $\rho  \ll \epsilon$.
By Lemma~\ref{lma:find-parity-graph}, $G_2$ contains an $F$-parity graph $P$ with respect to $\P$ such that 
\begin{align}
\Delta(P) \leq \gamma n . \label{eqn:DeltaP}
\end{align}

Let $G_3 := G_2 - P$.
Note that $\P$ is a $(k, \delta + \eps)$-partition for $G_3$ as $\gamma \ll \epsilon$.
Apply Lemma~\ref{lma:partial-decomposition} to $G_3$ to obtain a subgraph $G_4$ of $G_3$ such that \begin{enumerate}[label={\rm(\alph*)}]
	\item $G_3-G_4$ has an $F$-decomposition $\mathcal{F}_1$;
	\item $\Delta( G_4 [ \P ] ) \leq \gamma n $;
	\item for each $1 \leq i \le k$, $\Delta ( \bigcup\mathcal{F}_1[V_i] ) = \Delta(G_3[V_i] - G_4[V_i]) \le 2 \gamma |V_i| $.
\end{enumerate}
Recall that $P$ is an $F$-parity graph, so has an $F$-decomposition. 
Note that $G^* : = R \cup G_4 \cup G_0 = G - P - \bigcup \mathcal{F}_1$ is obtained from $G$ by removing a set of edge-disjoint copies of $F$, so $G^*$ is $r$-divisible. 
Since $P$ is an $F$-parity graph with respect to $\P$, there is a subgraph $P'$ of $P$ such that $P-P'$ has an $F$-decomposition $\mathcal{F}_2$ and $r$ divides $d_{G^* \cup P'}(x, V_i)$ for each $2 \leq i \leq k$ and each $x \in V_{<i}$.
Note that, by~\eqref{eqn:DeltaP},
\begin{align}
	\Delta ( \textstyle\bigcup\mathcal{F}_2[V_i] ) \le \Delta(P) \leq \gamma n. \label{eqn:DK2}
\end{align}

Let $G_5 : = G^* \cup P'  -  G_0$.\COMMENT{We added back $G_0$ for parity reasons, but we're done with it now so it can go away again.}
Note that 
\begin{align}
G_5  = G_1 - \textstyle\bigcup \mathcal{F}_1 - \textstyle\bigcup \mathcal{F}_2 = R \cup G_4 \cup P'. \label{eqn:G5}
\end{align}
We will now check that conditions (i)--(v) of Corollary~\ref{cor:sparseextendtoF} hold with $G_5$ playing the role of $H$.
Recall that $e(G_0[\P]) = 0$ and that $r$ divides $d_{G^* \cup P'}(x, V_i)$ for each $2 \leq i \leq k$ and each $x \in V_{<i}$.
So condition (i) holds. 
Consider $2 \le i \le k$ and $x \in V_{<i}$.
By \eqref{eqn:G5}, \eqref{random-neighbourhoods-are-not-too-large}, (b) and \eqref{eqn:DeltaP}, we have that
\begin{align}
d_{G_5} (x, V_i) & \le d_R(x, V_i) + \Delta( G_4 [ \P ] ) + \Delta(P) 
 \le \rho d_{G_1[\P] }(x, V_i) + 3 \gamma  n. \label{eqn:G52}
\end{align}
Therefore, using \eqref{random-neighbourhoods-have-high-minimum-degree}, (c) and \eqref{eqn:DK2} in the second line, we have that for $y \in N_{G_5}(x,V_i) \subseteq N_{G_1}(x,V_i)$,
\begin{align*}
\arraycolsep=0pt\def\arraystretch{1.4}
\begin{array}{rcl}
d_{G_5} ( y , N_{G_5}(x, V_i) ) &
 \overset{\eqref{eqn:G5}}{\ge} &  d_{G_1'}(y, N_R(x,V_i)) - \Delta ( \textstyle\bigcup\mathcal{F}_1[V_i] ) - \Delta ( \textstyle\bigcup\mathcal{F}_2[V_i] )\\
 & \ge & (1-1/r) \rho  d_{G_1}(x, V_i) + 10 r k \rho^{3/2}  |V_i|   - 2  \gamma  n  - \gamma n \\
 & \overset{\eqref{eqn:G52}}{\ge} &  (1-1/r) (  d_{G_5}( x, V_i ) -  3 \gamma n )  + 10 r k \rho^{3/2}  |V_i|   - 3 \gamma n\\
 & \ge &  (1-1/r)  d_{G_5}( x, V_i )  + 9 rk \rho^{3/2}  |V_i|   .
\end{array}
\end{align*}
Thus condition (ii) of Corollary~\ref{cor:sparseextendtoF} holds.

To see that conditions (iii) and (iv) of Corollary~\ref{cor:sparseextendtoF} hold, note that, for all distinct $x,x' \in V_{<i}$ and each $2 \le i \le k$,
\begin{align*}
|N_{G_5}(x,V_i) \cap N_{G_5}(x',V_i)| 
\overset{\eqref{eqn:G5}}{\le} 
d_R( \{ x, x'\} ,V_i) + \Delta( G_4 [ \P ] ) + \Delta(P) 
\le 2 \rho^2 |V_i|,
\end{align*}
where the second inequality holds by \eqref{random-codeg-are-not-too-large}, (b) and \eqref{eqn:DeltaP}.
Similarly, for each $y \in V_{i}$ and each $1 \le i \le k$,
\begin{align*}
	d_{G_5} (y , V_{<i}) &  \overset{\eqref{eqn:G5}}{\le} \Delta(R) + \Delta( G_4 [ \P ] ) + \Delta(P) 
\overset{\eqref{random-neighbourhoods-are-not-too-large}, (\text{b}), \eqref{eqn:DeltaP}}{\le} (\rho + 3 \gamma) n 
\le 2 \rho k | V_i|.
\end{align*}
Set $\alpha := \epsilon  / 8k^2$.
Recall that $\P$ is a $(k, 1- 1/(r+1) + \epsilon)$-partition for~$G_1$.
Thus, for each $1 \leq i \le k$, 
\begin{align*}
\arraycolsep=0pt\def\arraystretch{1.4}
\begin{array}{rcl}
\delta( G_5[V_i] ) & \overset{\eqref{eqn:G5}}{\ge} &
 \delta( G_1[V_i] )- \Delta ( \textstyle\bigcup\mathcal{F}_1[V_i] ) - \Delta ( \textstyle\bigcup\mathcal{F}_2[V_i] )\\
& \overset{(\text{c}),\eqref{eqn:DK2}}{\ge}   & ( 1 - 1/(r+1) +  \eps ) |V_i|   - 3  \gamma  n
\ge ( 1 - 1/r +  2\alpha ) |V_i|,
\end{array}
\end{align*}
implying condition (v) of Corollary~\ref{cor:sparseextendtoF}.
So by Corollary~\ref{cor:sparseextendtoF}, there is a subgraph $H_5$ of $G_5 - G_5[\P]$ such that $G_5[\P] \cup H_5$ has an $F$-decomposition~$\mathcal{F}_3$ and $\Delta(H_5) \leq \eps n/4k^2$.
Let $H:= ( \bigcup \mathcal{F}_1- \bigcup \mathcal{F}_1 [\mathcal{P}]) \cup (\bigcup \mathcal{F}_2 - \bigcup \mathcal{F}_2[\mathcal{P}]) \cup H_5$.
Note that $H \subseteq G - G[\mathcal{P}] - G_0$ and $\Delta(H) \le 2\gamma \lceil n /k \rceil +\gamma n +  \eps n/4k^2 \le \eps n /2k^2$ by (c) and \eqref{eqn:DK2}.
In particular, $G[\P] \cup H = G_1[\P] \cup H$ also has an $F$-decomposition $\mathcal{F}_1 \cup \mathcal{F}_2 \cup \mathcal{F}_3$ by~\eqref{eqn:G5}.\COMMENT{A:16/06 changed last paragraph}
\end{proof}

As described at the beginning of this section, we can now iteratively apply Lemma~\ref{lma:iterate} to a sequence of partitions to prove the following lemma, which immediately implies Lemma~\ref{lma:do-the-iteration2}.

\begin{lemma} \label{lma:do-the-iteration}
Let $r,f,m,k,\ell \in \mathbb{N}$ and let $\eps, \eta > 0$ with $1/m \ll \eta \ll 1/k \ll \epsilon, 1/r,1/f$.\COMMENT{Here, $k^{\ell}(m-1)  < n \le k^{\ell}m$. AL:(again) It turns out that it is better to work with $m$ than $n$ when we proceed by induction on $\ell$.}
Let $F$ be an $r$-regular graph on $f$ vertices and let $G$ be an $r$-divisible graph.
Let $\delta: = \max \{ \delta^{\eta}_F, 1- 1/(r+1) \}$.
Suppose that $\P_1, \ldots, \P_\ell$ is a sequence of partitions of $V(G)$ such that
\begin{enumerate}[label=\rm(\roman*)]
	\item $\P_1$ is a $(k, \delta + \epsilon)$-partition for $G$;
	\item for each $2 \leq i \leq \ell$ and each $V \in \P_{i-1}$, $\P_i[V]$ is a $(k, \delta + 2 \epsilon)$-partition for $G[V]$;
	\item each $V \in \P_\ell$ has size $m-1$ or $m$.\COMMENT{B: The previous $\max \{ |V| : V \in \P_{\ell }\} = m$ does not play nicely with induction.}
\end{enumerate}
Then there exists a subgraph $H^*$ of $\bigcup_{V \in \P_\ell} G[V]$ such that $G-H^*$ has an $F$-decomposition.
\end{lemma}

\begin{proof}
Let $n: = |G|$, and let $G_0 := G - G[\P_1] - G[\P_2]$.
(If $\ell$=1, then let $G_0$ be the empty graph.)

Note that $\P_1$ is a $(k, \delta + \epsilon/2)$-partition for $G - G_0$.%
	\COMMENT{The degrees inside $V \in \mathcal{P}_1$ are fine, since $\Delta(G_0) \le |V|/k \ll \eps |V|$ as $1/k \ll \eps$.
We also use the fact that $\P_1$ is a $(k, \delta + 2 \epsilon)$-partition for $G$, i.e. $\delta(G[V]) \ge (\delta + 2 \epsilon)|V|$.}
Apply Lemma~\ref{lma:iterate} to obtain a subgraph $H$ of $G - G[\P_1] - G_0$ such that $G[\P_1] \cup H$ has an $F$-decomposition $\mathcal{F}_0$ and $\Delta(H) \leq \epsilon n /2k^2$\COMMENT{A:16/06 changed}.
This proves the case when $\ell =1$ (by setting $H^* := G- G[\P_1] - H$), so we proceed by induction and assume that $\ell \ge 2$.

Obtain $H$ as above and let $G' := G - G[\P_1] - H$.
Consider $U \in \P_1$.
Note that $G'[U]$ is $r$-divisible. 
Since $\Delta(H) \leq \epsilon n / 2k^2$\COMMENT{A:16/06 changed}, we have that $\P_2[U]$ is a $(k, \delta + \epsilon)$-partition for $G'[U]$.
Since $H$ is edge-disjoint from $G_0$, for each $3 \leq i \leq \ell$ and each $V \in \P_{i-1}[U]$ we have that $\P_i[V]$ is a $(k, \delta + 2 \epsilon)$-partition for $G'[V]$.
So we can apply the induction hypothesis to $G'[U],\P_2[U], \ldots, \P_\ell[U]$ to obtain a subgraph $H^*_U$ of $\bigcup_{V \in \P_\ell[U]} G[V]$ such that $G'[U]-H^*_U$ has an $F$-decomposition $\mathcal{F}_U$.

Set $H^*: = \bigcup_{U \in \P_1} H^*_U$.
Observe that $H^*$ is a subgraph of $\bigcup_{V \in \P_\ell} G[V] = G - G[\P_{\ell}]$ and $G - H^*$ has an $F$-decomposition $\mathcal{F}_0 \cup \bigcup_{U \in \P_1} \mathcal{F}_U$.
\end{proof}

\subsection{A strengthening of Lemma~\ref{lma:do-the-iteration2} for certain graphs~$F$} \label{remarks}

Suppose that $F$ is an $r$-regular graph that is not a vertex-disjoint union of copies of $K_{r+1}$.
Then Lemma~\ref{lma:do-the-iteration2} still holds if we replace $\delta: = \max \{ \delta^{\eta}_F, 1- 1/(r+1) \}$ by $\delta: = \max \{ \delta^{\eta}_F, 1- 1/r \}$.

\begin{lemma} \label{lma:do-the-iteration3}
Let $r,f,m,k,\ell \in \mathbb{N}$ and let $\eps, \eta > 0$ with $1/m \ll \eta \ll 1/k \ll \epsilon, 1/r,1/f$.
Let $F$ be an $r$-regular graph on $f$ vertices such that $F$ is not a vertex-disjoint union of copies of $K_{r+1}$.
Let $G$ be an $r$-divisible graph.
Let $\delta: = \max \{ \delta^{\eta}_F, 1- 1/r \}$.
Suppose that $\P_1, \ldots, \P_\ell$ is a $(k, \delta + \epsilon, m)$-partition sequence for $G$.
Then there exists a subgraph $H^*$ of $\bigcup_{V \in \P_\ell} G[V]$ such that $G-H^*$ has an $F$-decomposition.%
In particular, if $G$ is $F$-divisible, then so is $H^*$.
\end{lemma}

We now sketch a proof of Lemma~\ref{lma:do-the-iteration3} obtained by modifying the proof of Lemma~\ref{lma:do-the-iteration2}.
Note that the application of Theorem~\ref{thm:HSz} (in the proof of Lemma~\ref{lma:findcliques}) is the only point in
the proof of Lemma~\ref{lma:do-the-iteration2} where we need that $\delta\ge 1-1/(r+1)$ (rather than $\delta\ge 1-1/r$).
Since $F$ is not a vertex-disjoint union of copies of $K_{r+1}$, there exists a vertex $x$ in $F$ such that $F_x: =F[N_F(x)]$ is not complete.
Note that $\chi(F_x) \le r-1$ as $|F_x| = r$ and $F_x \ne K_{r}$.
Suppose that $H$, $x$ and $V$ are as described at the beginning of Section~\ref{sec:boundmaxdeg}.
Then it suffices to find an $F_x$-factor in $N_H(x,V)$ (rather than a $K_r$-factor).
So we can replace Theorem~\ref{thm:HSz} by the following result.

\begin{theorem}[Alon and Yuster~\cite{AlonYuster}] \label{thm:alon-yuster}
For every graph $F$ and every $\eps >0$, there exists an $n_0 = n_0(\eps,F)$ such that every graph $G$ on $n \ge n_0$ vertices with $\delta(G) \ge (1 - 1/ \chi(F) + \eps ) n $ contains $\lfloor |G| / |F| \rfloor$ vertex-disjoint copies of $F$.
\end{theorem}

The proof of Lemma~\ref{lma:do-the-iteration3} is otherwise the same as that of Lemma~\ref{lma:do-the-iteration2}.


\section{Proof of Theorem~\ref{thm:regular}} \label{sec:main-theorem}

We can now complete the proof of Theorem~\ref{thm:regular}.

\begin{proof}[Proof of Theorem~$\ref{thm:regular}$]
Without loss of generality we may assume that $\eps \ll 1/r,1/f$.
Choose $k,m', n_0 \in \mathbb{N}$ and $\eta > 0$ such that $1/n_0 \ll 1/m' \ll \eta \ll 1/k \ll  \eps \ll 1/r, 1/f $,%
\COMMENT{Here we need $1/k \ll \eps$ because we need to bound $\delta(G_1)$ from below.}
 and let $\eps' : = \eps/3$.
Let $G$ be an $F$-divisible graph on $n \ge n_0$ vertices with $\delta(G) \geq (\delta + 3 \eps') n$.
By Lemma~\ref{lma:partition-structure}, there is  a $(k, \delta + 2 \eps', m)$-partition sequence $\P_1, \ldots, \P_\ell$ for $G$ such that $m' \le m \le k m'$.
Let $G_1 := G[\P_1]$, and let $G_{\ell+1} := G-G[\P_\ell]$.
Note that $\delta(G_1[\mathcal{P}_{\ell}]) \geq \delta(G_1) \ge (\delta + 2 \eps') ( n - \lceil  n/k \rceil)  \ge (\delta + \eps') n$%
\COMMENT{For $v \in V \in \mathcal{P}_1$, $d_{G_1}(v) = \sum_{U \in \mathcal{P}_1 \setminus \{V\}} d_{G_1}(v,U) \ge  (\delta + 2 \eps') ( n - |V|) \ge (\delta + 2 \eps') ( n - \lceil  n/k \rceil) $.}
	as $1/k \ll \eps'\ll 1$.
Since $\delta \geq 1-1/3r$, we can apply Lemma~\ref{lma:Krabsorber} to $G_1 \cup G_{\ell+1}$ (with $\P_{\ell}, (\eps'/2k)^{1/2}$ playing the roles of $\P, \eps$) to obtain an $F$-divisible subgraph $A^*$ of $G_1 \cup G_{\ell+1}$ such that
\begin{enumerate}[label={\rm(\roman*)}]
\item $\Delta(A^*[\P_{\ell}]) \le \eps' n /2 k$ and  $\Delta( A^*[V] ) \le  r$ for each $V \in \P_{\ell}$, and 
\item if $H^*$ is an $F$-divisible graph on $V(G)$ that is edge-disjoint from $A^*$ and has $e ( H^* [ \P_{\ell} ] ) = 0 $, then $A^* \cup H^*$ has an $F$-decomposition.
\end{enumerate}

Let $G' := G - A^*$; then $G'$ is $F$-divisible. 
Note that for each $V \in \mathcal{P}_1$ and each $v \in V(G)$, we have $d_{G'}(v,V) \ge d_{G}(v,V) - \Delta(A) \ge (\delta + \eps') |V|$. 
So $\P_1$ is a $(k, \delta+\eps')$-partition for~$G'$. 
Note that $\Delta(A^* - A^*[ \mathcal{P}_{1} ]) \le r$ by~(i), so $\P_1, \ldots, \P_{\ell}$ is a $(k, \delta+\eps',m)$-partition sequence for $G'$.%
	\COMMENT{For more details (but long): For each $2 \le j \le \ell$, each $V \in \mathcal{P}_{i-1}$, each $U \in \mathcal{P}_i[V]$
and each $v \in V$, we have $d_{G'}(v,U) \ge d_{G}(v,U) - r \ge (\delta + \eps') |V|$ as $r \le \eps' m \le \eps' |V|$.
	Thus for each $2 \le i \le \ell$ and each $V \in \mathcal{P}_{i-1}$, $\mathcal{P}_{i}[V]$ is a $(k, \delta + \eps')$-partition for $G'[V]$.}

Apply Lemma~\ref{lma:do-the-iteration2} to obtain an $F$-divisible subgraph $H$ of $\bigcup_{V \in \P_\ell} G'[V]$ such that $G' - H$ has an $F$-decomposition.
But now $A^* \cup H$ has an $F$-decomposition by~(ii).
\end{proof}

Note that all of our arguments can be carried out in polynomial time, and all probabilistic arguments give the desired structure with sufficiently high probability.
Haxell and R\"odl's original proof of Theorem~\ref{thm:haxellrodl} gave a polynomial time algorithm for converting a fractional decomposition to an approximate decomposition, and Kierstead, Kostochka, Mydlarz and Szemer\'edi \cite{HajnalSzemerediAlgo} found an alternative proof of Theorem~\ref{thm:HSz} which gave a polynomial time algorithm for finding $K_r$-factors.%
\COMMENT{Algo para 1}

We can actually obtain a stronger version of Theorem~$\ref{thm:regular}$ which can be applied to obtain better bounds for certain graphs~$F$.
It involves the parameter $d_F$ introduced in Section~\ref{sec:absremarks} that measures the degeneracy of the most efficient transformer for~$F$.
The proof of Theorem~\ref{thm:regularstrong} is the same as that of Theorem~\ref{thm:regular} except that we replace Lemma~\ref{lma:Krabsorber} with Lemma~\ref{lma:Krabsorber2} and Lemma~\ref{lma:do-the-iteration2} with Lemma~\ref{lma:do-the-iteration3} (if $F$ is not a vertex-disjoint union of copies of $K_{r+1}$).

\begin{theorem} \label{thm:regularstrong}
Let $F$ be an $r$-regular graph.
Then for all $\epsilon > 0$, there exists an $n_0 = n_0(\epsilon,F)$ and an $\eta: = \eta(\eps, F) $ such that every $F$-divisible graph $G$ on $n\geq n_0$ vertices with $\delta(G) \geq (\delta+\epsilon)n$, where
\begin{align*}
\delta : = &
\begin{cases}
\max \{ \delta^{\eta}_F ,  1-\frac1{d_F}, 1 - \frac1{r+1} \}
& \text{ if $F$ is a vertex-disjoint union of copies of $K_{r+1}$,}\\ 
\max \{ \delta^{\eta}_F ,  1-\frac1{d_F}, 1 - \frac1{r} \} & \text{otherwise,} 
\end{cases}
\end{align*}
has an $F$-decomposition.
\end{theorem}

Our proof of Theorem~\ref{thm:regularstrong} can also be carried out in polynomial time, since the $F$-factor guaranteed in Theorem~\ref{thm:alon-yuster} can be obtained in polynomial time (see the discussion after \cite[Theorem 2.6]{Survey}).%
\COMMENT{Algo para 2}

Note that Theorem~\ref{thm:regularstrong} implies Theorem~\ref{thm:regular} since Corollary~\ref{cor:regular} states that $d_F \le 3r$ for any $r$-regular graph~$F$.
However for some graphs $F$ one can obtain much better bounds on $d_F$, yielding improved overall bounds.
We illustrate this for the case of cycles and complete bipartite graphs in Section~\ref{sec:cycles}.



\section{Decompositions into cycles and bipartite graphs} \label{sec:cycles}

In this section we consider $C_{\ell}$-decompositions and deduce Theorem~\ref{thm:cycles} from Theorem~\ref{thm:regularstrong}.
For even $\ell$, the bounds in Theorem~\ref{thm:cycles} are asymptotically best possible.
We now describe the construction giving the lower bound in the following two propositions.

\begin{proposition} \label{cycle-lower-bounds}
 Let $\ell \in \mathbb{N}$ with $\ell \geq 3$ and $\ell \ne 4$, and let
\begin{align*}
	\delta: =
  \begin{cases}
     1/2  & \quad \text{if } \ell \geq 6 \text{ is even}; \\
    \frac \ell {2(\ell-1)} & \quad \text{if } \ell \text{ is odd}.
  \end{cases}
 \end{align*}
Then there are infinitely many $C_\ell$-divisible graphs $G$ with $\delta(G) \geq \delta |G| -1$ that are not $C_\ell$-decomposable.	
\end{proposition}

Note that the case $ \ell =3$ describes an extremal example for the triangle decomposition conjecture of Nash-Williams.

\begin{proof}

  \emph{Case $\ell \geq 6$ even.}
     Let $n$ be such that $n \equiv \ell + 1  \mod{2\ell}$.
     So $n-1$ is even and $\ell$ divides $n(n-1)$ but not $\binom{n}2$.\COMMENT{
If $\ell$ divides $\binom{n}2$, then $2\ell$ divides $n(n-1)$.
On the other hand, note that $ n(n-1) \equiv (\ell+1)\ell \equiv \ell \mod{2 \ell}$ as $\ell$ is even.}	 Let $G$ be the vertex-disjoint union of two cliques of order~$n$.
	 Then $G$ is $C_{\ell}$-divisible with $\delta (G) = |G|/2 - 1$, but neither connected component is itself $C_\ell$-divisible, so $G$ cannot be $C_\ell$-decomposable.
  
  \emph{Case $\ell$ odd.}
  Let $G$ be the graph obtained from $K_{\ell-1,\ell-1}$ by blowing up each vertex to a clique of odd order $n$ such that $\ell$ divides~$n$.
  The degree of each vertex is $\ell n-1$, which is even, and the total number of edges is $(\ell-1)n(\ell n-1)$, which is divisible by~$\ell$.
  We have that $\delta(G) = \ell n - 1 = \frac{\ell |G|}{2(\ell - 1)}   -1$, but each copy of $C_\ell$ in $G$ contains an edge of one of the copies of $K_n$, and the number of such edges is only $(\ell-1)n(n-1) < e(G)/\ell$, so $G$ cannot be $C_\ell$-decomposable.
\end{proof}

The existence of special constructions for the case $\ell=4$ is perhaps surprising, and was first observed by Kahn and Winkler (see~\cite{YusterBip}), who showed that there exist infinitely
many $C_4$-divisible graphs $G$ with $\delta(G) \geq \lfloor 3|G|/5 \rfloor - 1$ that are not $C_4$-decomposable.  A construction matching the asymptotic upper bound from
Theorem~\ref{thm:cycles} was found by Taylor~\cite{Taylor}. This construction also generalizes to $K_{r,r}$ for even~$r$.

\begin{proposition} \label{c4-lower-bounds}
Let $r \in \mathbb N$ be even.
Then there are infinitely many $K_{r,r}$-divisible graphs $G$ with $\delta(G) \geq \lfloor 2 |G| /3 \rfloor  - r$ that are not $K_{r,r}$-decomposable.	
\end{proposition}

\begin{proof}
We will construct one such graph for each $m \in \mathbb{N}$.
Let $V_1,V_2, V_3$ be disjoint vertex sets such that $|V_1| = 2r^2 m +r$, $|V_2| = 2r^2 m +1$, $|V_3| = 2r^2 m -r$.
Let $G$ be the graph on vertex set $V_1 \cup V_2 \cup  V_3$ consisting only of a clique on~$V_1$, a clique on~$V_3$ and a complete bipartite graph with vertex classes $V_1 \cup V_3 , V_2$. 
Thus $n: = |G| = 6 r^2 m +1$.
Note that $d(x) \in \{ 4 r^2 m +r, 4 r^2 m, 4 r ^2 m -r\}$ for all vertices~$x$.
The total number of edges is %
\COMMENT{$(2r^2m+r)(2r^2m+r-1)/2+(2r^2m-r)(2r^2m-r-1)/2+(4r^2 m)(2r^2 m +1)$}
\begin{align*}
	\binom{2r^2 m +r}{2} + \binom{2r^2 m -r}{2} + (4r^2 m)(2r^2 m +1) = (12 r^2 m^2 +2m +1) r^2. 
\end{align*}
Thus $G$ is $K_{r,r}$-divisible and $\delta(G) = \lfloor 2n/3 \rfloor -r$.

Now let $F$ be any copy of $K_{r,r}$ in $G$.
We claim that $e(F[V_1])$ is divisible by $r$.
Let $\{A,B\}$ be the natural bipartition of $F$.
If $e(F[V_1]) = 0$ then we are done, so assume that $A \cap V_1$ and $B \cap V_1$ are both non-empty.
Then $A \cap V_3$ and $B \cap V_3$ must both be empty, as there are no edges from $V_1$ to $V_3$ in $G$.
Since $V_2$ is an independent set in $G$, $A \cap V_2$ and $B \cap V_2$ cannot both be non-empty, so by relabelling if necessary we may assume that $A \cap V_2$ is empty; that is, $A \subseteq V_1$.
Hence $e(F[V_1]) = r|B \cap V_1|$, which is divisible by $r$.
But $e(G[V_1]) = (r^2 m + r/2)(2 r^2 m+r-1)$, which is not divisible by $r$ as the second factor is coprime to $r$ and the first factor is not divisible by $r$.
So $G[V_1]$ cannot be covered by any set of edge-disjoint copies of $F$ in $G$.
\end{proof}

We prove Theorem~$\ref{thm:cycles}$ using Theorem~\ref{thm:regularstrong}.
Recall the definition of $d_{C_{\ell}}$ in Section~\ref{sec:absremarks}.
We now bound $d_{C_{\ell}}$ above for $\ell \ge 3$.

\begin{lemma} \label{lma:cycle}
For $\ell \in \mathbb{N}$ with $\ell \ge 3$,
\begin{align*}
	d_{C_\ell} & \le 
	\begin{cases}
		 4 &\text{if $\ell = 3$,}\\
		 3 & \text{if $\ell =  4$,}\\
		 2 & \text{if $\ell \ge  5$.}
	\end{cases}
\end{align*}
\end{lemma}

\begin{proof}
Lemma~\ref{lma:cycletransformer} implies that the lemma holds for $\ell \ge 4$.
So we may assume that $\ell = 3$.
Let $H$ be a $2$-regular graph and let $H'$ be obtained from $H$ by identifying vertices.
Suppose that $H$ and $H'$ are vertex-disjoint. 
Recall that an $(H,H')_{C_3}$-transformer~$T$ is a graph such that
\begin{itemize}
	\item $T \cup H$ and $T \cup H'$ each have $C_3$-decompositions;
	\item $V(H \cup H') \subseteq V(T)$ and $T[V(H \cup H')]$ is empty.
\end{itemize}
To show that $d_{C_3} \le 4$, it suffices to show that there exists an $(H,H')_{C_3}$-transformer~$T$ such that the degeneracy of $T$ rooted at $V(H \cup H')$ is at most~$4$.

Let $\phi: H \to H'$ be a graph homomorphism from $H$ to $H'$ that is edge-bijective.
Note that $H$ is a union of vertex-disjoint cycles $C_{s_1}, \dots, C_{s_p}$. 
So $H'$ decomposes into $\phi(C_{s_1}), \dots, \phi(C_{s_p})$.
Suppose that, for each $1 \le j \le p$, there exists a $(C_{s_j}, \phi( C_{s_j} ) )_{C_3}$-transformer~$T_j$ such that the degeneracy of $T_j$ rooted at $V( C_{s_j} \cup \phi( C_{s_j} ))$ is at most~$4$.
We further choose the $T_j$ such that $V(T_j) \cap V(H \cup H') = V( C_{s_j} \cup \phi( C_{s_j} ))$ and $V(T_j) \cap V(T_{j'}) \subseteq V(H \cup H')$ for all $ j \ne j'$.
In particular, the $T_j$ are edge-disjoint.
Let $T: = \bigcup_{1 \le j \le p} T_j$.
Then $T$ is an $(H,H')_{C_3}$-transformer such that the degeneracy of $T$ rooted at $V(H \cup H')$ is at most~$4$.\COMMENT{To see the degeneracy, consider the vertices in $H$, $H'$, $T_1 \setminus (H \cup H')$, \dots, $T_p \setminus (H \cup H')$, where the vertices in $T_j \setminus (H \cup H')$ are ordered in the `obvious' way.}
Therefore, we may assume that $H$ is a cycle $x_1x_2 \dots x_s x_1$.

Let $ \{ u_i, v_i, w_i : 1 \le i \le s \}$ be a set of $3s$ vertices disjoint from $V(H \cup H')$.
Define a graph $T$ as follows:
\begin{enumerate}[label=(\roman*)]
\item $V(T) := V(H) \cup V(H')  \cup \{ u_i, v_i, w_i : 1 \le i \le s \}$;
\item $E_1 := \{ x_i u_i, x_i v_i , x_i  w_i , x_i u_{i+1} : 1 \le i \le s \}$;
\item $E_2 := \{ v_i w_i : 1 \le i \le s \}$;
\item $E_3 := \{ u_i v_i, w_iu_{i+1} : 1 \le i \le s \}$;
\item $E_4 := \{ \phi(x_i) u_i, \phi(x_i) v_i , \phi(x_i)  w_i , \phi(x_i) u_{i+1} : 1 \le i \le s \}$;
\item $E(T) := E_1 \cup E_2 \cup E_3 \cup E_4$.
\end{enumerate}
Here the indices are considered modulo~$s$.
Note that $T [ V(H \cup H')]$ is empty.

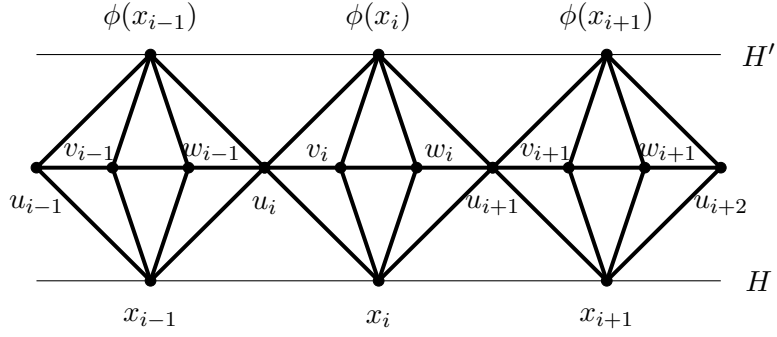
\begin{figure}[t]
\begin{tikzpicture}

	\begin{scope}
		\foreach \x in {0,1,...,2}
			{
			\filldraw[fill=black] (3*\x,0) circle (2pt);
			\filldraw[fill=black] (3*\x,3) circle (2pt);
			\foreach \y in {0,1,2,3}{			
			\filldraw[fill=black] (3*\x-1.5+1*\y,1.5) circle (2pt);
			\draw[line width=1.5pt]  (3*\x,0) to (3*\x-1.5+1*\y,1.5) to (3*\x,3);}		
			}
 	\end{scope}

\draw[line width=1.5pt] (-1.5,1.5) to (7.5,1.5);
\draw (-1.5,0) to (7.5,0);
\draw (-1.5,3) to (7.5,3);

\node  at (0,-0.5)  {$ x_{i-1} $};
\node  at (3,-0.5)  {$ x_{i} $};
\node  at (6,-0.5)  {$ x_{i+1} $};
\node  at (8,0)  {$ H $};

\node  at (0,3.5)  {$ \phi(x_{i-1}) $};
\node  at (3,3.5)  {$ \phi(x_{i}) $};
\node  at (6,3.5)  {$ \phi(x_{i+1}) $};
\node  at (8,3)  {$ H' $};

\node  at (-1.5,1)  {$ u_{i-1} $};
\node  at (1.5,1)  {$ u_{i} $};
\node  at (4.5,1)  {$ u_{i+1} $};
\node  at (7.5,1)  {$ u_{i+2} $};

\node  at (-0.8,1.7)  {$ v_{i-1} $};
\node  at (2.2,1.7)  {$ v_{i} $};
\node  at (5.2,1.7)  {$ v_{i+1} $};

\node  at (0.8,1.7)  {$ w_{i-1} $};
\node  at (3.8,1.7)  {$ w_{i} $};
\node  at (6.8,1.7)  {$ w_{i+1} $};

\end{tikzpicture}
\caption{A $(H , H')_{C_3}$- transformer $T$.}
\end{figure}

Note also that $H \cup E_1 \cup E_2$ can be decomposed into $2s$ copies of $C_3$, where each $C_3$ has vertex set either $\{x_i,x_{i+1}, u_{i+1}\}$ or $\{x_i, v_i, w_i\}$ for some $1 \le i \le s$.
Note also that $E_3 \cup E_4$ can be decomposed into $2s$ copies of $C_3$, where each $C_3$ has vertex set either $\{\phi(x_i),u_i,v_i\}$ or $\{\phi(x_i), w_i, u_{i+1}\}$ for some $1 \le i \le s$.
Thus $H \cup T$ has a $C_3$-decomposition. 
Similarly, $H' \cup T$ has a $C_3$-decomposition. 
Therefore $T$ is an $(H,H')_{C_3}$-transformer.
To see that the degeneracy of $T$ rooted at $V(H) \cup V(H')$ is at most $4$, consider the vertices in $H$, $H'$, $\{u_i : 1 \le i \le s\}$,  $\{v_i, w_i : 1 \le i \le s\}$ in that order.
This completes the proof of the lemma.
\end{proof}

We now prove Theorem~\ref{thm:cycles}.

\begin{proof}[Proof of Theorem~\ref{thm:cycles}]
We first prove (i). So let $\ell\ge 4$ be even.
By Theorem~\ref{thm:regularstrong} and Lemma~\ref{lma:cycle}, it suffices to show that $\lim_{\eta \to 0} \delta^\eta_{C_\ell} \leq \delta$.
But $\delta^{\eta}_{C_\ell} = 0$ for all $\eta > 0$ since $C_\ell$ is bipartite.
Indeed, it follows from the Erd\H{o}s--Simonovits--Stone theorem~\cite{ErdosSimonovits, ErdosStone} that we can obtain an $\eta$-approximate
$C_\ell$-decomposition greedily (since the Tur\'an density of bipartite graphs is 0).

To prove (ii), let $\ell\ge 3$ be odd. Note that $\lim_{\eta \to 0} \delta^\eta_{C_\ell}\ge 1/2$ (consider e.g.~$K_{n,n}$).
Moreover, $\lim_{\eta \to 0} \delta^\eta_{C_3}\ge 3/4$, see e.g.~Yuster~\cite{YusterFractKr}. Thus the first part
of Theorem~\ref{thm:cycles}(ii) follows from Theorem~\ref{thm:regularstrong} and Lemma~\ref{lma:cycle}.
The `moreover part' follows then from Corollary~\ref{cor:chromatic} and Theorem~\ref{thm:Dross}.
\end{proof}

If $F$ is an $r$-regular bipartite graph, then Theorem~\ref{thm:regularstrong} implies the following result,
which applies for instance to the complete bipartite graph $K_{r,r}$.%

\begin{corollary} \label{thm:bipartite}
Let $F$ be an $r$-regular bipartite graph.
Then for each $\epsilon > 0$, there is an $n_0 = n_0(\epsilon, F)$ such that every $F$-divisible graph $G$ on $n \geq n_0$ vertices with $\delta(G) \geq (1-1/(r+1) + \epsilon) n$ has an $F$-decomposition.
\end{corollary}

\begin{proof}
As observed in the proof of Theorem~\ref{thm:cycles}(i), $\delta^{\eta}_F = 0$ for all $\eta > 0$ since $F$ is bipartite.
Since $d_F \le r+1$ by Lemma~\ref{lma:cycletransformer}, the result now follows from Theorem~\ref{thm:regularstrong}.
\end{proof}

\section*{Acknowledgements}

We would like to thank Amelia Taylor for the extremal example for $C_4$-decompositions in Proposition~\ref{c4-lower-bounds}, and the referee, Stefan Glock, John Lapinskas and Amelia Taylor for their helpful comments.


\begin{thebibliography}{za10}

\bibitem{AlonYuster}
N.~Alon and R.~Yuster.
\newblock {$H$}-factors in dense graphs.
\newblock {\em J. Combin. Theory Ser. B}, 66(2):269--282, 1996.

\bibitem{BKLMO}
B.~Barber, D.~K\"uhn, A.~Lo, R.~Montgomery and D.~Osthus.
\newblock Fractional clique decompositions of dense graphs and hypergraphs.
\newblock {\em arXiv preprint arXiv:1507.04985}, 2015.


\bibitem{C4decomp}
D.~Bryant and N.~Cavenagh.
\newblock Decomposing graphs of high minimum degree into $4$-cycles.
\newblock {\em J. Graph Theory}, 79:167--177, 2015.

\bibitem{NPcomplete}
D.~Dor and M.~Tarsi.
\newblock Graph decomposition is {NP}-complete: a complete proof of Holyer's conjecture.
\newblock {\em SIAM J. Comput.}, 26:1166--1187, 1997.

\bibitem{Dross}
F.~Dross.
\newblock Fractional triangle decompositions in graphs with large minimum degree.
\newblock {\em arXiv preprint arXiv:1503.08191}, 2015.


\bibitem{Dukes}
P.~Dukes.
\newblock Rational decomposition of dense hypergraphs and some related
  eigenvalue estimates.
\newblock {\em Linear Algebra Appl.}, 436(9):3736--3746, 2012.

\bibitem{Dukes2}
P.~Dukes.
\newblock Corrigendum to ``Rational decomposition of dense hypergraphs and some related eigenvalue estimates'' [Linear Algebra Appl. 436 (9) (2012) 3736--3746].
\newblock {\em Linear Algebra Appl.}, 467:267--269, 2015.


\bibitem{Erdos}
P.~Erd{\H{o}}s.
\newblock On some new inequalities concerning extremal properties of graphs.
\newblock In {\em Theory of {G}raphs ({P}roc. {C}olloq., {T}ihany, 1966)},
  pages 77--81. Academic Press, New York, 1968.

\bibitem{ErdosSimonovits}
P.~Erd{\H{o}}s and M.~Simonovits.
\newblock A limit theorem in graph theory.
\newblock {\em Studia Sci. Math. Hungar}, 1:51--57, 1966.

\bibitem{ErdosStone}
P.~Erd{\H{o}}s and A.~H. Stone.
\newblock On the structure of linear graphs.
\newblock {\em Bull. Amer. Math. Soc.}, 52:1087--1091, 1946.


\bibitem{Garaschuk}
K.~Garaschuk.
\newblock {\em Linear methods for rational triangle decompositions}.
\newblock PhD thesis, Univ. of Victoria, https://dspace.library.uvic.ca//handle/1828/5665, 2014.

\bibitem{Gustavsson}
T.~Gustavsson.
\newblock {\em Decompositions of large graphs and digraphs with high minimum
  degree}.
\newblock PhD thesis, Univ. of Stockholm, 1991.

\bibitem{HajnalSzemeredi}
A.~Hajnal and E.~Szemer{\'e}di.
\newblock Proof of a conjecture of {P}. {E}rd{\H o}s.
\newblock In {\em Combinatorial theory and its applications, {II} ({P}roc.
  {C}olloq., {B}alatonf\"ured, 1969)}, pages 601--623. North-Holland,
  Amsterdam, 1970.

\bibitem{HaxellRodl}
P.E. Haxell and V.~R{\"o}dl.
\newblock Integer and fractional packings in dense graphs.
\newblock {\em Combinatorica}, 21(1):13--38, 2001.

\bibitem{JLR}
S.~Janson, T.~{\L}uczak, and A.~Ruci\'{n}ski.
\newblock {\em Random graphs}.
\newblock Wiley-Interscience Series in Discrete Mathematics and Optimization.
  Wiley-Interscience, New York, 2000.

\bibitem{Keevash}
P.~Keevash.
\newblock The existence of designs.
\newblock {\em arXiv preprint arXiv:1401.3665}, 2014.

\bibitem{HajnalSzemerediAlgo}
H.A.~Kierstead, A.V.~Kostochka, M.~Mydlarz, and E.~Szemer{\'e}di.
\newblock A fast algorithm for equitable coloring.
\newblock {\em Combinatorica}, 30(2):217--224, 2010.

\bibitem{Kirkman}
T.P.~Kirkman.
\newblock On a problem in combinations.
\newblock {\em Cambridge Dublin Mathematical Journal}, 2:191--204, 1847.

\bibitem{Krivelevich}
M.~Krivelevich.
\newblock Triangle factors in random graphs.
\newblock {\em Combin. Probab. Comput.}, 6(3):337--347, 1997.

\bibitem{KellyConj}
D.~K{\"u}hn and D.~Osthus.
\newblock Hamilton decompositions of regular expanders: a proof of {K}elly's
  conjecture for large tournaments.
\newblock {\em Adv. Math.}, 237:62--146, 2013.


\bibitem{NashWilliams}
C.St.J.A.~Nash-Williams.
\newblock An unsolved problem concerning decomposition of graphs into
  triangles.
\newblock In {\em Combinatorial Theory and its Applications III}, pages
  1179--183. North Holland, 1970.

\bibitem{super-chernoff}
R.~Raman.
\newblock The power of collision: randomized parallel algorithms for chaining
  and integer sorting.
\newblock In {\em Foundations of software technology and theoretical computer
  science ({B}angalore, 1990)}, volume 472 of {\em Lecture Notes in Comput.
  Sci.}, pages 161--175. Springer, Berlin, 1990.

\bibitem{RRSz}
V.~R{\"o}dl, A.~Ruci{\'n}ski, and E.~Szemer{\'e}di.
\newblock A {D}irac-type theorem for 3-uniform hypergraphs.
\newblock {\em Combin. Probab. Comput.}, 15(1-2):229--251, 2006.

\bibitem{Schlund}
M.~Schlund.
\newblock {\em Graph Decompositions, Latin Squares, and Games}.
\newblock Diploma thesis, Technische Universit\"at M\"unchen, 2011.



\bibitem{Simonovits}
M.~Simonovits.
\newblock A method for solving extremal problems in graph theory, stability
  problems.
\newblock In {\em Theory of {G}raphs ({P}roc. {C}olloq., {T}ihany, 1966)},
  pages 279--319. Academic Press, New York, 1968.

\bibitem{Taylor}
A.~Taylor.
\newblock Personal communication.

\bibitem{wilson1}
R.M.~Wilson.
\newblock An existence theory for pairwise balanced designs. {I}. {C}omposition
  theorems and morphisms.
\newblock {\em J. Combin. Theory Ser. A}, 13:220--245, 1972.

\bibitem{wilson2}
R.M.~Wilson.
\newblock An existence theory for pairwise balanced designs. {II}. {T}he
  structure of {PBD}-closed sets and the existence conjectures.
\newblock {\em J. Combin. Theory Ser. A}, 13:246--273, 1972.

\bibitem{wilson3}
R.M.~Wilson.
\newblock An existence theory for pairwise balanced designs. {III}. {P}roof of
  the existence conjectures.
\newblock {\em J. Combin. Theory Ser. A}, 18:71--79, 1975.

\bibitem{Wilson}
R.M.~Wilson.
\newblock Decompositions of complete graphs into subgraphs isomorphic to a
  given graph.
\newblock In {\em Proceedings of the {F}ifth {B}ritish {C}ombinatorial
  {C}onference ({U}niv. {A}berdeen, {A}berdeen, 1975)}, pages 647--659.
  Congressus Numerantium, No. XV, Utilitas Math., Winnipeg, Man., 1976.

\bibitem{YusterBip}
R.~Yuster.
\newblock The decomposition threshold for bipartite graphs with minimum degree
  one.
\newblock {\em Random Structures Algorithms}, 21(2):121--134, 2002.

\bibitem{YusterFractKr}
R.~Yuster.
\newblock Asymptotically optimal {$K_k$}-packings of dense graphs via
  fractional {$K_k$}-decompositions.
\newblock {\em J. Combin. Theory Ser. B}, 95(1):1--11, 2005.

\bibitem{YusterFracDecom}
R.~Yuster.
\newblock Integer and fractional packing of families of graphs.
\newblock {\em Random Structures Algorithms}, 26(1-2):110--118, 2005.

\bibitem{Survey}
R.~Yuster.
\newblock Combinatorial and computational aspects of graph packing and graph
  decomposition.
\newblock {\em Computer Science Review}, 1(1):12--26, 2007.

\bibitem{YusterChromatics}
R.~Yuster.
\newblock {$H$}-packing of {$k$}-chromatic graphs.
\newblock {\em Mosc. J. Comb. Number Theory}, 2(1):73--88, 2012.

\bibitem{Yusternew}
R.~Yuster.
\newblock Edge-disjoint cliques in graphs with high minimum degree.
\newblock {\em SIAM J. Combin.}, 28(2):893--910, 2014.

\end{thebibliography}


\end{document}